\documentclass[10pt]{article}
   \usepackage{graphicx}
   \usepackage{epsfig}
   \usepackage{amssymb}
   \usepackage{amsmath}
   \usepackage{amsthm}
   \newtheorem{lemma}{Lemma}[section]
   \newtheorem{theorem}{Theorem}[section]
   \newtheorem{proposition}{Proposition}[section]

   \newtheorem{remark}{Remark}[section]

   {\end{description}}

   {\hfill $\bullet$ \\}

   \newcommand{\be}{\begin{equation}}
   \newcommand{\ee}{\end{equation}}

\begin{document}
    \title{A MacCormack Method for Complete Shallow Water Equations with Source Terms}
   \author{Eric Ngondiep$^{\text{\,a\,b}},$ Alqahtani Rubayyi$^{\text{\,b\,}},$ Jean C. Ntonga$^{\text{\,a\,}}$}
   \date{$^{\text{\,a\,}}$\small{Department of Mathematics and Statistics, College of Science, Al-Imam Muhammad Ibn Saud\\ Islamic University
        (IMSIU), $90950$ Riyadh $11632,$ Kingdom of Saudi Arabia.}\\
     \text{\,}\\
       $^{\text{\,b\,}}$\small{Hydrological Research Centre, Institute for Geological and Mining Research, 4110 Yaounde-Cameroon.}\\
     \text{,}\\
        \textbf{Email addresses:} ericngondiep@gmail.com, engondiep@imamu.edu.sa, rtalqahtani@imamu.edu.sa,
        ntong$a_{-}$jc@yahoo.fr}
   \maketitle

   \textbf{Abstract.}
    In the last decades, more or less complex physically-based hydrological models, have been developed to solve the shallow water equations or their approximations using various numerical methods. The MacCormack method was developed for simulating overland flow with spatially variable infiltration and microtopography using the hydrodynamic flow equations. The basic MacCormack scheme is enhanced when it uses the method of fractional steps to treat the friction slope or a stiff source term and to upwind the convection term in order to control the numerical oscillations and stability. In this paper we describe, the MacCormack scheme for $1$D complete shallow water equations with source terms, analyze the stability condition of the method and we provide the convergence rate of the algorithm. This work improves some well known results deeply studied in the literature which concern the Saint-Venant problem and it represents an extension of the time dependent shallow water equations without source terms. The numerical evidences consider the rate of convergence of the method and compares the numerical solution respect to the analytical one.\\
    \text{\,}\\

   \\
   {\bf AMS Subject Classification (MSC). 65M10, 65M05}.

  \section{Introduction}\label{In}
   Most open-channel flows of interest in the physical, hydrological, biological, engineering and social sciences   are unsteady and can be considered to be one-dimensional ($1$D). In this paper we are interested in the numerical solutions of one-dimensional complete shallow water equations with source terms introduced in \cite{8nnnn} and still widely used in modelling flows in rivers, lakes and coastal areas as well as atmospheric and oceanic flows in certain regimes. In the case of a prismatic channel, the complete shallow water equations with source terms reads as follows
        \begin{equation}\label{1}
        \left\{
          \begin{array}{ll}
            \frac{\partial A}{\partial t}+\frac{\partial Q}{\partial x}=r, & \hbox{for $t\in(0,T_{1})\text{\,\,and\,\,}x\in\Omega=(0,L)$},\\
            \text{\,}\\
            \frac{\partial Q}{\partial t}+g\frac{A}{T}\frac{\partial A}{\partial x}+\frac{\partial}{\partial x}
            \left(\frac{Q^{2}}{A}\right)=gA(S_{0}-S_{f}),&
            \hbox{for $t\in(0,T_{1})\text{\,\,and\,\,}x\in\Omega=(0,L)$},
          \end{array}
        \right.
      \end{equation}
      where the bottom (or bed) slope $(S_{0})$ and the friction slope $(S_{f})$ (see \cite{8nnnn}) are defined as
      \begin{equation}\label{9}
        S_{0}=\frac{\overline{\tau} P}{\rho gA}\text{\,\,\,\,and\,\,\,\,}S_{f}=\frac{Q|Q|}{K^{2}},
      \end{equation}
       $r=r(x,t)$ is the lateral inflow per unit length along the channel, $T_{1}$ is the time interval length, $L$ is the rod interval length, $A=A(t,x)$ is the cross section, $Q=Q(t,x)$ is the discharge, $g$ is the acceleration of gravity. $T=T(x,t)$ is the top width assumed to be constant, $\overline{\tau}$ is the average shear stress on the water from the channel boundary, $\rho$ is the fluid density,
       $P=P(x,y(t,x))$ is the wetted perimeter, i.e., the the length of the boundary of the cross section that is under water for a given height of water $(y).$ As in \cite{8nnnn}, the conveyance for a compact channel is defined by
      \begin{equation}\label{10}
        K:=K(x,y)=\frac{1.49}{n_{1}}A(x,y)R(x,y)^{2/3},
      \end{equation}
       where $R=A/P$ is the hydraulic radius and $n_{1}$ is the manning's roughness coefficient.\\

        By straightforward computations, it is not hard to see that system $(\ref{1})$ becomes
      \begin{equation}\label{2}
        \frac{\partial W}{\partial t}+\frac{\partial F}{\partial x}=S,
      \end{equation}
      where $W=col(A,Q),$ $F=col(Q,\frac{gA^{2}}{2T}+\frac{Q^{2}}{A}),$ and $S=col(r,gA(S_{0}-S_{f})).$ Equation $(\ref{2})$ emphasizes the conservative character of system $(\ref{1})$.\\

    The one-dimensional shallow water equations with source terms $(\ref{2})$ are highly nonlinear and therefore do not have global analytical solutions \cite{8nnnn}. When solving the system of balance laws $(\ref{2})$ numerically, one typically faces several difficulties. One difficulty stems from the fact that many physically relevant solutions of $(\ref{2})$ are small perturbations of steady-state solutions.
   So, using a wrong balance between the flux and geometric source term in equation $(\ref{2}),$ the solution may develop spurious waves of a magnitude that can become larger than the exact solution. Another drawback is when the cross section is very small. In that case, even small numerical oscillation in the computed solution can result in a very large discharge, which is not only physically irrelevant, but cause the numerical scheme to break down. To overcome these numerical challenges, one needs to use a numerical scheme that is both well-balanced and positivity preserving.\\

   A number of well-balanced and positivity preserving numerical methods frequently used in the models based on the shallow water equations have been proposed in literature \cite{clg,lcg}, or on the boussinesq equations, which are reduced to shallow water equations, in order to simulate breaking waves \cite{glc,zzl,kp}. Although the MacCormack scheme is less accurate than the more recent methods, it is commonly used for engineering problems due to its greater simplicity. So, we have to approximate the exact solution of problem $(\ref{2})$ by a numerical method based on the MacCormack scheme. This algorithm is a class of higher order finite difference methods (second order convergent), which provides an effective way of joining spectral method for accuracy and robust characteristics of finite difference schemes. For example, to compute unsteady flow specifically in the presence of discontinuity, inherent dissipation and stability, one such widely used method is MacCormack method \cite{21db}. This technique has been used successfully to provide time-accurate solution for fluid flow and aeroacoustics problems. The applications of this technique to $1$D shock tube and $2$D acoustic scattering problems provide good result while comparing with the analytical solution. MacCormack introduced a simpler variation of Lax-Wendroff scheme which is basically a two-step scheme with second order Taylor series expansion in time and fourth order in spatial accuracy \cite{21db,22db}. This algorithm is computationally efficient and easy to implement which can be appropriate to obtain reliable results. By using
   this scheme with two nodes, the flow field can be simulated for unsteady flows especially for shallow water problem in the presence of discontinuity and strict gradient conditions. Furthermore, to capture fluid flow in transition over long periods of time and distance, numerical spatial derivative are required to be determined in few grid points while error controlled can be accurately computed. The authors \cite{en1,en2,en3,en4,23db} extended MacCormack scheme \cite{21db} to an implicit-explicit scheme by coupling the original MacCormack approach and the Crank-Nicolson method, and to implicit compact differencing scheme by splitting the derivative operator of a central compact scheme into one-sided forward and backward operators. The one-sided nature of the MacCormack technique is an essential advantage especially when severe gradients are present.\\

   In \cite{13nnnn,17nnnn,14nnnn} the authors compared the Lax-Wendroff scheme to many numerical methods of high order of accuracy, such as, the linear Central Weighed Essential Non-Oscillatory (CWENO) scheme which is superior to full nonlinear CWENO method, to high-resolution TVD conservative schemes along with high order Central Schemes for hyperbolic systems of conservative Laws \cite{4nnnn,6nnnn} and to Central-upwind schemes for the shallow water system \cite{15nnnn}. In a search for stable and more accurate shock capturing numerical approach, they observed that the Lax-Wendroff approach is one of the most frequently encountered in the
   literature related to classical Shock-capturing schemes. However, difficulties have been reported when trying to include source terms in the discretization and to keep the second order accuracy at the same time \cite{21nnnn}. The MacCormack approach which is a predictor-corrector version of the Lax-Wendroff algorithm provides a reasonably good result at discontinuities. This method is much easy to apply than the Lax-Wendroff scheme because the Jacobian does not appear. The amplification factor and stability constraint almost are the same as presented for the Lax-Wendroff method (see \cite{3nnnn}, P. $202$-$206$, case of inviscid burgers equation). It is also important to note that the solutions obtained for the same problem at the same courant number are different from those obtained using the Lax-Wendroff scheme. This is due both to the switched differencing in the predictor and the corrector and the nonlinear nature of the governing PDE. It should be noted that reversing the differencing in the predictor and corrector steps leads to quite different results. Another motivation of this work comes from the fact that, the explicit MacCormack time discretization for the complete nonlinear Burgers equation (which can be served as a model equation for many nonlinear PDEs: Navier-Stokes equations, Stokes-Darcy equations, Parabolized Navier-Stokes equations,...) gives a suitable stability restriction which can be used with an appropriate safety factor (see \cite{3nnnn},
   P. $227$-$228$). For the MacCormack solver, as with other explicit schemes, it requires a time step limitation. In general, the maximum time step (with respect to stability) allowable in the MacCormack scheme applied to linear hyperbolic equations is limited by the CFL condition, as are all explicit finite difference methods. The overland flow equations are nonlinear, however, and a rigorous stability analysis for these equations is exceedingly difficult. The source terms place additional and problem-dependent restrictions on the maximum admissible time step for stability. Therefore, the CFL condition can only be considered as a general guideline here, and the maximum allowable time step for any particular problem will be less than predicted by the CFL condition and determined by numerical experimentation (see \cite{fr}, page $223$).\\

   In a recent work \cite{nnnn}, we addressed the problem of mathematical model of complete shallow water equations with source terms in which we provided the stability analysis of the Lax-Wendroff scheme. In this paper we are still interested by the stability analysis, but in the sense of MacCormack. In particular, we consider the case where the channel is prismatic and the interesting result is that the algorithm is at least of second order accuracy in time and space, while the stability limitation does not coincide with the CFL condition widely studied in the literature for hyperbolic partial differential equations (for example: linear
   advection equation, wave equation, inviscid burgers equations, etc,...). However, while the stability requirement is highly unusual, the result has a potential positive implication since the stability constraint obtained in this work controls the CFL condition. Indeed the nice feature is that, as required in a stability context, we normally find a linear stability condition which can be considered as a necessary condition of stability from a Fourier stability analysis. On the other hand, it comes from this analysis that an instability occurs when $|\Delta t|$ is greater than some $|\Delta t|_{\max}$ which can be viewed as $(\Delta t)_{CFL}.$ More specifically, the attention is focused in the following three items:

   \begin{description}
     \item[(i1)] full description of the MacCormack method for $1$D complete shallow water equations with source terms;
     \item[(i2)] stability restriction of the algorithm: this item together with item (i1) represent our original contributions and they improve the works studied in \cite{nnnn,6nnnn,20nnnn};
     \item[(i3)] Numerical experiments concerning the convergence of the method, the simulation of the numerical solution obtained by the MacCormack approach along with the analytical solution, and regarding the effectiveness of this method according to the theoretical analysis given in the first two items.
   \end{description}

   The paper is organized as follows. Section $\ref{sms}$ deals with the full description of MacCormack method for $1$D complete shallow water equations with source terms. The stability analysis of this scheme is deeply studied in section $\ref{sas}$. In Section $\ref{ne},$ some numerical experiments which consider the convergence of the scheme and some simulations are presented and discussed. We draw the general conclusion and present the future direction of works in section $\ref{cfw}$.

   \section{Full description of MacCormack method}\label{sms}
   In this section, we give a detailed description of the MacCormack algorighm for system $(\ref{2}).$ First, we recall that the MacCormack scheme is a two step explicit method which consists in predictor-corrector steps. The scheme uses the forward difference in predictor step while the corrector step considers the backward difference. Since our aim is to analyze both stability and rate of convergence of the method, without loss of generality we should use a constant time step $\Delta t$ and mesh size $\Delta x$. Let $N$ and $M$ be two
   positive integers. Noticing $x_{j}=j\Delta x,$ $t^{n}=n\Delta t$ and let the superscript (resp., the subscript) denoting the time level (resp., space level) of the approximation. We denote by $W_{j}^{n}=(A_{j}^{n},Q_{j}^{n})^{T}$ the approximate solution of equations $(\ref{2}),$ obtained at time $t^{n}$ and at point $x_{j},$ using the MacCormack algorithm and $W(t^{n},x_{j})=(A(t^{n},x_{j}),Q(t^{n},x_{j}))^{T}$
   the value of the analytical solution of system $(\ref{2})$ at discrete time $t^{n}$ and at discrete point $x_{j}.$ Furthermore, the domain $\Omega =(0,L)$ is subdivided into $M+1$ grid points $\{x_{j}:\text{\,}j=0,1,...,M\},$ while the time interval $(0,T)$ is subdivided into $N+1$ grid points $\{t^{n}:\text{\,}n=0,1,...,N\}.$ Using this, the full description of MacCormack method for the system $(\ref{2})$ reads: Given $W^{n}_{j}$, find an approximate solution (the solution can be considered as weak) $w^{n+1}_{j},$ for $0\leq n\leq N-1$ and $0\leq j\leq M,$ satisfying\\
      \text{\,}\\
      Predictor step: solve equation $(\ref{3})$ for predicted value
      \begin{equation}\label{3}
        W_{j}^{\overline{n+1}}=W_{j}^{n}-\frac{\Delta t}{\Delta x}(F^{n}_{j+1}-F_{j}^{n})+\Delta tS_{j}^{n},
      \end{equation}
      Corrector step: use the predicted value obtained in equation $(\ref{3})$ to compute the exact one
      \begin{equation}\label{4}
        W_{j}^{n+1}=\frac{1}{2}\left[W_{j}^{n}+W_{j}^{\overline{n+1}}-\frac{\Delta t}{\Delta x}(F^{\overline{n+1}}_{j}
        -F_{j}^{\overline{n+1}})+\Delta tS_{j}^{\overline{n+1}}\right].
      \end{equation}

      \begin{proposition}\label{p1}
      Let $n$ and $j$ be two nonnegative integers. Setting $t^{n}=n\Delta t$ and $x_{j}=j\Delta x,$ where $\Delta t$ and $\Delta x$
      are time step and mesh size, respectively. The MacCormack scheme for system $(\ref{2})$ is given by\\

       Predictor step: solve equation $(\ref{64})$ for predicted values $A_{j}^{\overline{n+1}}$ and $Q_{j}^{\overline{n+1}}$
      \begin{equation}\label{64}
        A_{j}^{\overline{n+1}}=A_{j}^{n}-\frac{\Delta t}{\Delta x}(Q^{n}_{j+1}-Q_{j}^{n})+\Delta t r_{j}^{n},
      \end{equation}
      and
      \begin{equation}\label{65}
        Q_{j}^{\overline{n+1}}=Q_{j}^{n}-\frac{\Delta t}{\Delta x}\left\{\frac{g}{2T}((A^{n}_{j+1})^{2}-(A_{j}^{n})^{2})
        +\frac{(Q^{n}_{j+1})^{2}}{A^{n}_{j+1}}-\frac{(Q^{n}_{j})^{2}}{A^{n}_{j}}\right\}+
        gP\Delta t\left(\frac{\overline{\tau}}{\rho g}-\frac{n_{1}^{2}}{1.49^{2}}P^{\frac{1}{3}}\frac{Q_{j}^{n}|Q^{n}_{j}|}
        {(A_{j}^{n})^{\frac{7}{3}}}\right).
      \end{equation}
      Corrector step: use $A_{j}^{\overline{n+1}}$ and $Q_{j}^{\overline{n+1}}$ obtained in equations $(\ref{64})$-$(\ref{65})$
      to compute $A_{j}^{n+1}$ and $Q_{j}^{n+1}$
      \begin{equation}\label{66}
        A_{j}^{n+1}=\frac{1}{2}\left\{A_{j}^{n}+A_{j}^{\overline{n+1}}-\frac{\Delta t}{\Delta x}(Q^{\overline{n+1}}_{j}-
        Q_{j-1}^{\overline{n+1}})+\Delta t r_{j}^{\overline{n+1}}\right\},
      \end{equation}
         and
      \begin{equation*}
        Q_{j}^{n+1}=\frac{1}{2}\left\{Q_{j}^{n}+Q_{j}^{\overline{n+1}}-\frac{\Delta t}{\Delta x}\left\{\frac{g}{2T}
        \left[(A^{\overline{n+1}}_{j})^{2}-(A_{j-1}^{\overline{n+1}})^{2}\right]+\frac{(Q^{\overline{n+1}}_{j})^{2}}
        {A^{\overline{n+1}}_{j}}-\frac{(Q^{\overline{n+1}}_{j-1})^{2}}{A^{\overline{n+1}}_{j-1}}\right\}\right.+
      \end{equation*}
      \begin{equation}\label{67}
        \left.gP\Delta t\left(\frac{\overline{\tau}}{\rho g}-\frac{n_{1}^{2}}{1.49^{2}}P^{\frac{1}{3}}\frac{Q_{j}^{\overline{n+1}}
        |Q^{\overline{n+1}}_{j}|}{(A_{j}^{\overline{n+1}})^{\frac{7}{3}}}\right)\right\}.
      \end{equation}
      \end{proposition}

      \begin{proof}
       Since $W=col(A,Q)$, using relations $(\ref{3})$-$(\ref{4})$ we get\\
      Predictor step: solve equation $(\ref{5})$-$(\ref{6})$ for predicted values $A_{j}^{\overline{n+1}}$ and $Q_{j}^{\overline{n+1}}$
      \begin{equation}\label{5}
        A_{j}^{\overline{n+1}}=A_{j}^{n}-\frac{\Delta t}{\Delta x}(Q^{n}_{j+1}-Q_{j}^{n})+\Delta t r_{j}^{n},
      \end{equation}
       \begin{equation}\label{6}
        Q_{j}^{\overline{n+1}}=Q_{j}^{n}-\frac{\Delta t}{\Delta x}\left\{\frac{g}{2T}((A^{n}_{j+1})^{2}-(A_{j}^{n})^{2})
        +\frac{(Q^{n}_{j+1})^{2}}{A^{n}_{j+1}}-\frac{(Q^{n}_{j})^{2}}{A^{n}_{j}}\right\}+g\Delta t A_{j}^{n}
        \left((S_{0})_{j}^{n}-(S_{f})_{j}^{n}\right),
      \end{equation}
      Corrector step: use $A_{j}^{\overline{n+1}}$ and $Q_{j}^{\overline{n+1}}$ obtained in equations $(\ref{5})$-$(\ref{6})$
      to compute $A_{j}^{n+1}$ and $Q_{j}^{n+1}$
      \begin{equation}\label{7}
        A_{j}^{n+1}=\frac{1}{2}\left\{A_{j}^{n}+A_{j}^{\overline{n+1}}-\frac{\Delta t}{\Delta x}(Q^{\overline{n+1}}_{j}-
        Q_{j-1}^{\overline{n+1}})+\Delta t r_{j}^{\overline{n+1}}\right\},
      \end{equation}
       \begin{equation*}
        Q_{j}^{n+1}=\frac{1}{2}\left\{Q_{j}^{n}+Q_{j}^{\overline{n+1}}-\frac{\Delta t}{\Delta x}\left\{\frac{g}{2T}
        \left[(A^{\overline{n+1}}_{j})^{2}-(A_{j-1}^{\overline{n+1}})^{2}\right]+\frac{(Q^{\overline{n+1}}_{j})^{2}}
        {A^{\overline{n+1}}_{j}}-\frac{(Q^{\overline{n+1}}_{j-1})^{2}}{A^{\overline{n+1}}_{j-1}}\right\}\right.+
      \end{equation*}
      \begin{equation}\label{8}
        \left.g\Delta t A_{j}^{\overline{n+1}}\left((S_{0})_{j}^{\overline{n+1}}-(S_{f})_{j}^{\overline{n+1}}\right)\right\}.
      \end{equation}

      Substituting equation $(\ref{10})$ into relation $(\ref{9})$ yields
      \begin{equation}\label{11}
        S_{f}=\frac{n_{1}^{2}}{1.49^{2}}\frac{Q|Q|}{A^{2}R^{4/3}}=\frac{n_{1}^{2}}{1.49^{2}}\frac{Q|Q|}{A^{10/3}}P^{4/3}.
      \end{equation}
      In way similar, substituting $(\ref{9})$ and $(\ref{11})$ into equations $(\ref{6})$ and $(\ref{8})$ results in
       \begin{equation}\label{12}
        Q_{j}^{\overline{n+1}}=Q_{j}^{n}-\frac{\Delta t}{\Delta x}\left\{\frac{g}{2T}((A^{n}_{j+1})^{2}-(A_{j}^{n})^{2})
        +\frac{(Q^{n}_{j+1})^{2}}{A^{n}_{j+1}}-\frac{(Q^{n}_{j})^{2}}{A^{n}_{j}}\right\}+
        gP\Delta t\left(\frac{\overline{\tau}}{\rho g}-\frac{n_{1}^{2}}{1.49^{2}}P^{\frac{1}{3}}
        \frac{Q_{j}^{n}|Q^{n}_{j}|}{(A_{j}^{n})^{\frac{7}{3}}}\right),
      \end{equation}
         and
        \begin{equation*}
        Q_{j}^{n+1}=\frac{1}{2}\left\{Q_{j}^{n}+Q_{j}^{\overline{n+1}}-\frac{\Delta t}{\Delta x}\left\{\frac{g}{2T}
        \left[(A^{\overline{n+1}}_{j})^{2}-(A_{j-1}^{\overline{n+1}})^{2}\right]+\frac{(Q^{\overline{n+1}}_{j})^{2}}
        {A^{\overline{n+1}}_{j}}-\frac{(Q^{\overline{n+1}}_{j-1})^{2}}{A^{\overline{n+1}}_{j-1}}\right\}\right.+
      \end{equation*}
      \begin{equation}\label{13}
        \left.gP\Delta t\left(\frac{\overline{\tau}}{\rho g}-\frac{n_{1}^{2}}{1.49^{2}}
        P^{\frac{1}{3}}\frac{Q_{j}^{\overline{n+1}}|Q^{\overline{n+1}}_{j}|}{(A_{j}^{\overline{n+1}})^{\frac{7}{3}}}\right)\right\}.
      \end{equation}
      An assembling of relations $(\ref{5}),$ $(\ref{7}),$ $(\ref{12})$ and $(\ref{13})$ provides the full description
      of MacCormack scheme which is given by\\
      \text{\,}\\
      Predictor step:
      \begin{equation*}
        A_{j}^{\overline{n+1}}=A_{j}^{n}-\frac{\Delta t}{\Delta x}(Q^{n}_{j+1}-Q_{j}^{n})+\Delta t r_{j}^{n},
      \end{equation*}
      and
      \begin{equation*}
        Q_{j}^{\overline{n+1}}=Q_{j}^{n}-\frac{\Delta t}{\Delta x}\left\{\frac{g}{2T}((A^{n}_{j+1})^{2}-(A_{j}^{n})^{2})
        +\frac{(Q^{n}_{j+1})^{2}}{A^{n}_{j+1}}-\frac{(Q^{n}_{j})^{2}}{A^{n}_{j}}\right\}+
        gP\Delta t\left(\frac{\overline{\tau}}{\rho g}-\frac{n_{1}^{2}}{1.49^{2}}P^{\frac{1}{3}}\frac{Q_{j}^{n}|Q^{n}_{j}|}
        {(A_{j}^{n})^{\frac{7}{3}}}\right).
      \end{equation*}
      Corrector step: use $A_{j}^{\overline{n+1}}$ and $Q_{j}^{\overline{n+1}}$ obtained above to compute $A_{j}^{n+1}$ and $Q_{j}^{n+1}$
      \begin{equation*}
        A_{j}^{n+1}=\frac{1}{2}\left\{A_{j}^{n}+A_{j}^{\overline{n+1}}-\frac{\Delta t}{\Delta x}(Q^{\overline{n+1}}_{j}-
        Q_{j-1}^{\overline{n+1}})+\Delta t r_{j}^{\overline{n+1}}\right\},
      \end{equation*}
         and
      \begin{equation*}
        Q_{j}^{n+1}=\frac{1}{2}\left\{Q_{j}^{n}+Q_{j}^{\overline{n+1}}-\frac{\Delta t}{\Delta x}\left\{\frac{g}{2T}
        \left[(A^{\overline{n+1}}_{j})^{2}-(A_{j-1}^{\overline{n+1}})^{2}\right]+\frac{(Q^{\overline{n+1}}_{j})^{2}}
        {A^{\overline{n+1}}_{j}}-\frac{(Q^{\overline{n+1}}_{j-1})^{2}}{A^{\overline{n+1}}_{j-1}}\right\}\right.+
      \end{equation*}
      \begin{equation*}
        \left.gP\Delta t\left(\frac{\overline{\tau}}{\rho g}-\frac{n_{1}^{2}}{1.49^{2}}P^{\frac{1}{3}}\frac{Q_{j}^{\overline{n+1}}
        |Q^{\overline{n+1}}_{j}|}{(A_{j}^{\overline{n+1}})^{\frac{7}{3}}}\right)\right\}.
      \end{equation*}
      \end{proof}
       Here, the terms $A^{\overline{n+1}}$ and $Q^{\overline{n+1}}$ are "predicted" values of $A$ and $Q,$ respectively, at the time level
      $n+1.$ Assuming further that the superscript $\overline{n+1}$ is a time level, it is easy to see that MacCormack algorithm
      is a tree level method, so the initial data $A^{0}$ and $Q^{0}$ are needed to begin the algorithm. However, appropriate initial and
      boundary conditions must be specified. Further, the presence of cross section in the denominator of several terms disallows zero cross
      sections,
      therefore, a finite minimum cross section is assigned to each node that is ponded. It is primarily the $A^{10/3}$ in the denominator
      of the friction slope term given by relation $(\ref{11})$ that limits the magnitude of the minimum cross section and discharge. When
      the cross sections are very small, the friction slope is very large compared with the other terms in second equation of system $(\ref{1})$.
      As cross sections increase rapidly during the early stages of flow development, the friction slope term magnitude changes much faster
      than the other terms. This phenomenon renders the second equation in system $(\ref{1})$ stiff and severely limits the maximum admissible
      time step for stability. Indeed, this phenomenon likely forced previous researchers to use very small time steps relative to their mesh
      size (courant number $\ll1$) and keep lateral inflows and initial cross sections large \cite{fr}.\\

   \section{Stability analysis of MacCormack scheme}\label{sas}
   This section deals with the stability analysis of the MacCormack numerical scheme for $1$D complete shallow water equations
   with source terms in the case where the channel is prismatic. First, we present a rainfall hydrograph test, based on experimental
   measurements realized thanks to the SATREPS project METHOD in a flume at the rain simulation facility at Benou\'{e}-Garoua
   (Cameroon). The flume is $1150m$ long with a slope of $4\%$. The simulation duration is $1$s. The rainfall intensity $I(x,t)$
   is described by
   \begin{equation}\label{hg}
    I(x,t)=\left\{
             \begin{array}{ll}
               1.18\times10^{-5}m/s & \hbox{if $(t,x)\in[0;\text{\,}1]\times[0;\text{\,}1]$;} \\
               0 & \hbox{otherwise.}
             \end{array}
           \right.
   \end{equation}
   For this test, as there is no rain on the last $150$m, we have a wet/dry transition. The measured output is an hydrograph,
   that is a plot of the discharge versus time. The mathematical model for this ideal overland flow is the following: we
   consider a uniform plane catchment whose overall length in the direction of flow is $L$. The surface roughness and slope are
   assumed to be invariant in space and time. We consider a constant rainfall excess such that
   \begin{equation}\label{hg1}
    r(x,t)=\left\{
             \begin{array}{ll}
               I & \hbox{if $t_{0}\leq t\leq T_{1}$,\text{\,} $0\leq x\leq L$;} \\
               0 & \hbox{otherwise,}
             \end{array}
           \right.
    \end{equation}
     where $I$ is the rainfall intensity and $T_{1}$ is the final time of the rainfall excess. According to relations $(\ref{hg})$ and
    $(\ref{hg1})$ we assume in the following that $r$ is more less that $A$ and $Q,$ i.e., $r\ll A,Q$. Furthermore, Proposition $\ref{p2}$
    gives the "temporary" stability constraint of the MacCormack algorithm described in section $\ref{sms}.$

    \begin{proposition}\label{p2}
   The numerical scheme $(\ref{66})$ is stable if estimate $(\ref{11e})$ holds.
   \begin{equation}\label{11e}
   \frac{\Delta t^{3}}{\Delta x^{2}}\left(1+\frac{2\Delta t}{3}\Gamma_{0}\mu^{n}|A^{n}|^{-\frac{4}{3}}\right)\leq
   3\Gamma_{0}^{-1}(\mu^{n})^{-3}|A^{n}|^{\frac{4}{3}}|\phi|^{-2},
   \end{equation}
   where $\mu=\frac{|Q|}{|A|},$ $\Gamma_{0}=\frac{gn_{1}^{2}}{1.49^{2}}P^{\frac{4}{3}}$ and $\phi=k\Delta x,$ where $k\neq0$ is the wave number.
   \end{proposition}

   \begin{proof}
   First of all, we give an explicit form of $A^{n+1}.$ Combining relations $(\ref{64})$ and $(\ref{65}),$ simple computations provide
   \begin{equation*}
    Q_{j}^{\overline{n+1}}-Q_{j-1}^{\overline{n+1}}=Q_{j}^{n}-Q_{j-1}^{n}-\frac{\Delta t}{\Delta x}\left\{\frac{g}
    {2T}\left[(A^{n}_{j+1})^{2}-2(A_{j}^{n})^{2}+(A^{n}_{j-1})^{2}\right]+\frac{(Q^{n}_{j+1})^{2}}{A^{n}_{j+1}}
    -2\frac{(Q^{n}_{j})^{2}}{A^{n}_{j}}+\frac{(Q^{n}_{j-1})^{2}}{A^{n}_{j-1}}\right\}
   \end{equation*}
   \begin{equation}\label{1e}
    -\Delta t\frac{gn_{1}^{2}}{1.49^{2}}P^{\frac{4}{3}}\left(\frac{Q_{j}^{n}|Q^{n}_{j}|}{(A_{j}^{n})^{\frac{7}{3}}}-
        \frac{Q_{j-1}^{n}|Q^{n}_{j-1}|}{(A_{j-1}^{n})^{\frac{7}{3}}}\right),
   \end{equation}
   and
   \begin{equation}\label{2e}
    A_{j}^{n}+A_{j}^{\overline{n+1}}=2A_{j}^{n}-\frac{\Delta t}{\Delta x}\left(Q_{j+1}^{n}-Q_{j}^{n}\right)+\Delta t r_{j}^{n}.
   \end{equation}
   Substituting equations $(\ref{1e})$ and $(\ref{2e})$ into $(\ref{66})$ results in
   \begin{equation*}
    A_{j}^{n+1}=A_{j}^{n}-\frac{\Delta t}{2\Delta x}\left(Q_{j+1}^{n}-Q_{j-1}^{n}\right)+\frac{1}{2}\left(\frac{\Delta t}{\Delta x}\right)^{2}
    \left\{\frac{g}{2T}\left[(A^{n}_{j+1})^{2}-2(A_{j}^{n})^{2}+(A^{n}_{j-1})^{2}\right]+\frac{(Q^{n}_{j+1})^{2}}{A^{n}_{j+1}}\right.
   \end{equation*}
   \begin{equation}\label{3e}
    \left.-2\frac{(Q^{n}_{j})^{2}}{A^{n}_{j}}+\frac{(Q^{n}_{j-1})^{2}}{A^{n}_{j-1}}\right\}+\frac{\Delta t^{2}}{2\Delta x}
    \frac{gn_{1}^{2}}{1.49^{2}}P^{\frac{4}{3}}\left(\frac{Q_{j}^{n}|Q^{n}_{j}|}{(A_{j}^{n})^{\frac{7}{3}}}-\frac{Q_{j-1}^{n}|Q^{n}_{j-1}|}
        {(A_{j-1}^{n})^{\frac{7}{3}}}\right)+\frac{1}{2}\Delta t\left(r_{j}^{n}+r_{j}^{\overline{n+1}}\right).
   \end{equation}
   Neglecting the last term in $(\ref{3e}),$ we obtain
   \begin{equation*}
    A_{j}^{n+1}=A_{j}^{n}-\frac{\Delta t}{2\Delta x}\left(Q_{j+1}^{n}-Q_{j-1}^{n}\right)+\frac{1}{2}\left(\frac{\Delta t}{\Delta x}\right)^{2}
    \left\{\frac{g}{2T}\left[(A^{n}_{j+1})^{2}-2(A_{j}^{n})^{2}+(A^{n}_{j-1})^{2}\right]+\frac{(Q^{n}_{j+1})^{2}}{A^{n}_{j+1}}\right.
   \end{equation*}
   \begin{equation}\label{4e}
    \left.-2\frac{(Q^{n}_{j})^{2}}{A^{n}_{j}}+\frac{(Q^{n}_{j-1})^{2}}{A^{n}_{j-1}}\right\}+\frac{\Delta t^{2}}{2\Delta x}
    \Gamma_{0}\left(\frac{Q_{j}^{n}|Q^{n}_{j}|}{(A_{j}^{n})^{\frac{7}{3}}}-\frac{Q_{j-1}^{n}|Q^{n}_{j-1}|}{(A_{j-1}^{n})^{\frac{7}{3}}}\right),
    \end{equation}
    where
    \begin{equation}\label{3a}
    \Gamma_{0}=\frac{gn_{1}^{2}}{1.49^{2}}P^{\frac{4}{3}}.
    \end{equation}
    Indeed, since $|r(x,t)|<<|A(x,t)|,$ $|Q(x,t)|,$ $\forall$ $(x,t)\in[0,T_{1}]\times[0,L],$ the tracking of the last term in $(\ref{3e})$
    does not compromise the result.\\

    Since the analysis considers the Von Neumann approach, we should put $\phi=k\Delta x$ and take $A_{j}^n=e^{at^{n}}e^{ikx_{j}}$ and
     $Q_{j}^n=e^{bt^{n}}e^{ikx_{j}}$ (where $a,b\in \mathbb{C}$ with $b=b_{1}+ib_{2}$ and $a=a_{1}+ia_{2},$ where $a_{j},b_{j}\in\mathbb{R}$
    and $k$ is the wave number). For the sake of readability, we assume in the following that $a_{2}=b_{2}.$ Replacing this into relation
    $(\ref{4e})$ to get
   \begin{equation*}
    e^{a(t^{n}+\Delta t)}e^{ikx_{j}}=e^{at^{n}}e^{ikx_{j}}-\frac{\Delta t}{2\Delta x}\left(e^{bt^{n}}e^{ik(x_{j}+\Delta x)}
    -e^{bt^{n}}e^{ik(x_{j}-\Delta x)}\right)+
   \end{equation*}
   \begin{equation*}
    \frac{1}{2}\left(\frac{\Delta t}{\Delta x}\right)^{2} \left\{\frac{g}{2T}\left[e^{2at^{n}}e^{2ik(x_{j}+\Delta x)}
    -2e^{2at^{n}}e^{2ik(x_{j})}+e^{2at^{n}}e^{2ik(x_{j}-\Delta x)}\right]+e^{2(b-a)t^{n}}\right.
   \end{equation*}
   \begin{equation}\label{5e}
    \left.-2e^{2(b-a)t^{n}}+e^{2(b-a)t^{n}}\right\}+\frac{\Delta t^{2}}{2\Delta x}
    \Gamma_{0}\left(\frac{e^{bt^{n}}e^{ikx_{j}}|e^{bt^{n}}|}{e^{\frac{7}{3}at^{n}}e^{\frac{7}{3}ikx_{j}}}-
    \frac{e^{bt^{n}}e^{ik(x_{j}-\Delta x)}|e^{bt^{n}}|}{e^{\frac{7}{3}at^{n}}e^{\frac{7}{3}ik(x_{j}-\Delta x)}}\right).
    \end{equation}
    Dividing side by side relation $(\ref{5e})$ by $e^{at^{n}}e^{ikx_{j}}$ results in
   \begin{equation*}
    e^{a\Delta t}=1-\frac{\Delta t}{2\Delta x}\left(e^{i\phi}-e^{-i\phi}\right)e^{(b-a)t^{n}}+
    \frac{1}{2}\left(\frac{\Delta t}{\Delta x}\right)^{2}\frac{g}{2T}\left[e^{at^{n}}e^{ikx_{j}}e^{2i\phi}
    -2e^{at^{n}}e^{ikx_{j}}\right.
   \end{equation*}
   \begin{equation}\label{6e}
    \left.+e^{at^{n}}e^{ikx_{j}}e^{-2i\phi}\right]+\frac{\Delta t^{2}}{2\Delta x}
    \Gamma_{0}|e^{bt^{n}}|e^{(b-\frac{10}{3}a)t^{n}}e^{-\frac{7}{3}ikx_{j}}(1-e^{\frac{4}{3}i\phi}).
    \end{equation}
    Using the identities: $e^{i\phi}-e^{-i\phi}=2i\sin(\phi),$ $e^{2i\phi}-2+e^{-2i\phi}=-4\sin^{2}(\phi)$\\ and
    $1-e^{\frac{4}{3}i\phi}=2\sin^{2}(\frac{2}{3}\phi)+i\sin(\frac{4}{3}\phi),$ equation $(\ref{6e})$ becomes
    \begin{eqnarray}\label{E8}
      \notag e^{a\Delta t} &=& 1-i\frac{\Delta t}{\Delta x}\sin(\phi)e^{(b_{1}-a_{1})t^{n}}-\left(\frac{\Delta t}{\Delta x}\right)^{2}
      \frac{g}{T}\sin^{2}(\phi)e^{a_{1}t^{n}}[\cos(a_{2}t^{n}+kx_{j})+i\sin(a_{2}t^{n}+kx_{j})]\\
     \notag &-& i\frac{\Delta t^{2}}{\Delta x}\Gamma_{0}\sin(\frac{2}{3}\phi)\cos(\frac{2}{3}\phi)e^{(2b_{1}-\frac{10}{3}a_{1})t^{n}}
     \left(\cos\left(\frac{7}{3}[a_{2}t^{n}+kx_{j}]\right)-i\sin\left(\frac{7}{3}[a_{2}t^{n}+kx_{j}]\right)\right)\\
    \notag &-& 2\left(\frac{\Delta t}{\Delta x}\right)^{2}\sin^{2}(\frac{1}{2}\phi)e^{2(b_{1}-a_{1})t^{n}}\\
    \notag &=&  1-i\frac{\Delta t}{\Delta x}\mu^{n}\sin(\phi)-\left(\frac{\Delta t}{\Delta x}\right)^{2}
      \frac{g}{T}\sin^{2}(\phi)e^{a_{1}t^{n}}[\cos(a_{2}t^{n}+kx_{j})+i\sin(a_{2}t^{n}+kx_{j})]\\
     \notag &-& i\frac{\Delta t^{2}}{\Delta x}\Gamma_{0}\sin(\frac{2}{3}\phi)\cos(\frac{2}{3}\phi)e^{(2b_{1}-\frac{10}{3}a_{1})t^{n}}
     \left(\cos\left(\frac{7}{3}[a_{2}t^{n}+kx_{j}]\right)-i\sin\left(\frac{7}{3}[a_{2}t^{n}+kx_{j}]\right)\right)\\
      &-& 2\left(\frac{\Delta t}{\Delta x}\right)^{2}(\mu^{n})^{2}\sin^{2}(\frac{1}{2}\phi),
    \end{eqnarray}
    where
    \begin{equation}\label{E8a}
        \mu^{n}=\frac{|Q^{n}|}{|A^{n}|}=\frac{e^{b_{1}t^{n}}}{e^{a_{1}t^{n}}}=e^{(b_{1}-a_{1})t^{n}}.
    \end{equation}
    Putting
    \begin{equation}\label{E9a}
    \alpha_{2}=\frac{7}{3}(a_{2}t^{n}+kx_{j});\text{\,\,\,\,}\gamma_{2}=\Gamma_{0}(\mu^{n})^{2}e^{-\frac{4}{3}a_{1}t^{n}}e^{-i\alpha_{2}};
    \text{\,\,\,}\alpha_{3}=a_{2}t^{n}+kx_{j}\text{\,\,\,\,}\gamma_{1}=\mu^{n};\text{\,\,\,}\gamma_{4}=\gamma_{1}^{2};
    \text{\,\,\,}\gamma_{3}=\frac{g}{T}e^{a_{1}t^{n}}e^{i\alpha_{3}}.
    \end{equation}
    Utilizing this, equation $(\ref{E8})$ yields
     \begin{eqnarray}\label{E10}
      \notag e^{a\Delta t} &=& 1-i\frac{\Delta t}{\Delta x}\gamma_{1}\sin(\phi)-\left(\frac{\Delta t}{\Delta x}\right)^{2}
      \gamma_{3}\sin^{2}(\phi)-i\frac{\Delta t^{2}}{\Delta x}\gamma_{2}\sin(\frac{2}{3}\phi)\cos(\frac{2}{3}\phi)-
      2\left(\frac{\Delta t}{\Delta x}\right)^{2}\gamma_{4}\sin^{2}(\frac{1}{2}\phi)\\
     \notag &=& 1-\frac{\Delta t^{2}}{\Delta x}|\gamma_{2}|\sin\alpha_{2}\sin(\frac{2}{3}\phi)\cos(\frac{2}{3}\phi)
     -\left(\frac{\Delta t}{\Delta x}\right)^{2}\left[|\gamma_{3}|\cos\alpha_{3}\sin^{2}(\phi)+2|\gamma_{4}|\sin^{2}
     (\frac{1}{2}\phi)\right]-.\\
      && i\left\{\frac{\Delta t}{\Delta x}|\gamma_{1}|\sin(\phi)+\frac{\Delta t^{2}}{\Delta x}|\gamma_{2}|\cos\alpha_{2}
      \sin(\frac{2}{3}\phi)\cos(\frac{2}{3}\phi)+\left(\frac{\Delta t}{\Delta x}\right)^{2}|\gamma_{3}|\sin\alpha_{3}\sin^{2}(\phi)\right\}
    \end{eqnarray}
    Of course the aim of this report is to give the general picture of necessary condition of stability. Since the formulae can
    become quite heavy, for the sake of simplicity, we consider in all the proofs the following remark which plays a crucial role
    in our study. However, the obtained result is a linear stability condition which can be considered as a necessary condition of stability.

   \begin{remark}\label{r1}
   Since the considered problem is a nonlinear partial differential equations, the solution may contain discontinuity even if the initial
   conditions are smooth enough. To overcome this numerical challenge, we should assume that the phase angle $\phi=k\Delta x$ satisfies $|\phi|<<\pi.$
   In fact, the method could be generally stabilized by adding additional dissipation to the scheme without affecting the order of accuracy.
   \end{remark}

   Now, we must analyze some extreme cases. The extreme cases are obtained when $|\phi|=\pi,$ on the one hand, and when $\phi$ equals
   zero on the other.\\

   $\bullet$ \textbf{Case $|\phi|=\pi.$} For $|\phi|=\pi,$ it comes from equation $(\ref{E10})$ that the amplification factor becomes
   \begin{equation*}
     e^{a\Delta t}= 1+\frac{\sqrt{3}\Delta t^{2}}{4\Delta x}|\gamma_{2}|\sin\alpha_{2}-2\left(\frac{\Delta t}{\Delta x}
    \right)^{2}|\gamma_{4}|-i\frac{\sqrt{3}\Delta t^{2}}{4\Delta x}|\gamma_{2}|\cos\alpha_{2}.
   \end{equation*}
   The squared modulus of the amplification factor equals
   \begin{equation*}
     |e^{a\Delta t}|^{2}=1+\frac{3\Delta t^{4}}{16\Delta x^{2}}|\gamma_{2}|^{2}+4\left(\frac{\Delta t}{\Delta x}
    \right)^{4}|\gamma_{4}|^{2}+2\left(\frac{\sqrt{3}(\Delta t)^{2}}{4\Delta x}|\gamma_{2}|(1+\Delta t^{2}|\gamma_{4}|)\sin\alpha_{2}-
    2\left(\frac{\Delta t}{\Delta x}\right)^{2}|\gamma_{4}|\right).
   \end{equation*}
    But there exists values of $\alpha_{2}$ for which $\sin\alpha_{2}=1.$ So $|e^{a\Delta t}|^{2}>1.$ Thus the scheme is unconditionally
   unstable.\\

   $\bullet$ \textbf{Case $\phi=0.$} In that case, the amplification factor given by equation $(\ref{E10})$ provides $e^{a\Delta t}=1.$ Then the
   modulus is $|e^{a\Delta t}|=1.$ Then, the numerical scheme is neutrally stable. Thus, the MacCormack method is not dissipative in the sense of
   Kreiss \cite{hok} and when applied to complete uni-dimensional shallow water equations with source terms $(\ref{1}).$ That is, the
   computations should become unstable in certain circumstances. This instability is entirely due to the non-linearity of the equations,
   since the same scheme applied to linear shallow water equations without source terms does not diverge, although strong oscillations are
   generated (see for example, \cite{gk,21db,23db}).\\

   $\bullet$ \textbf{Case where $0<|\phi|<<\pi.$} Using the Taylor expansion around $\phi=0,$ and neglecting high-order terms, the squared
   modulus of the amplification factor given by $(\ref{E10})$ is approximated as
   \begin{equation}\label{E12}
    |e^{a\Delta t}|^{2}=\left(1-\frac{2\Delta t^{2}}{3\Delta x}|\gamma_{2}|\phi\sin\alpha_{2}\right)^{2}+
       \left(\frac{\Delta t}{\Delta x}|\gamma_{1}|+\frac{2\Delta t^{2}}{3\Delta x}|\gamma_{2}|\cos\alpha_{2}\right)^{2}\phi^{2}.
   \end{equation}
   For $|e^{a\Delta t}|^{2}$ to be less than one, the quantity $\left(1-\frac{2\Delta t^{2}}{3\Delta x}|\gamma_{2}|\phi\sin\alpha_{2}\right)^{2}+
       \left(\frac{\Delta t}{\Delta x}|\gamma_{1}|+\frac{2\Delta t^{2}}{3\Delta x}|\gamma_{2}|\cos\alpha_{2}\right)^{2}\phi^{2}$ must be less
   than one. This implies
   \begin{equation}\label{E13}
    \left|1-\frac{2\Delta t^{2}}{3\Delta x}|\gamma_{2}|\phi\sin\alpha_{2}\right|\leq1\text{\,\,\,\,and\,\,\,\,}
    \frac{\Delta t}{\Delta x}|\gamma_{1}|\left|1+\frac{2\Delta t}{3}\left|\frac{\gamma_{2}}{\gamma_{1}}\right|\cos\alpha_{2}\right||\phi|\leq1.
   \end{equation}
    Using simple calculations, it is not hard to see that estimates given by $(\ref{E13})$ are equivalent to
    \begin{equation}\label{E14}
    0\leq\frac{\Delta t^{2}}{\Delta x}|\gamma_{2}|\phi\sin\alpha_{2}\leq3\text{\,\,\,\,and\,\,\,\,}\frac{\Delta t}{\Delta x}
    \left|1+\frac{2\Delta t}{3}\left|\frac{\gamma_{2}}{\gamma_{1}}\right|\cos\alpha_{2}\right|\leq|\gamma_{1}|^{-1}|\phi|^{-1}.
   \end{equation}
    Since $\frac{\Delta t^{2}}{\Delta x}|\gamma_{2}|\phi\sin\alpha_{2}\leq\frac{\Delta t^{2}}{\Delta x}|\gamma_{2}||\phi|$ and
    $\left|1+\frac{2\Delta t}{3}\left|\frac{\gamma_{2}}{\gamma_{1}}\right|\cos\alpha_{2}\right|\leq1+\frac{2\Delta t}{3}
    \left|\frac{\gamma_{2}}{\gamma_{1}}\right|,$ it comes from this fact and inequalities $(\ref{E14})$ that the numerical scheme $(\ref{66})$
    is stable if
   \begin{equation*}
   \frac{\Delta t^{2}}{\Delta x}|\gamma_{2}||\phi|\leq3\text{\,\,\,\,and\,\,\,\,}\frac{\Delta t}{\Delta x}
    \left(1+\frac{2\Delta t}{3}\left|\frac{\gamma_{2}}{\gamma_{1}}\right|\right)\leq|\gamma_{1}|^{-1}|\phi|^{-1},
   \end{equation*}
   which is equivalent to
   \begin{equation}\label{E15}
    \frac{\Delta t^{3}}{\Delta x^{2}}\left(1+\frac{2\Delta t}{3}\left|\frac{\gamma_{2}}{\gamma_{1}}\right|\right)\leq3|
    \gamma_{2}\gamma_{1}|^{-1}|\phi|^{-2}.
   \end{equation}
    Since $\mu^{n}\mu_{j}^{n}=|\mu^{n}|=\left|\frac{Q_{j}^{n}}{A_{j}^{n}}\right|=e^{(b_{1}-a_{1})t^{n}},$ it comes from relation
    $(\ref{E9a})$ that
   \begin{equation}\label{E16}
    |\gamma_{2}|=\Gamma_{0}(\mu^{n})^{2}e^{-\frac{4}{3}a_{1}t^{n}}\text{\,\,\,and\,\,\,} |\gamma_{1}|=\mu^{n}.\text{\,\,\,\,So\,\,\,\,}
    |\gamma_{2}\gamma_{1}|=\Gamma_{0}(\mu^{n})^{3}e^{-\frac{4}{3}a_{1}t^{n}}\text{\,\,\,and\,\,\,}
    \left|\frac{\gamma_{2}}{\gamma_{1}}\right|=\Gamma_{0}\mu^{n}e^{-\frac{4}{3}a_{1}t^{n}}.
   \end{equation}
   The proof of Proposition $\ref{p2}$ is completed thank to relations $(\ref{E16})$ and equality $|A^{n}|=e^{a_{1}t^{n}}$
   \end{proof}

   \begin{remark}\label{r2}
    The Von Neumann stability approach, based on a Fourier analysis in the space domain has been developed for non-linear one-dimensional
    complete shallow water equations with source terms. Although the stability condition has not be derived analytically, we have analyzed
    the properties of amplification factor numerically (by use of Taylor series expansion), which contain information on the dispersion
    and diffusion errors of the considered numerical scheme. It is worth noticing that we used a local, linearized stability analysis to
    obtain estimate $(\ref{11e}),$ which must be considered as a necessary condition of stability for the numerical scheme $(\ref{66}).$
    Furthermore, since the aim of the work is to analyze the stability condition of the MacCormack scheme $(\ref{66})$-$(\ref{67}),$ from
    now on, we should focus the study on the case $0<|\phi|<<\pi.$
   \end{remark}

   Now, we are going to give a necessary condition of stability for the numerical scheme $(\ref{67}).$

   \begin{proposition}\label{p3}
   The numerical scheme $(\ref{67})$ is stable if the following estimate is satisfied
      \begin{equation*}
    \Delta t\left(3P_{n}W_{1}(\Delta t,\Delta x)+\frac{1}{\Delta x}\max\left\{W_{2}(\Delta t,\Delta x);
    W_{3}(\Delta t,\Delta x)\right\}\right)\leq\max\left\{1+\sqrt{1-r^{*}};\sqrt{r^{*}}\right\},
    \end{equation*}
   where $r^{*}\in(0;1)$ and
    \begin{equation*}
    W_{1}(\Delta t,\Delta x)=\frac{1}{2}+\left[1+4\left(\Delta t+\Delta t^{2}+\frac{\Delta t}{\Delta x}+\frac{\Delta t^{2}}{\Delta x}
    +\frac{\Delta t^{2}}{\Delta x}+\left(+\frac{\Delta t^{2}}{\Delta x}\right)^{2}\right)
    \max\{P_{n},2P_{n}^{2},\frac{1}{2}R_{n}|\phi|,\mu^{n}|\phi|,\right.
   \end{equation*}
   \begin{equation*}
    \left.P_{n}\mu^{n}|\phi|,6P_{n}^{2}\mu^{n}|\phi|,2R_{n}P_{n}|\phi|\}\right];
    \end{equation*}
   \begin{equation*}
    W_{2}(\Delta t,\Delta x)=(P_{n}+\frac{1}{2}R_{n})|\phi|+\mu^{n}\left\{1+4\left[\Delta t+\Delta t^{2}+\frac{\Delta t}{\Delta x}
    +\frac{\Delta t^{2}}{\Delta x}+\frac{\Delta t^{3}}{\Delta x}+\left(\frac{\Delta t}{\Delta x}\right)^{2}\right]\right.
   \end{equation*}
   \begin{equation*}
    \left[1+(4+R_{n}(\mu^{n})^{-1}+R_{n}^{-1}+N_{n}R_{n}(\mu^{n})^{-1})\left(1+\frac{\Delta t}{\Delta x}\mu^{n}\right)|\phi|\right]
   \end{equation*}
   \begin{equation*}
   \left.\max\left\{P_{n},N_{n},R_{n}P_{n},N_{n}P_{n},R_{n}\mu^{n},P_{n}^{2},N_{n}^{2},P_{n}^{2}\mu^{n},R_{n}^{2}\mu^{n},
   R_{n}^{2}(\mu^{n})^{2},R_{n}P_{n}\mu^{n},R_{n}N_{n}\mu^{n}\right\}\right\};
   \end{equation*}
   \begin{equation*}
    W_{3}(\Delta t,\Delta x)=(P_{n}+\frac{1}{2}R_{n})|\phi|+\frac{3}{2}\mu^{n}\left\{|\phi|+2\left[\Delta t+\Delta t^{2}+\frac{\Delta t}
    {\Delta x}+\frac{\Delta t^{2}}{\Delta x}+\frac{\Delta t^{3}}{\Delta x}+\left(\frac{\Delta t}{\Delta x}\right)^{2}\right]\right.
   \end{equation*}
   \begin{equation*}
    \left[1+N_{n}R_{n}(\mu^{n})^{-1}+\left(4+2P_{n}+R_{n}^{-1}+R_{n}^{-1}\mu^{n}+4N_{n}R_{n}^{-1}(\mu^{n})^{-1}+\frac{\Delta t}
    {\Delta x}\mu^{n}\right)|\phi|\right]
   \end{equation*}
   \begin{equation*}
   \left.\max\left\{P_{n},N_{n},N_{n}P_{n},R_{n}\mu^{n},P_{n}^{2},N_{n}^{2},R_{n}^{2}\mu^{n},P_{n}R_{n}\mu^{n},R_{n}^{2}(\mu^{n})^{2},
   R_{n}N_{n}\mu^{n}\right\}\right\};
   \end{equation*}
    with
      \begin{equation*}
     \mu^{n}=\frac{|Q^{n}_{j}|}{|A_{j}^{n}|}=e^{(b_{1}-a_{1})t^{n}};\text{\,\,\,}R_{n}=\frac{g}{T}|A_{j}^{n}||\mu^{n}|^{-1}+\mu^{n};
    \text{\,\,}P_{n}=\frac{P\overline{\tau}}{\rho}|Q_{j}^{n}|^{-1}+\Gamma_{0}\mu^{n}|A_{j}^{n}|^{-\frac{4}{3}};\text{\,\,\,}
    N_{n}=\frac{g}{2T}|Q_{j}^{n}||\mu^{n}|^{-1}.
   \end{equation*}
   \end{proposition}
   The proof of Proposition $\ref{p3}$ requires some intermediate results which play a crucial role in the study of the amplification
   factor associated with $(\ref{67}).$

   \begin{lemma}\label{l1}
   Let $n$ and $j$ be nonnegative integers. Then the terms $\frac{1}{2}\left(Q^{n}_{j}+Q_{j}^{\overline{n+1}}\right)$ and\\
   $\frac{g}{4T}\left[(A_{j}^{\overline{n+1}})^{2}-(A_{j-1}^{\overline{n+1}})^{2}\right]$ can be
   approximated as
   \begin{equation}\label{E21a}
    \frac{1}{2}\left(Q^{n}_{j}+Q_{j}^{\overline{n+1}}\right)=Q^{n}_{j}\left\{1+\frac{\Delta t}{2\Delta x}C_{11}\phi
    +\frac{1}{2}\Delta tC_{12}+i\left(\frac{\Delta t}{2\Delta x}\overline{C}_{11}\phi+\frac{1}{2}\Delta t\overline{C}_{12}\right)\right\}+O(\phi^{2}),
   \end{equation}
   and
   \begin{equation}\label{E24}
    \frac{g}{4T}\left[(A_{j}^{\overline{n+1}})^{2}-(A_{j-1}^{\overline{n+1}})^{2}\right]=Q_{j}^{n}(C_{21}+i\overline{C}_{21})\phi+O(\phi^{2}),
   \end{equation}
   where
   \begin{equation*}
    C_{11}=\frac{g}{T}|A^{n}||\mu^{n}|^{-1}\sin\alpha_{3};\text{\,\,}\overline{C}_{11}=-\frac{g}{T}|A^{n}||\mu^{n}|^{-1}\cos\alpha_{3}-\mu^{n};
    \text{\,\,}C_{12}=\frac{P\tau}{\rho}|Q^{n}|^{-1}\cos\alpha_{3}-\Gamma_{0}\mu^{n}|A^{n}|^{-\frac{4}{3}}\cos\alpha_{2};
   \end{equation*}
   \begin{equation}\label{E22}
    \overline{C}_{12}=\frac{P\tau}{\rho}|Q^{n}|^{-1}\sin\alpha_{3}-\Gamma_{0}\mu^{n}|A^{n}|^{-\frac{4}{3}}\sin\alpha_{2};\text{\,\,\,}
    C_{21}=-\frac{g}{2T}|Q^{n}||\mu^{n}|^{-1}\sin\alpha_{3};\text{\,\,\,}\overline{C}_{21}=\frac{g}{2T}|Q^{n}||\mu^{n}|^{-1}\cos\alpha_{3};
   \end{equation}
   where
   \begin{equation}\label{E23}
    \alpha_{3}=a_{2}t^{n}+kx_{j},\text{\,\,\,and\,\,\,}\alpha_{2}=\frac{7}{3}\alpha_{3}.
   \end{equation}
   \end{lemma}

   \begin{proof}
   First, we recall that $A_{j}^{n}=e^{at^{n}}e^{ikx_{j}}=e^{(a_{1}+ia_{2})t^{n}}e^{ikx_{j}}$ and $Q_{j}^{n}=e^{bt^{n}}e^{ikx_{j}}
   =e^{(b_{1}+ia_{2})t^{n}}e^{ikx_{j}}.$ Expanding the Taylor series around $\phi$ and neglecting the terms of high-order to
     obtain $e^{i\phi}-1=i\phi+O(\phi^{2})$ and $e^{2i\phi}-1=2i\phi+O(\phi^{2})$ (this is true according to Remark $\ref{r1}$).
   Utilizing this, we get
   \begin{equation*}
    \frac{1}{2}\left(Q^{n}_{j}+Q_{j}^{\overline{n+1}}\right)= Q^{n}_{j}-\frac{\Delta t}{2\Delta x}
     \left\{\frac{g}{2T}\left[(A_{j+1}^{n})^{2}-(A_{j}^{n})^{2}\right]+\frac{(Q_{j+1}^{n})^{2}}{A_{j+1}^{n}}-
     \frac{(Q_{j}^{n})^{2}}{A_{j}^{n}}\right\}+\frac{1}{2}\Delta t\left(\frac{P\overline{\tau}}{\rho}-\Gamma_{0}
      \frac{Q_{j}^{n}|Q^{n}_{j}|}{(A_{j}^{n})^{\frac{7}{3}}}\right)
   \end{equation*}
    \begin{equation*}
     =Q^{n}_{j}-\frac{\Delta t}{2\Delta x}\left\{\frac{g}{2T}\left[e^{2at^{n}}e^{2ikx_{j+1}}-e^{2at^{n}}e^{2ikx_{j}}\right]
     +\frac{e^{2bt^{n}}e^{2ikx_{j+1}}}{e^{at^{n}}e^{ikx_{j+1}}}-\frac{e^{2bt^{n}}e^{2ikx_{j}}}
     {e^{at^{n}}e^{ikx_{j}}}\right\}+\frac{1}{2}\Delta t\left(\frac{P\overline{\tau}}{\rho}-\right.
    \end{equation*}
     \begin{equation*}
        \left.\Gamma_{0}\frac{Q_{j}^{n}|e^{bt^{n}}e^{ikx_{j}}|}{e^{\frac{7}{3}at^{n}}e^{\frac{7}{3}ikx_{j}}}\right)=
          Q^{n}_{j}-\frac{\Delta t}{2\Delta x}\left\{\frac{g}{2T}e^{2at^{n}}e^{2ikx_{j}}(e^{2i\phi}-1)+
          e^{(2b-a)t^{n}}e^{ikx_{j}}(e^{i\phi}-1)\right\}+\frac{1}{2}Q_{j}^{n}\Delta t\left(\frac{P\overline{\tau}}{\rho }
          (Q_{j}^{n})^{-1}\right.
     \end{equation*}
     \begin{equation*}
       \left.-\Gamma_{0}\frac{|e^{bt^{n}}|}{e^{\frac{7}{3}at^{n}}e^{\frac{7}{3}ikx_{j}}}\right)=Q^{n}_{j}\left\{1-i\frac{\Delta t}
       {2\Delta x}\phi\left\{\frac{g}{T}e^{(2a-b)t^{n}}e^{ikx_{j}}+e^{(b-a)t^{n}}\right\}+\frac{1}{2}\Delta t\left(\frac{P\overline{\tau}}
       {\rho}(Q_{j}^{n})^{-1}-\right.\right.
     \end{equation*}
        \begin{equation*}
    \left.\left.\Gamma_{0}\frac{|Q^{n}_{j}||A_{j}^{n}|^{-\frac{7}{3}}}{e^{\frac{7}{3}i(a_{2}t^{n}+kx_{j})}}\right)\right\}
     +O(\phi^{2})= Q^{n}_{j}\left\{1-i\frac{\Delta t}{2\Delta x}\phi\left\{\frac{g}{T}e^{(2a_{1}-b_{1})t^{n}}e^{i(a_{2}t^{n}+kx_{j})}
     +e^{(b_{1}-a_{1})t^{n}}\right\}+\frac{1}{2}\Delta t\left(\frac{P\overline{\tau}}{\rho}(Q_{j}^{n})^{-1}\right.\right.
        \end{equation*}
        \begin{equation*}
        \left.\left.-\Gamma_{0}\frac{|Q^{n}_{j}||A_{j}^{n}|^{-\frac{7}{3}}}{e^{\frac{7}{3}i(a_{2}t^{n}
    +kx_{j})}}\right)\right\}+O(\phi^{2})=Q^{n}_{j}\left\{1-i\frac{\Delta t}{2\Delta x}\phi
    \left\{\frac{g}{T}|A^{n}|(\mu^{n})^{-1}e^{i\alpha_{3}}+\mu^{n}\right\}+\frac{1}{2}\Delta t
    \left(\frac{P\overline{\tau}}{\rho}(Q_{j}^{n})^{-1}-\right.\right.
        \end{equation*}
    \begin{equation*}
    \left.\left.\Gamma_{0}\frac{|\mu^{n}||A_{j}^{n}|^{-\frac{4}{3}}}{e^{\frac{7}{3}i(a_{2}t^{n}+kx_{j})}}\right)\right\}+O(\phi^{2})=
    Q^{n}_{j}\left\{1+\frac{\Delta t}{2\Delta x}C_{11}\phi+\frac{1}{2}\Delta tC_{12}+i\left(\frac{\Delta t}{2\Delta x}\overline{C}_{11}\phi
    +\frac{1}{2}\Delta t\overline{C}_{12}\right)\right\}+O(\phi^{2}),
    \end{equation*}
   where $C_{11},$ $\overline{C}_{11},$ $C_{12}$ and $\overline{C}_{12}$ are defined by relation $(\ref{E22}).$\\

   On the other hand, using equation $(\ref{64})$ and applying the Taylor expansion around $\phi$ and neglecting the high-order terms
   together with the term in $r,$ we get
   \begin{equation*}
    \frac{g}{4T}\left[(A_{j}^{\overline{n+1}})^{2}-(A_{j-1}^{\overline{n+1}})^{2}\right]=\frac{g}{4T}\left(A_{j}^{\overline{n+1}}-
    A_{j-1}^{\overline{n+1}}\right)\left(A_{j}^{\overline{n+1}}+A_{j-1}^{\overline{n+1}}\right)=\frac{g}{4T}\left(A_{j}^{n}
    -A_{j-1}^{n}\right.
   \end{equation*}
   \begin{equation*}
    \left.-\frac{\Delta t}{\Delta x}(Q_{j+1}^{n}-2Q_{j}^{n}+Q_{j-1}^{n})+\Delta t(r_{j}^{n}-r_{j-1}^{n})\right)\left(A_{j}^{n}
    +A_{j-1}^{n}-\frac{\Delta t}{\Delta x}(Q_{j+1}^{n}-Q_{j-1}^{n})+\Delta t(r_{j}^{n}+r_{j-1}^{n})\right)
   \end{equation*}
   \begin{equation*}
    \approx \frac{g}{4T}\left(e^{at^{n}}e^{ikx_{j}}-e^{at^{n}}e^{ikx_{j-1}}-\frac{\Delta t}{\Delta x}
    (e^{bt^{n}}e^{ikx_{j+1}}-2e^{bt^{n}}e^{ikx_{j}}+e^{bt^{n}}e^{ikx_{j-1}})\right)\left(e^{at^{n}}e^{ikx_{j}}+e^{at^{n}}e^{ikx_{j-1}}\right.
   \end{equation*}
   \begin{equation*}
    \left.-\frac{\Delta t}{\Delta x}(e^{bt^{n}}e^{ikx_{j+1}}-e^{bt^{n}}e^{ikx_{j-1}})\right)=\frac{g}{4T}e^{2bt^{n}}e^{2ikx_{j}}
    \left\{e^{(a-b)t^{n}}(1-e^{-i\phi})-\frac{\Delta t}{\Delta x}(e^{i\phi}-2+e^{-i\phi})\right\}\times
   \end{equation*}
   \begin{equation*}
    \left\{e^{(a_{1}-b_{1})t^{n}}(1+e^{-i\phi})-\frac{\Delta t}{\Delta x}(e^{i\phi}-e^{-i\phi})\right\}=i\frac{g}{2T}
    \phi Q_{j}^{n}|Q^{n}|e^{i\alpha_{3}}e^{(a_{1}-b_{1})t^{n}}+O(\phi^{2})=
   \end{equation*}
   \begin{equation*}
    =Q_{j}^{n}(C_{21}+i\overline{C}_{21})\phi+O(\phi^{2}),
   \end{equation*}
   where $C_{21}$ and $\overline{C}_{21}$ are given by relation $(\ref{E22}).$\\
   \end{proof}

   \begin{lemma}\label{l2}
   Let consider $n$ and $j$ be two nonnegative integers. Then the term $\frac{\left(Q_{j}^{\overline{n+1}}\right)^{2}}
   {A_{j}^{\overline{n+1}}}-\frac{\left(Q_{j-1}^{\overline{n+1}}\right)^{2}}{A_{j-1}^{\overline{n+1}}}$ can be approximated as
   \begin{equation*}
    \frac{(Q_{j}^{\overline{n+1}})^{2}}{A_{j}^{\overline{n+1}}}-\frac{(Q_{j-1}^{\overline{n+1}})^{2}}{A_{j-1}^{\overline{n+1}}}
    =Q^{n}_{j}\mu^{n}\left\{H_{1}-K_{1}^{2}+K^{2}_{2}+2K_{1}K_{2}(1-\frac{\Delta t}{\Delta x}\mu^{n})\phi+i\left(H_{2}-2K_{1}K_{2}
    +\right.\right.
   \end{equation*}
   \begin{equation}\label{Ee}
    \left.\left.(K_{2}^{2}-K_{1}^{2})(1-\frac{\Delta t}{\Delta x}\mu^{n})\phi\right)\right\}+O(\phi^{2}),
   \end{equation}
   where the functions $H_{1}$, $H_{2}$, $K_{1}^{2}$, $K_{2}^{2}$ and $K_{1}K_{2}$ are given by
    \begin{equation}\label{E30}
    H_{1}=1+2\Delta tC_{12}+\Delta t^{2}(C_{12}^{2}-\overline{C}_{12}^{2})+2\frac{\Delta t^{2}}{\Delta x}\left[C_{11}C_{12}-\overline{C}_{11}
    \overline{C}_{12}-\mu^{n}(\overline{C}_{12}+\Delta tC_{12}\overline{C}_{12})\right]\phi;
    \end{equation}
    \begin{equation}\label{E30a}
    H_{2}=\Delta t\overline{C}_{12}+\Delta t^{2}\overline{C}_{12}C_{12}+\frac{\Delta t}{\Delta x}\left[\overline{C}_{11}+
    \mu^{n}+\Delta t(C_{11}\overline{C}_{12}+\overline{C}_{11}C_{12}+\mu^{n}(2C_{12}+\Delta t(C_{12}^{2}-\overline{C}_{12}^{2})))\right]
    \phi;
    \end{equation}
   \begin{equation*}
    K_{1}^{2}=1+\left(\frac{\Delta t}{2\Delta x}\right)^{2}\left[\overline{C}_{21}^{2}+4C_{21}\overline{C}_{21}\phi\right]+
    \Delta t^{2}\left[\overline{C}_{12}^{2}-2\overline{C}_{12}\overline{C}_{22}\phi\right]-\frac{\Delta t}{\Delta x}(\overline{C}_{21}+
    2C_{21}\phi)+
    \end{equation*}
    \begin{equation}\label{E41c}
    \Delta t(C_{12}-2C_{22}\phi)-\frac{\Delta t^{2}}{\Delta x}\left[C_{12}\overline{C}_{21}-(\overline{C}_{21}C_{22}-2C_{12}C_{21})\phi\right];
    \end{equation}
    \begin{equation*}
    K_{2}^{2}=\left(\frac{\Delta t}{2\Delta x}\right)^{2}(\mu^{n})^{2}\left[C_{11}^{2}+2C_{11}(1-\overline{C}_{11}-\mu^{n})\phi\right]+
    \Delta t^{2}\left[C_{12}^{2}-2C_{12}\overline{C}_{22}\phi\right]+\frac{\Delta t}{\Delta x}\mu^{n}C_{11}\phi+
    \end{equation*}
    \begin{equation}\label{E41d}
    2\Delta t\overline{C}_{12}\phi+\frac{\Delta t^{2}}{\Delta x}\mu^{n}\left[C_{11}\overline{C}_{12}-C_{11}\overline{C}_{22}\phi+
    \overline{C}_{12}(1-\overline{C}_{11}-\mu^{n})\phi\right];
    \end{equation}
    and
     \begin{equation*}
    K_{1}K_{2}=-\left\{\phi-\left(\frac{\Delta t}{2\Delta x}\right)^{2}\mu^{n}\left[C_{11}\overline{C}_{21}
    +[\overline{C}_{21}(1-\overline{C}_{11}-\mu^{n})+2C_{11}C_{21}]\phi\right]+
    \Delta t\left[C_{12}-(C_{12}\overline{C}_{22}-C_{12})\phi\right]\right.
    \end{equation*}
    \begin{equation*}
    +\frac{\Delta t}{2\Delta x}\mu^{n}\left[C_{11}+(1-\overline{C}_{11}-\mu^{n}-\overline{C}_{21}(\mu^{n})^{-1})\phi\right]
    +\frac{\Delta t^{2}}{2\Delta x}\mu^{n}\left[C_{11}C_{12}-\overline{C}_{12}\overline{C}_{21}(\mu^{n})^{-1}+[C_{12}(1-\overline{C}_{11}-
    \mu^{n})-\right.
    \end{equation*}
    \begin{equation}\label{E41e}
    \left.\left.C_{11}C_{22}+\overline{C}_{21}C_{22}(\mu^{n})^{-1}-2C_{21}\overline{C}_{12}(\mu^{n})^{-1}]\phi\right]+\Delta t^{2}
    [C_{12}\overline{C}_{12}-(C_{12}\overline{C}_{22}+\overline{C}_{12}C_{22})\phi]\right\};
    \end{equation}
    $C_{rs},$ $\overline{C}_{rs},$ $r,s=1,2,$ are given by $(\ref{E22})$ and $\alpha_{2}$, $\alpha_{3}$ come from relation
   $(\ref{E23}).$ Furthermore
   \begin{equation}\label{E22a}
    C_{22}=\frac{4}{3}\Gamma_{0}\mu^{n}|A_{j}^{n}|^{-\frac{4}{3}}\sin\alpha_{2}\text{\,\,\,and\,\,\,}
    \overline{C}_{22}=\frac{4}{3}\Gamma_{0}\mu^{n}|A_{j}^{n}|^{-\frac{4}{3}}\cos\alpha_{2}.
   \end{equation}
   \end{lemma}

   \begin{proof}
   The following identity holds
   \begin{equation*}
    Q_{j}^{\overline{n+1}}=2\left(\frac{1}{2}(Q_{j}^{\overline{n+1}}+Q_{j}^{n})\right)-Q_{j}^{n}.
   \end{equation*}
   This fact, along with equation $(\ref{E21a})$ result in
   \begin{equation}\label{E26}
    Q_{j}^{\overline{n+1}}=Q^{n}_{j}\left\{1+\frac{\Delta t}{\Delta x}C_{11}\phi+\Delta tC_{12}+
    i\left(\frac{\Delta t}{\Delta x}\overline{C}_{11}\phi+\Delta t\overline{C}_{12}\right)\right\}+O(\phi^{2}),
   \end{equation}
   where $C_{1j}$ and $\overline{C}_{1j},$ $j=1,2,$ are given by relations $(\ref{E22})$ and $(\ref{E23}).$\\

   Taking the square of $Q_{j}^{\overline{n+1}}$ and neglecting the high-order terms in $\phi$, it comes from equation
   $(\ref{E26})$ that
   \begin{equation*}
    (Q_{j}^{\overline{n+1}})^{2}=(Q^{n}_{j})^{2}\left\{1+\frac{\Delta t}{\Delta x}C_{11}\phi+\Delta tC_{12}+
    i\left(\frac{\Delta t}{\Delta x}\overline{C}_{11}\phi+\Delta t\overline{C}_{12}\right)\right\}^{2}+O(\phi^{2})=
   \end{equation*}
   \begin{equation*}
   (Q^{n}_{j})^{2}\left\{\left(1+\frac{\Delta t}{\Delta x}C_{11}\phi+\Delta tC_{12}\right)^{2}-\left(\frac{\Delta t}{\Delta x}
   \overline{C}_{11}\phi+\Delta t\overline{C}_{12}\right)^{2}+2i\left(1+\frac{\Delta t}{\Delta x}C_{11}\phi
    +\Delta tC_{12}\right)\times\right.
   \end{equation*}
   \begin{equation*}
    \left.\left(\frac{\Delta t}{\Delta x}\overline{C}_{11}\phi+\Delta t\overline{C}_{12}\right)\right\}+O(\phi^{2})=
    (Q^{n}_{j})^{2}\left\{1+2\Delta tC_{12}+\Delta t^{2}(C_{12}^{2}-\overline{C}_{12}^{2})+2\frac{\Delta t^{2}}{\Delta x}(C_{11}C_{12}-
    \overline{C}_{11}\overline{C}_{12})\phi+\right.
   \end{equation*}
   \begin{equation}\label{E27}
    \left.2i\left(\Delta t\overline{C}_{12}+\Delta t^{2}\overline{C}_{12}C_{12}+\frac{\Delta t}{\Delta x}\left[\overline{C}_{11}+
    \Delta t(C_{11}\overline{C}_{12}+\overline{C}_{11}C_{12})\right]\phi\right)\right\}+O(\phi^{2}).
   \end{equation}
   Utilizing equation $(\ref{64})$ together with Remark $\ref{r1}$, simple calculations give
   \begin{equation}\label{E28a}
    A_{j}^{\overline{n+1}}=e^{at^{n}}e^{ikx_{j}}-\frac{\Delta t}{\Delta x}e^{bt^{n}}e^{ikx_{j}}(e^{i\phi}-1)=Q_{j}^{n}\left\{e^{(a-b)t^{n}}
    -\frac{\Delta t}{\Delta x}(e^{i\phi}-1)\right\}=Q_{j}^{n}(\mu^{n})^{-1}\left\{1-i\frac{\Delta t}{\Delta x}\mu^{n}\phi\right\}+O(\phi^{2}),
   \end{equation}
   and
   \begin{equation*}
    A_{j-1}^{\overline{n+1}}=e^{at^{n}}e^{ikx_{j-1}}-\frac{\Delta t}{\Delta x}e^{bt^{n}}e^{ikx_{j}}(1-e^{-i\phi})=Q_{j}^{n}\left\{e^{(a-b)t^{n}}
    e^{-i\phi}-\frac{\Delta t}{\Delta x}(1-e^{-i\phi})\right\}=
   \end{equation*}
   \begin{equation}\label{E28}
    Q_{j}^{n}\left\{e^{(a_{1}-b_{1})t^{n}}e^{-i\phi}-\frac{\Delta t}{\Delta x}(1-e^{-i\phi})\right\}=
    Q_{j}^{n}(\mu^{n})^{-1}\left\{1+i\left(-1+\frac{\Delta t}{\Delta x}\mu^{n}\right)\phi\right\}+O(\phi^{2}).
   \end{equation}
   In way similar, one easily shows that
   \begin{equation*}
    Q_{j-1}^{\overline{n+1}}=Q_{j}^{n}\left\{e^{-i\phi}-\frac{\Delta t}{\Delta x}\left\{\frac{g}{2T}|A^{n}_{j}|e^{i\alpha_{3}}
    (1-e^{-2i\phi})+\mu^{n}(1-e^{-i\phi})\right\}+\Delta t\left(\frac{P\overline{\tau}}{\rho}|Q_{j}^{n}|^{-1}e^{-i\alpha_{3}}-
    \right.\right.
   \end{equation*}
   \begin{equation}\label{E41}
    \left.\left.\Gamma_{0}\mu^{n}|A_{j}^{n}|^{-\frac{4}{3}}e^{-i\alpha_{2}}e^{i\frac{4}{3}\phi}\right)\right\}=
    Q^{n}_{j}\left(K_{1}+iK_{1}\right)+O(\phi^{2}),
   \end{equation}
   where
   \begin{equation}\label{E41a}
    K_{1}=1-\frac{\Delta t}{2\Delta x}(\overline{C}_{21}+2C_{21}\phi)+\Delta t(C_{12}-C_{22}\phi),
   \end{equation}
   and
   \begin{equation}\label{E41aa}
    K_{2}=-\phi-\frac{\Delta t}{2\Delta x}\mu^{n}[C_{11}+(1-\overline{C}_{11}-\mu^{n})\phi]-
    \Delta t(\overline{C}_{12}-\overline{C}_{22}\phi),
   \end{equation}
   where $C_{rs},$ $\overline{C}_{rs},$ $r,s=1,2,$ are given by $(\ref{E22})$ and $\alpha_{2}$, $\alpha_{3}$ come from relation
   $(\ref{E23}).$ Furthermore
   \begin{equation*}
    C_{22}=\frac{4}{3}\Gamma_{0}\mu^{n}|A_{j}^{n}|^{-\frac{4}{3}}\sin\alpha_{2}\text{\,\,\,and\,\,\,}
    \overline{C}_{22}=\frac{4}{3}\Gamma_{0}\mu^{n}|A_{j}^{n}|^{-\frac{4}{3}}\cos\alpha_{2}.
   \end{equation*}
   From equation $(\ref{E41}),$ simple computations give
   \begin{equation}\label{E41b}
    (Q_{j-1}^{\overline{n+1}})^{2}=(Q^{n}_{j})^{2}\left(K_{1}^{2}-K_{2}^{2}+2iK_{1}K_{2}\right)+O(\phi^{2}).
   \end{equation}
    Applying Remark $\ref{r1},$ we obtain the following approximation
    \begin{equation*}
    K_{1}^{2}=1+\left(\frac{\Delta t}{2\Delta x}\right)^{2}\left[\overline{C}_{21}^{2}+4C_{21}\overline{C}_{21}\phi\right]+
    \Delta t^{2}\left[\overline{C}_{12}^{2}-2\overline{C}_{12}\overline{C}_{22}\phi\right]-\frac{\Delta t}{\Delta x}(\overline{C}_{21}+
    2C_{21}\phi)+
    \end{equation*}
    \begin{equation*}
    \Delta t(C_{12}-2C_{22}\phi)-\frac{\Delta t^{2}}{\Delta x}\left[C_{12}\overline{C}_{21}-(\overline{C}_{21}C_{22}-2C_{12}C_{21})\phi\right]
    +O(\phi^{2}),
    \end{equation*}
    \begin{equation*}
    K_{2}^{2}=\left(\frac{\Delta t}{2\Delta x}\right)^{2}(\mu^{n})^{2}\left[C_{11}^{2}+2C_{11}(1-\overline{C}_{11}-\mu^{n})\phi\right]+
    \Delta t^{2}\left[C_{12}^{2}-2C_{12}\overline{C}_{22}\phi\right]+\frac{\Delta t}{\Delta x}\mu^{n}C_{11}\phi+
    \end{equation*}
    \begin{equation*}
    2\Delta t\overline{C}_{12}\phi+\frac{\Delta t^{2}}{\Delta x}\mu^{n}\left[C_{11}\overline{C}_{12}-C_{11}\overline{C}_{22}\phi+
    \overline{C}_{12}(1-\overline{C}_{11}-\mu^{n})\phi\right]+O(\phi^{2}),
    \end{equation*}
    and
     \begin{equation*}
    K_{1}K_{2}=-\left\{\phi-\left(\frac{\Delta t}{2\Delta x}\right)^{2}\mu^{n}\left[C_{11}\overline{C}_{21}+
    [\overline{C}_{21}(1-\overline{C}_{11}-\mu^{n})+2C_{11}C_{21}]\phi\right]+
    \Delta t\left[C_{12}-(C_{12}\overline{C}_{22}-C_{12})\phi\right]\right.
    \end{equation*}
    \begin{equation*}
    +\frac{\Delta t}{2\Delta x}\mu^{n}\left[C_{11}+(1-\overline{C}_{11}-\mu^{n}-\overline{C}_{21}(\mu^{n})^{-1})\phi\right]
    +\frac{\Delta t^{2}}{2\Delta x}\mu^{n}\left[C_{11}C_{12}-\overline{C}_{12}\overline{C}_{21}(\mu^{n})^{-1}+[C_{12}(1-\overline{C}_{11}-
    \mu^{n})-\right.
    \end{equation*}
    \begin{equation*}
    \left.\left.C_{11}C_{22}+\overline{C}_{21}C_{22}(\mu^{n})^{-1}-2C_{21}\overline{C}_{12}(\mu^{n})^{-1}]\phi\right]+\Delta t^{2}
    [C_{12}\overline{C}_{12}-(C_{12}\overline{C}_{22}+\overline{C}_{12}C_{22})\phi]\right\}+O(\phi^{2}).
    \end{equation*}
    Combining approximations $(\ref{E27})$ and $(\ref{E28a})$ on the one hand, $(\ref{E41b})$ and $(\ref{E28})$ on the other hand,
    straightforward computations yield
    \begin{equation}\label{E38}
    \frac{(Q_{j}^{\overline{n+1}})^{2}}{A_{j}^{\overline{n+1}}}=Q^{n}_{j}\mu^{n}\left\{H_{1}+iH_{2}\right\}+O(\phi^{2}),
   \end{equation}
     and
     \begin{equation}\label{E47}
    \frac{(Q_{j-1}^{\overline{n+1}})^{2}}{A_{j-1}^{\overline{n+1}}}=Q^{n}_{j}\mu^{n}\left\{K_{1}^{2}-K_{2}^{2}-2K_{1}K_{2}\left(1-
    \frac{\Delta t}{\Delta x}\mu^{n}\right)\phi+i\left[2K_{1}K_{2}+\left(1-\frac{\Delta t}{\Delta x}\mu^{n}\right)(K_{1}^{2}-K_{2}^{2})
    \phi\right]\right\}+O(\phi^{2}),
   \end{equation}
   where $H_{1}$, $H_{2}$, $K_{1}^{2}$, $K_{2}^{2}$ and $K_{1}K_{2}$ are given by equations $(\ref{E30})$, $(\ref{E30a})$ $(\ref{E41c})$,
     $(\ref{E41d})$ and $(\ref{E41e})$, respectively. Subtracting equation $(\ref{E47})$ from approximation $(\ref{E38})$ completes the proof.
   \end{proof}

   \begin{lemma}\label{l3}
   For $n$ and $j$ be nonnegative integers, the term $\frac{Q_{j}^{\overline{n+1}}|Q_{j}^{\overline{n+1}}|}
   {\left(A_{j}^{\overline{n+1}}\right)^{\frac{7}{3}}}$ can be approximated as
     \begin{equation*}
     \frac{Q_{j}^{\overline{n+1}}|Q_{j}^{\overline{n+1}}|}{(A_{j}^{\overline{n+1}})^{\frac{7}{3}}}=Q^{n}|Q^{n}|^{-\frac{4}{3}}
     (\mu^{n})^{\frac{7}{3}}\left\{1+\left(\frac{\Delta t}{\Delta x}\right)^{2}(\mu^{n})^{2}\phi^{2}\right\}^{-\frac{7}{3}}
     \left\{\left(1+\frac{\Delta t}{\Delta x}C_{11}\phi+2\Delta t C_{12}\right)^{2}+\right.
     \end{equation*}
     \begin{equation*}
        \left.\left(\frac{\Delta t}{\Delta x}\overline{C}_{11}\phi+2\Delta t \overline{C}_{12}\right)^{2}\right\}^{\frac{1}{2}}
     \left\{1+\frac{\Delta t}{\Delta x}C_{11}\phi+2\Delta t C_{12}+i\left(\frac{\Delta t}{\Delta x}\overline{C}_{11}\phi
     +2\Delta t \overline{C}_{12}\right)\right\}\times
     \end{equation*}
     \begin{equation}\label{E52a}
     \left\{C_{31}+\frac{\Delta t}{\Delta x}\overline{C}_{32}\phi+i\left[-\overline{C}_{31}+\frac{\Delta t}{\Delta x}C_{32}
     \phi\right]\right\}^{\frac{7}{3}}+O(\phi^{2}).
     \end{equation}
     where the functions $C_{1l}$ and $\widehat{C}_{1l},$ $l=1,2,$ are defined by relations $(\ref{E22})$ and $(\ref{E23}),$ $C_{3l}$
     and $\overline{C}_{3l}$, $l=1,2,$ are given by
      \begin{equation}\label{E52aa}
      C_{31}=\cos\alpha_{2};\text{\,\,\,}\overline{C}_{31}=-\sin\alpha_{2};\text{\,\,\,}
      C_{32}=\mu^{n}\cos\alpha_{2};\text{\,\,\,}\overline{C}_{32}=\mu^{n}\sin\alpha_{2}.
      \end{equation}
       \end{lemma}

   \begin{proof}
   It comes from approximation $(\ref{E26})$ that
   \begin{equation}\label{E50}
    Q_{j}^{\overline{n+1}}=Q^{n}_{j}\left\{1+\frac{\Delta t}{\Delta x}C_{11}\phi+\Delta tC_{12}+
    i\left(\frac{\Delta t}{\Delta x}\overline{C}_{11}\phi+\Delta t\overline{C}_{12}\right)\right\}+O(\phi^{2}),
   \end{equation}
   where $C_{1j}$ and $\overline{C}_{1j},$ $j=1,2,$ are given by relations $(\ref{E22})$ and $(\ref{E23}).$
   So, the modulus of the discharge at the temporary time level $n+1$ and at the position $x_{j}$ is approximated as
   \begin{equation}\label{E50a}
    |Q_{j}^{\overline{n+1}}|=|Q_{j}^{n}|\left\{\left(1+\frac{\Delta t}{\Delta x}C_{12}\phi+\Delta t C_{12}\right)^{2}
    +\left(\frac{\Delta t}{\Delta x}\overline{C}_{11}\phi+\Delta t \overline{C}_{12}\right)^{2}\right\}^{\frac{1}{2}}+O(\phi^{2}),
   \end{equation}
    Similarly, from relation $(\ref{E28a})$ we have that
    \begin{equation*}
    A_{j}^{\overline{n+1}}=Q_{j}^{n}\left\{e^{(a-b)t^{n}}-\frac{\Delta t}{\Delta x}(e^{i\phi}-1)\right\}=
    Q_{j}^{n}(\mu^{n})^{-1}\left\{1-i\frac{\Delta t}{\Delta x}\mu^{n}\phi\right\}+O(\phi^{2}).
   \end{equation*}
    So, the modulus of the cross section at predicted time $t^{n+1}$ and at position $x_{j}$ is approximated by
    \begin{equation}\label{E50b}
     |A_{j}^{\overline{n+1}}|=|Q_{j}^{n}|(\mu^{n})^{-1}\left\{1+\left(\frac{\Delta t}{\Delta x}\right)^{2}
     (\mu^{n})^{2}\phi^{2}\right\}^{\frac{1}{2}}+O(\phi^{2}).
    \end{equation}
     On the other hand, the following equality holds
     \begin{equation}\label{E51}
        \frac{Q_{j}^{\overline{n+1}}|Q_{j}^{\overline{n+1}}|}{(A_{j}^{\overline{n+1}})^{\frac{7}{3}}}=
        |A_{j}^{\overline{n+1}}|^{-\frac{14}{3}}|Q_{j}^{\overline{n+1}}|Q_{j}^{\overline{n+1}}
        \left(\overline{A_{j}^{\overline{n+1}}}\right)^{\frac{7}{3}},
     \end{equation}
     where $\overline{A_{j}^{\overline{n+1}}}$ designates the conjugate of $A_{j}^{\overline{n+1}}.$\\
     Now, combining relations $(\ref{E50}),$ $(\ref{E50a}),$ $(\ref{E28a}),$ and $(\ref{E50b}),$ by straightforward
     computations, an approximate formula of $(\ref{E51}),$ is given by
     \begin{equation*}
     \frac{Q_{j}^{\overline{n+1}}|Q_{j}^{\overline{n+1}}|}{(A_{j}^{\overline{n+1}})^{\frac{7}{3}}}=|Q^{n}||Q^{n}|^{-\frac{14}{3}}Q_{j}^{n}
     (\overline{Q_{j}^{n}})^{\frac{7}{3}}(\mu^{n})^{\frac{7}{3}}\left\{1+\left(\frac{\Delta t}{\Delta x}\right)^{2}
     (\mu^{n})^{2}\phi^{2}\right\}^{-\frac{7}{3}}
     \end{equation*}
     \begin{equation*}
     \left\{\left(1+\frac{\Delta t}{\Delta x}C_{11}\phi+2\Delta t C_{12}\right)^{2}+\left(\frac{\Delta t}{\Delta x}\overline{C}_{11}\phi
     +2\Delta t \overline{C}_{12}\right)^{2}\right\}^{\frac{1}{2}}
     \end{equation*}
     \begin{equation}\label{E51a}
     \left\{1+\frac{\Delta t}{\Delta x}C_{11}\phi+2\Delta t C_{12}+i\left(\frac{\Delta t}{\Delta x}\overline{C}_{11}\phi
     +2\Delta t \overline{C}_{12}\right)\right\}\left\{1+i\frac{\Delta t}{\Delta x}\mu^{n}\phi\right\}^{\frac{7}{3}}+O(\phi^{2}).
     \end{equation}
     Since $Q_{j}^{n}=e^{(b_{1}+ia_{2})t^{n}}e^{ikx_{j}}=e^{b_{1}t^{n}}e^{i(a_{2}t^{n}+kx_{j})}=e^{b_{1}t^{n}}e^{i\alpha_{3}},$ and the
     conjugate of $Q_{j}^{n}$ is given by $\overline{Q_{j}^{n}}=e^{b_{1}t^{n}}e^{-i\alpha_{3}}=|Q_{j}^{n}|e^{-i\alpha_{3}}.$ This
     fact together with equation $(\ref{E51a})$ yield
     \begin{equation*}
     \frac{Q_{j}^{\overline{n+1}}|Q_{j}^{\overline{n+1}}|}{(A_{j}^{\overline{n+1}})^{\frac{7}{3}}}=Q^{n}|Q^{n}|^{-\frac{4}{3}}
     (\mu^{n})^{\frac{7}{3}}\left\{1+\left(\frac{\Delta t}{\Delta x}\right)^{2}(\mu^{n})^{2}\phi^{2}\right\}^{-\frac{7}{3}}
     \left\{\left(1+\frac{\Delta t}{\Delta x}C_{11}\phi+2\Delta t C_{12}\right)^{2}+\right.
     \end{equation*}
     \begin{equation*}
        \left.\left(\frac{\Delta t}{\Delta x}\overline{C}_{11}\phi+2\Delta t \overline{C}_{12}\right)^{2}\right\}^{\frac{1}{2}}
     \left\{1+\frac{\Delta t}{\Delta x}C_{11}\phi+2\Delta t C_{12}+i\left(\frac{\Delta t}{\Delta x}\overline{C}_{11}\phi
     +2\Delta t \overline{C}_{12}\right)\right\}\times
     \end{equation*}
     \begin{equation*}
     \left\{C_{31}+\frac{\Delta t}{\Delta x}\overline{C}_{32}\phi+i\left[-\overline{C}_{31}+\frac{\Delta t}{\Delta x}C_{32}
     \phi\right]\right\}^{\frac{7}{3}}+O(\phi^{2}).
     \end{equation*}
     where $C_{3l}$ and $\overline{C}_{3l}$, $l=1,2,$ are defined in relation $(\ref{E52aa}).$ This completes the proof of Lemma $\ref{l3}.$
     \end{proof}

    Armed with Lemmas $\ref{l1},$ $\ref{l2}$ and $\ref{l3},$ we are ready to describe the amplification factor of the numerical
    scheme $(\ref{67}).$

   \begin{lemma}\label{l4}
   The amplification factor of the numerical scheme $(\ref{67})$ is approximated by
   \begin{equation*}
    e^{b\Delta t}=1+\Delta t\left\{\frac{1}{2}C_{12}+C_{33}-\Gamma_{0}|Q_{j}^{n}|^{-\frac{4}{3}}(\mu^{n})^{\frac{7}{3}}\left[1+
    \left(\frac{\Delta t}{\Delta x}\right)^{2}(\mu^{n})^{2}\phi^{2}\right]^{-\frac{7}{3}}\left[\left(1+2\Delta tC_{12}+\frac{\Delta t}
    {\Delta x}C_{11}\phi\right)^{2}+\right.\right.
   \end{equation*}
   \begin{equation*}
   \left. \left.\left(2\Delta t\overline{C}_{12}+\frac{\Delta t}{\Delta x}\overline{C}_{11}\phi\right)^{2}\right]
    \left[\left(C_{31}+\frac{\Delta t}{\Delta x}\overline{C}_{32}\phi\right)^{2}+\left(\overline{C}_{31}-
    \frac{\Delta t}{\Delta x}C_{32}\phi\right)^{2}\right]^{\frac{7}{6}}\cos\left(\theta_{1}+\frac{7}{3}\theta_{2}\right)\right\}-
   \end{equation*}
   \begin{equation*}
    \frac{\Delta t}{\Delta x}\left\{(C_{21}-\frac{1}{2}C_{11})\phi+\frac{1}{2}\mu^{n}\left[H_{1}-K_{1}^{2}+K^{2}_{2}+2K_{1}K_{2}\left(1-
    \frac{\Delta t}{\Delta x}\mu^{n}\right)\phi\right]\right\}+
   \end{equation*}
    \begin{equation*}
    i\left\{ \Delta t\left\{\frac{1}{2}\overline{C}_{12}-\overline{C}_{33}-\Gamma_{0}|Q_{j}^{n}|^{-\frac{4}{3}}(\mu^{n})^{\frac{7}{3}}\left[1+
    \left(\frac{\Delta t}{\Delta x}\right)^{2}(\mu^{n})^{2}\phi^{2}\right]^{-\frac{7}{3}}\left[\left(1+2\Delta tC_{12}+\frac{\Delta t}
    {\Delta x}C_{11}\phi\right)^{2}+\right.\right.\right.
   \end{equation*}
   \begin{equation*}
   \left. \left.\left(2\Delta t\overline{C}_{12}+\frac{\Delta t}{\Delta x}\overline{C}_{11}\phi\right)^{2}\right]
    \left[\left(C_{31}+\frac{\Delta t}{\Delta x}\overline{C}_{32}\phi\right)^{2}+\left(\overline{C}_{31}-
    \frac{\Delta t}{\Delta x}C_{32}\phi\right)^{2}\right]^{\frac{7}{6}}\sin\left(\theta_{1}+\frac{7}{3}\theta_{2}\right)\right\}-
   \end{equation*}
   \begin{equation}\label{E48a}
    \left.\frac{\Delta t}{\Delta x}\left\{(\overline{C}_{21}-\frac{1}{2}\overline{C}_{11})\phi+\frac{1}{2}\mu^{n}\left[H_{2}-2K_{1}K_{2}
    +(K^{2}_{2}-K_{1}^{2})\left(1-\frac{\Delta t}{\Delta x}\mu^{n}\right)\phi\right]\right\}\right\}+O(\phi^{2}).
   \end{equation}
   where
   \begin{equation}\label{E48b}
    C_{33}=\frac{P\overline{\tau}}{\rho}|Q^{n}_{j}|^{-1}\cos\alpha_{3},\text{\,\,\,\,}
    \overline{C}_{33}=\frac{P\overline{\tau}}{\rho}|Q^{n}_{j}|^{-1}\sin\alpha_{3},
   \end{equation}
   $H_{1}$, $H_{2}$, $K_{1}^{2}$, $K_{2}^{2}$ and $K_{1}K_{2}$ are given by equations $(\ref{E30})$, $(\ref{E30a})$ $(\ref{E41c})$,
     $(\ref{E41d})$ and $(\ref{E41e})$, respectively; $C_{lj},$ $\overline{C}_{lj},$ $j,l=1,2;$ come from $(\ref{E22})$; $\alpha_{2}$
     and $\alpha_{3}$ follow from equation $(\ref{E23})$ and $C_{3l},$ $\overline{C}_{3l},$ $l=1,2,$ are defined by relation
      $(\ref{E52aa}).$ The functions $\theta_{1}$ and $\theta_{2}$ are given implicitly by relations
      \begin{equation*}
      e^{i\theta_{1}}=\frac{1+2\Delta tC_{12}+\frac{\Delta t}{\Delta x}C_{11}\phi+i\left(2\Delta t\overline{C}_{12}
      +\frac{\Delta t}{\Delta x}\overline{C}_{11}\phi\right)}{\sqrt{\left(1+2\Delta tC_{12}+\frac{\Delta t}{\Delta x}
      C_{11}\phi\right)^{2}+\left(2\Delta t\overline{C}_{12}+\frac{\Delta t}{\Delta x}\overline{C}_{11}\phi\right)^{2}}}
      \end{equation*}
      and
      \begin{equation*}
      e^{i\theta_{2}}=\frac{C_{31}+\frac{\Delta t}{\Delta x}\overline{C}_{32}\phi+i\left(-\overline{C}_{31}
      +\frac{\Delta t}{\Delta x}C_{32}\phi\right)}{\sqrt{\left(C_{31}+\frac{\Delta t}{\Delta x}\overline{C}_{32}\phi\right)^{2}
      +\left(\overline{C}_{31}-\frac{\Delta t}{\Delta x}C_{32}\phi\right)^{2}}}
      \end{equation*}
   \end{lemma}

   \begin{proof}
   The proof is obvious according to Lemmas $\ref{l1},$ $\ref{l2}$ and $\ref{l3}.$
   \end{proof}

   Now, let us prove Proposition $\ref{p3}.$
   \begin{proof}
   (of Proposition $\ref{p3}$)\\
    Considering relation $(\ref{E48a})$ and applying Remark $\ref{r1},$ the squared modulus of the amplification factor
    of the numerical scheme $(\ref{67})$ is approximated as
    \begin{equation*}
    |e^{b\Delta t}|^{2}=\left\{1+\Delta t\left\{\frac{1}{2}C_{12}+C_{33}-\Gamma_{0}|Q_{j}^{n}|^{-\frac{4}{3}}(\mu^{n})^{\frac{7}{3}}\left[1+
    \left(\frac{\Delta t}{\Delta x}\right)^{2}(\mu^{n})^{2}\phi^{2}\right]^{-\frac{7}{3}}\left[\left(1+2\Delta tC_{12}+\frac{\Delta t}
    {\Delta x}C_{11}\phi\right)^{2}+\right.\right.\right.
   \end{equation*}
   \begin{equation*}
   \left. \left.\left(2\Delta t\overline{C}_{12}+\frac{\Delta t}{\Delta x}\overline{C}_{11}\phi\right)^{2}\right]
    \left[\left(C_{31}+\frac{\Delta t}{\Delta x}\overline{C}_{32}\phi\right)^{2}+\left(\overline{C}_{31}-
    \frac{\Delta t}{\Delta x}C_{32}\phi\right)^{2}\right]^{\frac{7}{6}}\cos\left(\theta_{1}+\frac{7}{3}\theta_{2}\right)\right\}-
   \end{equation*}
   \begin{equation*}
    \left.\frac{\Delta t}{\Delta x}\left\{(C_{21}-\frac{1}{2}C_{11})\phi+\frac{1}{2}\mu^{n}\left[H_{1}-K_{1}^{2}+K^{2}_{2}+2K_{1}K_{2}\left(1-
    \frac{\Delta t}{\Delta x}\mu^{n}\right)\phi\right]\right\}\right\}^{2}+
   \end{equation*}
    \begin{equation*}
    \left\{ \Delta t\left\{\frac{1}{2}\overline{C}_{12}-\overline{C}_{33}-\Gamma_{0}|Q_{j}^{n}|^{-\frac{4}{3}}(\mu^{n})^{\frac{7}{3}}\left[1+
    \left(\frac{\Delta t}{\Delta x}\right)^{2}(\mu^{n})^{2}\phi^{2}\right]^{-\frac{7}{3}}\left[\left(1+2\Delta tC_{12}+\frac{\Delta t}
    {\Delta x}C_{11}\phi\right)^{2}+\right.\right.\right.
   \end{equation*}
   \begin{equation*}
   \left. \left.\left(2\Delta t\overline{C}_{12}+\frac{\Delta t}{\Delta x}\overline{C}_{11}\phi\right)^{2}\right]
    \left[\left(C_{31}+\frac{\Delta t}{\Delta x}\overline{C}_{32}\phi\right)^{2}+\left(\overline{C}_{31}-
    \frac{\Delta t}{\Delta x}C_{32}\phi\right)^{2}\right]^{\frac{7}{6}}\sin\left(\theta_{1}+\frac{7}{3}\theta_{2}\right)\right\}-
   \end{equation*}
   \begin{equation}\label{E54}
    \left.\frac{\Delta t}{\Delta x}\left\{(\overline{C}_{21}-\frac{1}{2}\overline{C}_{11})\phi+\frac{1}{2}\mu^{n}\left[H_{2}-2K_{1}K_{2}
    +(K^{2}_{2}-K_{1}^{2})\left(1-\frac{\Delta t}{\Delta x}\mu^{n}\right)\phi\right]\right\}\right\}^{2}+O(\phi^{4}).
   \end{equation}
   Of course, the aim of this paper is to give the general picture of a necessary stability condition. Since we are working under
   the assumptions $0<\Delta t<\Delta x<1,$ $0<|k\Delta x|=|\phi|<<\pi,$ and the notations can become quite heavy, for the sake of
   readability we neglect the $O(\phi^{4}).$ Furthermore, applying Remark $\ref{r1}$, $\left[1+\left(\frac{\Delta t}{\Delta x}
   \right)^{2}(\mu^{n})^{2}\phi^{2}\right]^{-\frac{7}{3}}=1+O(\phi^{2}),$ neglecting the term $O(\phi^{2})$, $\left[1+\left(\frac{\Delta t}
   {\Delta x}\right)^{2}(\mu^{n})^{2}\phi^{2}\right]^{-\frac{7}{3}}$ is approximated by $1.$ However, the tracking of the infinitesimal
   terms does not compromise the result. Using this, estimate $|e^{b\Delta t}|\leq1,$ means that there exists a nonnegative number
   $r^{*}$ between $0$ and $1$ such that
      \begin{equation*}
     \left|1+\Delta t\left\{\frac{1}{2}C_{12}+C_{33}-\Gamma_{0}|Q_{j}^{n}|^{-\frac{4}{3}}(\mu^{n})^{\frac{7}{3}}
     \left[\left(1+2\Delta tC_{12}+\frac{\Delta t}{\Delta x}C_{11}\phi\right)^{2}+\left(2\Delta t\overline{C}_{12}+\frac{\Delta t}
     {\Delta x}\overline{C}_{11}\phi\right)^{2}\right]\right.\right.
   \end{equation*}
   \begin{equation*}
    \left.\left[\left(C_{31}+\frac{\Delta t}{\Delta x}\overline{C}_{32}\phi\right)^{2}+\left(\overline{C}_{31}-
    \frac{\Delta t}{\Delta x}C_{32}\phi\right)^{2}\right]^{\frac{7}{6}}\cos\left(\theta_{1}+\frac{7}{3}\theta_{2}\right)\right\}-
   \end{equation*}
   \begin{equation}\label{E55}
    \left.\frac{\Delta t}{\Delta x}\left\{(C_{21}-\frac{1}{2}C_{11})\phi+\frac{1}{2}\mu^{n}\left[H_{1}-K_{1}^{2}+K^{2}_{2}+2K_{1}K_{2}\left(1-
    \frac{\Delta t}{\Delta x}\mu^{n}\right)\phi\right]\right\}\right|\leq \sqrt{1-r^{*}}
   \end{equation}
   and
    \begin{equation*}
    \left|\Delta t\left\{\frac{1}{2}\overline{C}_{12}-\overline{C}_{33}-\Gamma_{0}|Q_{j}^{n}|^{-\frac{4}{3}}(\mu^{n})^{\frac{7}{3}}
    \left[\left(1+2\Delta tC_{12}+\frac{\Delta t}{\Delta x}C_{11}\phi\right)^{2}+\left(2\Delta t\overline{C}_{12}+\frac{\Delta t}
    {\Delta x}\overline{C}_{11}\phi\right)^{2}\right]\right.\right.
   \end{equation*}
   \begin{equation*}
   \left.\left[\left(C_{31}+\frac{\Delta t}{\Delta x}\overline{C}_{32}\phi\right)^{2}+\left(\overline{C}_{31}-
    \frac{\Delta t}{\Delta x}C_{32}\phi\right)^{2}\right]^{\frac{7}{6}}\sin\left(\theta_{1}+\frac{7}{3}\theta_{2}\right)\right\}-
   \end{equation*}
   \begin{equation}\label{E56}
    \left.\frac{\Delta t}{\Delta x}\left\{(\overline{C}_{21}-\frac{1}{2}\overline{C}_{11})\phi+\frac{1}{2}\mu^{n}\left[H_{2}-2K_{1}K_{2}
    +(K^{2}_{2}-K_{1}^{2})\left(1-\frac{\Delta t}{\Delta x}\mu^{n}\right)\phi\right]\right\}\right|\leq \sqrt{r^{*}},
   \end{equation}
   which are equivalent to
   \begin{equation*}
    1-\sqrt{1-r^{*}}\leq -\Delta t\left\{\frac{1}{2}C_{12}+C_{33}-\Gamma_{0}|Q_{j}^{n}|^{-\frac{4}{3}}(\mu^{n})^{\frac{7}{3}}
    \left[\left(1+2\Delta tC_{12}+\frac{\Delta t}{\Delta x}C_{11}\phi\right)^{2}+\left(2\Delta t\overline{C}_{12}
    +\frac{\Delta t}{\Delta x}\overline{C}_{11}\phi\right)^{2}\right]\right.
   \end{equation*}
   \begin{equation*}
   \left. \left[\left(C_{31}+\frac{\Delta t}{\Delta x}\overline{C}_{32}\phi\right)^{2}+\left(\overline{C}_{31}-
    \frac{\Delta t}{\Delta x}C_{32}\phi\right)^{2}\right]^{\frac{7}{6}}\cos\left(\theta_{1}+\frac{7}{3}\theta_{2}\right)\right\}+
   \end{equation*}
   \begin{equation}\label{E55a}
    \frac{\Delta t}{\Delta x}\left\{(C_{21}-\frac{1}{2}C_{11})\phi+\frac{1}{2}\mu^{n}\left[H_{1}-K_{1}^{2}+K^{2}_{2}+2K_{1}K_{2}\left(1-
    \frac{\Delta t}{\Delta x}\mu^{n}\right)\phi\right]\right\}\leq 1+\sqrt{1-r^{*}},
   \end{equation}
   and
    \begin{equation*}
    \left|\Delta t\left\{\frac{1}{2}\overline{C}_{12}-\overline{C}_{33}-\Gamma_{0}|Q_{j}^{n}|^{-\frac{4}{3}}(\mu^{n})^{\frac{7}{3}}
    \left[\left(1+2\Delta tC_{12}+\frac{\Delta t}{\Delta x}C_{11}\phi\right)^{2}+\left(2\Delta t\overline{C}_{12}+\frac{\Delta t}
    {\Delta x}\overline{C}_{11}\phi\right)^{2}\right]\right.\right.
   \end{equation*}
   \begin{equation*}
   \left.\left[\left(C_{31}+\frac{\Delta t}{\Delta x}\overline{C}_{32}\phi\right)^{2}+\left(\overline{C}_{31}-
    \frac{\Delta t}{\Delta x}C_{32}\phi\right)^{2}\right]^{\frac{7}{6}}\sin\left(\theta_{1}+\frac{7}{3}\theta_{2}\right)\right\}-
   \end{equation*}
   \begin{equation}\label{E56a}
    \left.\frac{\Delta t}{\Delta x}\left\{(\overline{C}_{21}-\frac{1}{2}\overline{C}_{11})\phi+\frac{1}{2}\mu^{n}\left[H_{2}-2K_{1}K_{2}
    +(K^{2}_{2}-K_{1}^{2})\left(1-\frac{\Delta t}{\Delta x}\mu^{n}\right)\phi\right]\right\}\right|\leq \sqrt{r^{*}}.
   \end{equation}
   Since we are interested in a linear stability condition, we should find a restriction satisfies by $\Delta t$ and $\Delta x$ for which
   inequalities $(\ref{E55a})$ and $(\ref{E56a})$ hold. First, we must bound each term in estimates $(\ref{E55a})$ and $(\ref{E56a}).$
   Using relations $(\ref{E9a}),$ $(\ref{E22}),$ $(\ref{E52aa}),$ $(\ref{E48b})$ and $(\ref{E22a}),$ we get
   \begin{equation*}
    \gamma_{4}=(\mu^{n})^{2};\text{\,\,\,}|C_{11}|,|\overline{C}_{11}|\leq\frac{g}{T}|A_{j}^{n}||\mu^{n}|^{-1}+\mu^{n}:=R_{n};
    \text{\,\,}|C_{12}|,|\overline{C}_{12}|\leq\frac{P\overline{\tau}}{\rho}|Q_{j}^{n}|^{-1}+\Gamma_{0}\mu^{n}|A_{j}^{n}|^{-\frac{4}{3}}:=P_{n};
   \end{equation*}
   \begin{equation*}
    |C_{31}|,|\overline{C}_{31}|\leq1;\text{\,\,\,}|C_{32}|,|\overline{C}_{32}|\leq\mu^{n};\text{\,\,\,}|C_{33}|,|\overline{C}_{33}|
    \leq\frac{P\overline{\tau}}{\rho}|Q_{j}^{n}|^{-1}\leq P_{n};
   \end{equation*}
   \begin{equation}\label{E57}
    |C_{21}|,|\overline{C}_{21}|\leq\frac{g}{2T}|Q_{j}^{n}||\mu^{n}|^{-1}:=N_{n};\text{\,\,\,}|C_{22}|,|\overline{C}_{22}|\leq
    \frac{4}{3}\Gamma_{0}\mu^{n}|A_{j}^{n}|^{-\frac{4}{3}}\leq 2P_{n}.
   \end{equation}
    From this and by simple calculations, it comes from equations $(\ref{E30}),$ $(\ref{E30a}),$ $(\ref{E41c}),$ $(\ref{E41d})$
    and $(\ref{E41e})$ that the quantities $H_{1}$, $H_{2}$, $K_{1}^{2}$, $K_{2}^{2}$ and $K_{1}K_{2}$ can be bounded as
    \begin{equation}\label{E60}
    H_{1}\leq 1+2[1+(2+R_{n}^{-1}\mu^{n})|\phi|]\max\left\{P_{n},P_{n}R_{n},P_{n}^{2},P_{n}^{2}\mu^{n}\right\}\left[\Delta t+
    \Delta t^{2}+\frac{\Delta t^{2}}{\Delta x}+\frac{\Delta t^{3}}{\Delta x}\right];
   \end{equation}
   \begin{equation}\label{E61}
    H_{2}\leq2\max\left\{P_{n},P_{n}R_{n},P_{n}^{2},P_{n}(R_{n}+\mu^{n})|\phi|,\mu^{n}P_{n}^{2}|\phi|\right\}\left[\Delta t+
    \Delta t^{2}+\frac{\Delta t^{2}}{\Delta x}+\frac{\Delta t^{3}}{\Delta x}\right];
   \end{equation}
    \begin{equation}\label{E62}
    K_{1}^{2}\leq 1+(1+4|\phi|)\max\left\{P_{n},N_{n},P_{n}N_{n},P_{n}^{2},N_{n}^{2}\right\}\left[\Delta t+
    \Delta t^{2}+\frac{\Delta t}{\Delta x}+\frac{\Delta t^{2}}{\Delta x}+\left(\frac{\Delta t}{\Delta x}\right)^{2}\right];
   \end{equation}
   \begin{equation*}
    K_{2}^{2}\leq [1+(4+R_{n}^{-1}+R_{n}^{-1}(\mu^{n})^{-1})|\phi|]\max\left\{P_{n},P_{n}R_{n}\mu^{n},P_{n}^{2},R_{n}^{2}
     \mu^{n},R_{n}^{2}(\mu^{n})^{2}\right\}
   \end{equation*}
   \begin{equation}\label{E63}
    \left[\Delta t+\Delta t^{2}+\frac{\Delta t}{\Delta x}+\frac{\Delta t^{2}}{\Delta x}+\left(\frac{\Delta t}{\Delta x}
    \right)^{2}\right];
   \end{equation}
   \begin{equation*}
    |K_{1}K_{2}|\leq|\phi|+[1+(4+N_{n}R_{n}^{-1}+(4+R_{n}^{-1}+2P_{n}+R_{n}^{-1}\mu^{n}+4N_{n}R_{n}^{-1}\mu^{n})^{-1})|\phi|]
   \end{equation*}
   \begin{equation}\label{E64}
    \max\left\{P_{n},P_{n}R_{n}\mu^{n},P_{n}^{2},R_{n}\mu^{n},R_{n}N_{n}\mu^{n}\right\}\left[\Delta t+\Delta t^{2}
    +\frac{\Delta t}{\Delta x}+\frac{\Delta t^{2}}{\Delta x}+\left(\frac{\Delta t}{\Delta x}\right)^{2}\right];
   \end{equation}
   A combination of inequalities $(\ref{E60}),$ $(\ref{E61}),$ $(\ref{E62}),$ $(\ref{E63})$ and $(\ref{E64}),$ results in
   \begin{equation*}
    |H_{1}|+K_{1}^{2}+K_{2}^{2}+2|K_{1}K_{2}|\left|1-\frac{\Delta t}{\Delta x}\mu^{n}\right||\phi|\leq2+8\left[\Delta t+\Delta t^{2}
    +\frac{\Delta t}{\Delta x}+\frac{\Delta t^{2}}{\Delta x}+\frac{\Delta t^{3}}{\Delta x}+\left(\frac{\Delta t}{\Delta x}\right)^{2}\right]
   \end{equation*}
   \begin{equation*}
     \left[1+(4+R_{n}(\mu^{n})^{-1}+R_{n}^{-1}+N_{n}R_{n}(\mu^{n})^{-1})\left(1+\frac{\Delta t}{\Delta x}\mu^{n}\right)|\phi|\right]
   \end{equation*}
   \begin{equation}\label{E65}
   \max\left\{P_{n},N_{n},R_{n}P_{n},N_{n}P_{n},R_{n}\mu^{n},P_{n}^{2},N_{n}^{2},P_{n}^{2}\mu^{n},R_{n}^{2}\mu^{n},
   R_{n}^{2}(\mu^{n})^{2},R_{n}P_{n}\mu^{n},R_{n}N_{n}\mu^{n}\right\};
   \end{equation}
   \begin{equation*}
    |H_{2}|+2|K_{1}K_{2}|+(K_{1}^{2}+K_{2}^{2})\left|1-\frac{\Delta t}{\Delta x}\mu^{n}\right||\phi|\leq \left(3+\frac{\Delta t}
    {\Delta x}\mu^{n}\right)|\phi|+6\left[\Delta t+\Delta t^{2}+\frac{\Delta t}{\Delta x}+\frac{\Delta t^{2}}{\Delta x}
    +\frac{\Delta t^{3}}{\Delta x}\right.
   \end{equation*}
   \begin{equation*}
    \left.+\left(\frac{\Delta t}{\Delta x}\right)^{2}\right]\left[1+N_{n}R_{n}(\mu^{n})^{-1}+\left(4+2P^{n}+R_{n}^{-1}+R_{n}^{-1}\mu^{n}+
    4N_{n}R_{n}^{-1}(\mu^{n})^{-1}+\frac{\Delta t}{\Delta x}\mu^{n}\right)|\phi|\right]
   \end{equation*}
   \begin{equation}\label{E66}
   \max\left\{P_{n},N_{n},N_{n}P_{n},R_{n}\mu^{n},P_{n}^{2},N_{n}^{2},R_{n}^{2}\mu^{n},P_{n}R_{n}\mu^{n},R_{n}^{2}(\mu^{n})^{2},
   R_{n}N_{n}\mu^{n}\right\}.
   \end{equation}
   But
   \begin{equation}\label{E67}
   \Gamma_{0}|Q^{n}|^{-\frac{4}{3}}(\mu^{n})^{\frac{7}{3}}=\Gamma_{0}\mu^{n}|A^{n}|^{-\frac{4}{3}}\leq P_{n};
   \end{equation}
   applying Remark $\ref{r1},$ and neglecting the terms of high order in $\phi$, it is easy to see that
   \begin{equation}\label{E68}
    \left(1+2\Delta tC_{12}+\frac{\Delta t}{\Delta x}C_{11}\phi\right)^{2}+\left(2\Delta t\overline{C}_{12}+\frac{\Delta t}
    {\Delta x}\overline{C}_{11}\phi\right)^{2}\leq1+4\Delta tP_{n}+8\Delta t^{2}P_{n}^{2}+2\frac{\Delta t}
    {\Delta x}R_{n}|\phi|+8\frac{\Delta t^{2}}{\Delta x}P_{n}R_{n}|\phi|;
   \end{equation}
   \begin{equation*}
    \left[\left(C_{31}+\frac{\Delta t}{\Delta x}\overline{C}_{32}\phi\right)^{2}+\left(\overline{C}_{31}-\frac{\Delta t}{\Delta x}
    C_{32}\phi\right)^{2}\right]^{\frac{7}{6}}\leq\left[2\left(1+\frac{\Delta t}{\Delta x}\mu^{n}|\phi|\right)^{2}\right]^{\frac{7}{6}}
    \leq 2\sqrt[6]{2}\left(1+\frac{\Delta t}{\Delta x}\mu^{n}|\phi|\right)^{\frac{7}{3}}\leq
   \end{equation*}
   \begin{equation}\label{E69}
     3\left(1+3\frac{\Delta t}{\Delta x}\mu^{n}|\phi|\right).
   \end{equation}
   We recall that the aim of this work is to find a linear (or necessary) stability condition of the numerical scheme $(\ref{67}).$
   Plugging estimates $(\ref{E65})$-$(\ref{E69}),$ by straightforward computations, a necessary condition to obtain inequalities
   $(\ref{E55a})$ and $(\ref{E56a}),$ is given by
   \begin{equation*}
    3P_{n}\Delta t\left\{\frac{1}{2}+\left[1+4\left(\Delta t+\Delta t^{2}+\frac{\Delta t}{\Delta x}+\frac{\Delta t^{2}}{\Delta x}
    +\frac{\Delta t^{2}}{\Delta x}+\left(+\frac{\Delta t^{2}}{\Delta x}\right)^{2}\right)
    \max\{P_{n},2P_{n}^{2},\frac{1}{2}R_{n}|\phi|,\mu^{n}|\phi|,P_{n}\mu^{n}|\phi|,\right.\right.
   \end{equation*}
   \begin{equation*}
    \left.\left.6P_{n}^{2}\mu^{n}|\phi|,2R_{n}P_{n}|\phi|\}\right]\right\}+\frac{\Delta t}{\Delta x}\left\{(P_{n}+\frac{1}{2}R_{n})|\phi|
    +\mu^{n}\left\{1+4\left[\Delta t+\Delta t^{2}+\frac{\Delta t}{\Delta x}+\frac{\Delta t^{2}}{\Delta x}
    +\frac{\Delta t^{3}}{\Delta x}+\left(\frac{\Delta t}{\Delta x}\right)^{2}\right]\right.\right.
   \end{equation*}
   \begin{equation*}
     \left[1+(4+R_{n}(\mu^{n})^{-1}+R_{n}^{-1}+N_{n}R_{n}(\mu^{n})^{-1})\left(1+\frac{\Delta t}{\Delta x}\mu^{n}\right)|\phi|\right]
   \end{equation*}
   \begin{equation}\label{E70}
   \left.\left.\max\left\{P_{n},N_{n},R_{n}P_{n},N_{n}P_{n},R_{n}\mu^{n},P_{n}^{2},N_{n}^{2},P_{n}^{2}\mu^{n},R_{n}^{2}\mu^{n},
   R_{n}^{2}(\mu^{n})^{2},R_{n}P_{n}\mu^{n},R_{n}N_{n}\mu^{n}\right\}\right\}\right\}\leq1+\sqrt{1-r^{*}};
   \end{equation}
   and
     \begin{equation*}
    3P_{n}\Delta t\left\{\frac{1}{2}+\left[1+4\left(\Delta t+\Delta t^{2}+\frac{\Delta t}{\Delta x}+\frac{\Delta t^{2}}{\Delta x}
    +\frac{\Delta t^{2}}{\Delta x}+\left(+\frac{\Delta t^{2}}{\Delta x}\right)^{2}\right)
    \max\{P_{n},2P_{n}^{2},\frac{1}{2}R_{n}|\phi|,\mu^{n}|\phi|,P_{n}\mu^{n}|\phi|,\right.\right.
   \end{equation*}
   \begin{equation*}
    \left.\left.6P_{n}^{2}\mu^{n}|\phi|,2R_{n}P_{n}|\phi|\}\right]\right\}+\frac{\Delta t}{\Delta x}\left\{(P_{n}+\frac{1}{2}R_{n})|\phi|
    +\frac{3\mu^{n}}{2}\left\{|\phi|+2\left[\Delta t+\Delta t^{2}+\frac{\Delta t}{\Delta x}+\frac{\Delta t^{2}}{\Delta x}
    +\frac{\Delta t^{3}}{\Delta x}\right.\right.\right.
   \end{equation*}
   \begin{equation*}
    \left.+\left(\frac{\Delta t}{\Delta x}\right)^{2}\right]\left[1+N_{n}R_{n}(\mu^{n})^{-1}+\left(4+2P^{n}+R_{n}^{-1}+R_{n}^{-1}\mu^{n}+
    4N_{n}R_{n}^{-1}(\mu^{n})^{-1}+\frac{\Delta t}{\Delta x}\mu^{n}\right)|\phi|\right]
   \end{equation*}
   \begin{equation}\label{E71}
   \left.\left.\max\left\{P_{n},N_{n},N_{n}P_{n},R_{n}\mu^{n},P_{n}^{2},N_{n}^{2},R_{n}^{2}\mu^{n},P_{n}R_{n}\mu^{n},R_{n}^{2}(\mu^{n})^{2},
   R_{n}N_{n}\mu^{n}\right\}\right\}\right\}\leq \sqrt{r^{*}}.
   \end{equation}
   Now, setting
   \begin{equation*}
    W_{1}(\Delta t,\Delta x)=\frac{1}{2}+\left[1+4\left(\Delta t+\Delta t^{2}+\frac{\Delta t}{\Delta x}+\frac{\Delta t^{2}}{\Delta x}
    +\frac{\Delta t^{2}}{\Delta x}+\left(+\frac{\Delta t^{2}}{\Delta x}\right)^{2}\right)
    \max\{P_{n},2P_{n}^{2},\frac{1}{2}R_{n}|\phi|,\mu^{n}|\phi|,\right.
   \end{equation*}
   \begin{equation}\label{E72}
    \left.P_{n}\mu^{n}|\phi|,6P_{n}^{2}\mu^{n}|\phi|,2R_{n}P_{n}|\phi|\}\right];
    \end{equation}
   \begin{equation*}
    W_{2}(\Delta t,\Delta x)=(P_{n}+\frac{1}{2}R_{n})|\phi|+\mu^{n}\left\{1+4\left[\Delta t+\Delta t^{2}+\frac{\Delta t}{\Delta x}
    +\frac{\Delta t^{2}}{\Delta x}+\frac{\Delta t^{3}}{\Delta x}+\left(\frac{\Delta t}{\Delta x}\right)^{2}\right]\right.
   \end{equation*}
   \begin{equation*}
    \left[1+(4+R_{n}(\mu^{n})^{-1}+R_{n}^{-1}+N_{n}R_{n}(\mu^{n})^{-1})\left(1+\frac{\Delta t}{\Delta x}\mu^{n}\right)|\phi|\right]
   \end{equation*}
   \begin{equation}\label{E73}
   \left.\max\left\{P_{n},N_{n},R_{n}P_{n},N_{n}P_{n},R_{n}\mu^{n},P_{n}^{2},N_{n}^{2},P_{n}^{2}\mu^{n},R_{n}^{2}\mu^{n},
   R_{n}^{2}(\mu^{n})^{2},R_{n}P_{n}\mu^{n},R_{n}N_{n}\mu^{n}\right\}\right\};
   \end{equation}
    and
   \begin{equation*}
    W_{3}(\Delta t,\Delta x)=(P_{n}+\frac{1}{2}R_{n})|\phi|+\frac{3}{2}\mu^{n}\left\{|\phi|+2\left[\Delta t+\Delta t^{2}+\frac{\Delta t}
    {\Delta x}+\frac{\Delta t^{2}}{\Delta x}+\frac{\Delta t^{3}}{\Delta x}+\left(\frac{\Delta t}{\Delta x}\right)^{2}\right]\right.
   \end{equation*}
   \begin{equation*}
    \left[1+N_{n}R_{n}(\mu^{n})^{-1}+\left(4+2P^{n}+R_{n}^{-1}+R_{n}^{-1}\mu^{n}+4N_{n}R_{n}^{-1}(\mu^{n})^{-1}+\frac{\Delta t}
    {\Delta x}\mu^{n}\right)|\phi|\right]
   \end{equation*}
   \begin{equation}\label{E74}
   \left.\max\left\{P_{n},N_{n},N_{n}P_{n},R_{n}\mu^{n},P_{n}^{2},N_{n}^{2},R_{n}^{2}\mu^{n},P_{n}R_{n}\mu^{n},R_{n}^{2}(\mu^{n})^{2},
   R_{n}N_{n}\mu^{n}\right\}\right\}.
   \end{equation}
    This fact, together with inequalities $(\ref{E70})$ and $(\ref{E71})$ provide
   \begin{equation*}
    3P_{n}\Delta t W_{1}(\Delta t,\Delta x)+\frac{\Delta t}{\Delta x}W_{2}(\Delta t,\Delta x)\leq1+\sqrt{1-r^{*}}\text{\,\,\,and\,\,\,}
     3P_{n}\Delta tW_{1}(\Delta t,\Delta x)+\frac{\Delta t}{\Delta x}W_{3}(\Delta t,\Delta x)\leq \sqrt{r^{*}},
   \end{equation*}
   which are equivalent to
   \begin{equation}\label{E75}
    \Delta t\left(3P_{n}W_{1}(\Delta t,\Delta x)+\frac{1}{\Delta x}\max\left\{W_{2}(\Delta t,\Delta x);
    W_{3}(\Delta t,\Delta x)\right\}\right)\leq\max\left\{1+\sqrt{1-r^{*}};\sqrt{r^{*}}\right\},
   \end{equation}
    where $r^{*}\in(0;1),$ $W_{1}(\Delta t,\Delta x),$ $W_{2}(\Delta t,\Delta x)$ and $W_{3}(\Delta t,\Delta x),$  are given by
    relations $(\ref{E72})$, $(\ref{E73})$ and $(\ref{E74}),$ respectively. Estimate $(\ref{E75})$ comes from the inequality:
    $\max\{a+dx,a+dy\}=a+d\max\{x,y\},$ whenever the numbers $a$, $d$, $x$ and $y$ are nonnegative. Furthermore, Estimate $(\ref{E75})$
    represents a necessary condition of stability for the numerical scheme $(\ref{67}).$ This ends the proof of Proposition $\ref{p3}$.
    \end{proof}

    Using the above results (namely, Propositions $\ref{p2}$ and $\ref{p3}$) we are ready to give a necessary stability constraint
    of the MacCormack method $(\ref{64})$-$(\ref{67})$ and to compare it with what is available in the literature (for example,
    Courant-Friedrich-Lewy condition for linear hyperbolic partial differential equations).

    \begin{theorem}\label{t1}
   The MacCormack scheme for $1$D complete shallow water equations with source terms $(\ref{2})$ is stable if
    \begin{equation*}
    \frac{\Delta t^{4}}{\Delta x^{2}}\left(3P_{n}W_{1}(\Delta t,\Delta x)+\frac{1}{\Delta x}\max\left\{W_{2}(\Delta t,\Delta x);
    W_{3}(\Delta t,\Delta x)\right\}\right)\left(1+\frac{2\Delta t}{3}\Gamma_{0}\mu^{n}|A^{n}|^{-\frac{4}{3}}\right)\leq
    \end{equation*}
    \begin{equation}\label{101}
    3\max\left\{1+\sqrt{1-r^{*}};\sqrt{r^{*}}\right\}\Gamma_{0}^{-1}(\mu^{n})^{-3}|A^{n}|^{\frac{4}{3}}|\phi|^{-2},
    \end{equation}
   with the requirement: $|\phi|=|k\Delta x|<<\pi.$ In relations $(\ref{101}):$ $e^{a_{1}t}=|A|,$ $e^{b_{1}t}=|Q|,$
    $\mu=\left|\frac{Q}{A}\right|,$
   $r^{*}\in(0;1),$ $\Gamma_{0}=\frac{gn_{1}^{2}}{1.49^{2}}P^{\frac{4}{3}}$, $W_{1}(\Delta t,\Delta x),$ $W_{2}(\Delta t,\Delta x)$
   and $W_{3}(\Delta t,\Delta x),$ are given by relations $(\ref{E72})$, $(\ref{E73})$ and $(\ref{E74}),$ respectively.
   \end{theorem}

    \begin{proof}
    The proof of Theorem $\ref{t1}$ is obvious according to Propositions $\ref{p2}$ and $\ref{p3}.$
    \end{proof}

    The Von Neumann stability approach, based on a Fourier analysis in the space domain has been developed for nonlinear one-dimensional
    complete shallow water equations with source terms. Although the stability condition has not be derived analytically, we have analyzed
    the properties of amplification factor numerically (by use of Taylor series expansion), which contain information on the dispersion
    and diffusion errors of the considered numerical scheme. It is worth noticing that we used a local, linearized stability analysis to
    obtain estimate $(\ref{101}),$ which must be considered as a necessary condition of stability for the numerical scheme
    $(\ref{66})$-$(\ref{67}).$

    \subsection*{Some important remarks on stability analysis}
    This section considers some useful remarks on the stability restrictions obtained in this note and compares it with
    what is known in the literature (Courant-Friedrichs-Lewy condition).

   \begin{enumerate}
   \item The stability restrictions $(\ref{101})$ suggests that a small space step $\Delta x$ forces the time step
   $\Delta t$ to be more potentially small. This makes the MacCormack scheme extremely slow. However, because consistency requires that
   $\frac{\Delta t^{n}}{\Delta x}$ $(n\geq1)$ approached zero as $\Delta t$ and $\Delta x$ approach zero, a much smaller time step than
    allowed by the stability condition $(\ref{101})$ is implied. For this reason, the MacCormack method seems better suitable for the
    calculation of steady solutions (where time accuracy is unimportant) than for for the unsteady solutions (for example, see \cite{pv}
    for analysis of numerical solutions for time dependent PDEs).\\

   \item
   The MacCormack approach $(\ref{64})$-$(\ref{67})$ for $1$D complete surface water equations has a stability
   limitation $(\ref{101})$ that limits the maximum time step. This stability requirement does not coincide
   with the Courant-Friedrichs-Lewy (CFL) condition obtained for linear hyperbolic partial differential equations (for example:
   linear advection equation, wave equation, linearized burgers equations, etc...) because
   the MacCormack method is applied to complex time dependent partial differential equations. As discussion on the stability
   restrictions one can refer to the stability analysis of the two-step Lax-Wendroff method and the MacCormack scheme
   applied to complete burgers equations (for example, see \cite{3nnnn}, P. $245$-$247$). The linear stability condition $(\ref{101})$
   is highly unusual. Since we normally find this condition from a Fourier stability analysis, it follows from inequalities $(\ref{101})$
   that an instability occurs when $|\Delta t|$ is greater than some $|\Delta t|_{\max}$ which can be viewed as $(\Delta t)_{CFL}.$
   As observed in proving Proposition $\ref{p3},$ it was extremely difficult to obtain the stability criterion for our numerical method.
   However, it comes from condition given by relation $(\ref{101})$ that the empirical formula
   \begin{equation*}
    \frac{\Delta t^{4}}{\Delta x^{2}}\left(3|A^{n}|^{\frac{4}{3}}+2\Delta t\Gamma_{0}\mu^{n}\right)
    \leq9\max\left\{1+\sqrt{1-r^{*}};\sqrt{r^{*}}\right\}\Gamma_{0}^{-1}(\mu^{n})^{-3}|A^{n}|^{\frac{8}{3}}
    |\phi|^{-2},
   \end{equation*}
    can be used with an appropriate safety factor. It should be remembered that the "heuristic" stability analysis, i.e., estimates $(\ref{101})$
    can only provide a necessary condition for stability. Thus, for some finite difference algorithms, only partial information about the complete
    stability bound is obtained and for others (such as algorithms for the heat equation, wave equation and linearized Burgers equations)
    a more complete theory must be employed.
   \end{enumerate}

   \section{Numerical experiments}\label{ne}
   This section simulates the MacCormack scheme described in section \ref{sas} for $1$D complete shallow water equations
   with source terms. We focus on a practical application of a shallow water flow based on the Benou\'{e} river. This river
   is a $7000m$ long reach of the upstream part (altitude=$174.22$ m) and it is located in Cameroon. Being a mountain river,
   it is characterized by strong irregularities in the cross section, by a rather steep part in the first kilometers and by
   a low base discharge $(708m^{3}/s)$ which, altogether, produce a high velocity basic flow, transcritical in some parts.
   More specifically, we consider the problem of floods observed in this river in $2012$ because it is a classical example of
   time dependent nonlinear flow with shocks to expect floods and to test conservation in numerical schemes. Furthermore, we assume
   that this problem is generated by the $1$D complete shallow water equations with source terms for the \textbf{ideal case}
   of a flat and frictionless channel with prismatic cross section, i.e., constants top width ($T=348$ m) and wetted perimeter
   ($P=366,4$ m). We also use the initial data given by relation $(\ref{18f}).$\\

   Before describing the analytical solution considered in this note, we first approximate the $L^{2}$-norm of the space
   $L^{2}(0,T_{1};L^{2}(0,L))$ by a full discrete norm which plays a crucial role in the analysis of the error estimates together
   with the convergence rate of our method. Let $w\in L^{2}(0,T_{1};L^{2}(0,L)),$ we have that
            \begin{eqnarray*}
              \left[\int_{0}^{T_{1}}\int_{0}^{L}|w(t,x)|^{2}dxdt\right]^{1/2} &\approx&
               \left[\underset{n=0}{\overset{N}\sum}\int_{t^{n}}^{t^{n+1}}\int_{0}^{L}|w(t,x)|^{2}dxdt\right]^{1/2}  \\
                &\approx& \left[\Delta t\underset{n=0}{\overset{N}\sum}\int_{0}^{L}|w(t^{n},x)|^{2}dx\right]^{1/2} \\
                &\approx& \left[\Delta t\underset{n=0}{\overset{N}\sum}\underset{j=0}{\overset{M}\sum}
                \int_{x_{j}}^{x_{j+1}}|w(t^{n},x)|^{2}dx\right]^{1/2} \\
                &\approx& \left[\Delta t\cdot\Delta x\underset{n=0}{\overset{N}\sum}\underset{j=0}{\overset{M}\sum}
                |w(t^{n},x_{j})|^{2}\right]^{1/2}.
            \end{eqnarray*}
           Using this, we introduce the following fully discrete norm
            \begin{equation}\label{27l}
            \||w|\|_{L^{2}(0,T_{1};L^{2}(0,L))}=\left[\Delta t\cdot\Delta x\underset{n=0}{\overset{N}\sum}\underset{j=0}
             {\overset{M}\sum}|w(t^{n},x_{j})|^{2}\right]^{1/2}.
            \end{equation}
            Denoting by $w_{j}^{n}=(A_{j}^{n},Q_{j}^{n})$ the value of the approximate solution at time $t^{n}$ and point
             $x_{j}$ obtained with the MacCormack scheme and by $w(t^{n},x_{j})=\left(A(t^{n},x_{j}),Q(t^{n},x_{j})\right)$ the
             value of the analytical solution at $(t^{n},x_{j}),$ the exact error at time $t^{n}$ and point $x_{j}$ is defined by
            $e(t^{n},x_{j})=w(t^{n},x_{j})-w_{j}^{n}=\left(A(t^{n},x_{j})-A_{j}^{n},Q(t^{n},x_{j})-Q_{j}^{n}\right).$ Thus the
            errors-norm are given by
            \begin{equation}\label{28l}
            \||e_{A}|\|_{L^{2}(0,T_{1};L^{2}(0,L))}=\left[\Delta t\cdot\Delta x\underset{n=0}{\overset{N}\sum}\underset{j=0}
             {\overset{M}\sum}|A(t^{n},x_{j})-A_{j}^{n}|^{2}\right]^{1/2},
            \end{equation}
             and
            \begin{equation}\label{29l}
            \||e_{Q}|\|_{L^{2}(0,T_{1};L^{2}(0,L))}=\left[\Delta t\cdot\Delta x\underset{n=0}{\overset{N}\sum}\underset{j=0}
             {\overset{M}\sum}|Q(t^{n},x_{j})-Q_{j}^{n}|^{2}\right]^{1/2}.
            \end{equation}

   The exact solution considered in this paper is due to Dressler's dam break with friction \cite{swashes}. In the literature
   different approaches are presented and deeply studied for this case. Dressler's analyzed Ch\'{e}zy friction law and has used a
   perturbation scheme in the Ritter's method, i.e., both velocity $(u)$ and height $(h)$ of the water are expanded as power series in
   the friction coefficient $C_{f}=1/C^{2}.$ We consider the initial conditions defined as
        \begin{equation}\label{18f}
        h(0,x)=h^{0}(x)=\left\{
                          \begin{array}{ll}
                            h_{l}>0, & \hbox{for $0\leq x\leq x_{0}$;} \\
                            \text{\,}\\
                            0, & \hbox{for $x_{0}< x\leq L$,}
                          \end{array}
                        \right.\text{\,\,\,\,}u(0,x)=u^{0}(x)=\left\{
                          \begin{array}{ll}
                            10^{-1}, & \hbox{for $0\leq x\leq x_{0}$;} \\
                            \text{\,}\\
                            0, & \hbox{for $x_{0}< x\leq L$.}
                          \end{array}
                        \right.
        \end{equation}
        We assume that $C=40m^{1/2}/s$ (Ch\'{e}zy coefficient), $h_{l}=5\times 10^{-3}$m, $x_{0}=L/2,$ $T_{1}=1$s, and $L=1$m.
        Dressler's first order developments for the flow resistance give the following corrected height and velocity
        \begin{equation}\label{19f}
         \left\{
           \begin{array}{ll}
             h_{c}(t,x)=\frac{1}{g}\left(\frac{2}{3}\sqrt{gh_{l}}-\frac{x-x_{0}}{3t}+\frac{g^{2}}{C^{2}}\alpha_{1}t\right)^{2}, & \hbox{} \\
             \text{\,}\\
             u_{c}(t,x)=\frac{2}{3}\sqrt{gh_{l}}+\frac{2(x-x_{0})}{3t}+\frac{g^{2}}{C^{2}}\alpha_{2}t, & \hbox{}
           \end{array}
         \right.
        \end{equation}
       where
       \begin{equation}\label{19af}
        \alpha_{1}=\frac{6}{5\left(2-\frac{x-x_{0}}{t\sqrt{gh_{l}}}\right)}-\frac{2}{3}+\frac{4\sqrt{3}}{135}
       \left(2-\frac{x-x_{0}}{t\sqrt{gh_{l}}}\right)^{3/2},
       \end{equation}
       and
       \begin{equation}\label{19bf}
        \alpha_{2}=\frac{12}{2-\frac{x-x_{0}}{t\sqrt{gh_{l}}}}-\frac{8}{3}+\frac{8\sqrt{3}}{189}
       \left(2-\frac{x-x_{0}}{t\sqrt{gh_{l}}}\right)^{3/2}-\frac{108}{7\left(2-\frac{x-x_{0}}{t\sqrt{gh_{l}}}\right)^{2}}.
       \end{equation}
       Following the Dressler's approach we consider four regions: from upstream to downstream (a steady state region
       $(h_{l},10^{-1})$ for $x\leq x_{1}(t)$);
       a corrected region ($(h_{c},u_{c})$ for $x_{1}(t)\leq x\leq x_{2}(t)$); the tip region (for $x_{2}(t)\leq x\leq x_{3}(t)$) and the
       dry region ($(0,0)$ for $x_{3}(t)\leq x\leq L$). In the tip region, friction term is preponderant thus $(\ref{19f})$ is no more valid.
       In the corrected region, the velocity increases with $x.$ Dressler assumed that at $x_{2}(t)$ the velocity reaches the maximum of
       $u_{c}$ and that the velocity is constant in space in the tip region
       \begin{equation}\label{20f}
        u_{tip}(t)=\underset{x\in[x_{2}(t),x_{3}(t)]}{\max}u_{c}(t,x).
       \end{equation}
       Utilizing these assumptions together with relations $(\ref{19f})$-$(\ref{20f})$, the analytic solution of $1$D
       complete shallow water equations with friction terms is then given by
       \begin{equation}\label{21f}
         h(t,x)=\left\{
                          \begin{array}{ll}
                            h_{l}, & \hbox{for $0\leq x\leq x_{1}(t)$ and $t\in(0,T_{1}]$,} \\
                           \text{\,}\\
                           \frac{1}{g}\left(\frac{2}{3}\sqrt{gh_{l}}-\frac{x-x_{0}}{3t}+\frac{g^{2}}{C^{2}}\alpha_{1}t\right)^{2},
                           & \hbox{for $x_{1}(t)\leq x\leq x_{3}(t)$ and $t\in(0,T_{1}]$,} \\
                           \text{\,}\\
                            0, & \hbox{for $x_{3}(t)\leq x\leq L$ and $t\in(0,T_{1}]$,}
                          \end{array}
                       \right.
       \end{equation}
        and
        \begin{equation}\label{22f}
         u(t,x)=\left\{
                          \begin{array}{ll}
                            0, & \hbox{for $0\leq x\leq x_{1}(t)$ and $t\in(0,T_{1}]$,} \\
                           \text{\,}\\
                           \frac{2}{3}\sqrt{gh_{l}}+\frac{2(x-x_{0})}{3t}+\frac{g^{2}}{C^{2}}\alpha_{2}t,
                           & \hbox{for $x_{1}(t)\leq x\leq x_{2}(t)$ and $t\in(0,T_{1}]$,} \\
                           \text{\,}\\
                          \underset{x\in[x_{2}(t),x_{3}(t)]}{\max}u_{c}(t,x),
                           & \hbox{for $x_{2}(t)\leq x\leq x_{3}(t)$ and $t\in(0,T_{1}]$,} \\
                           \text{\,}\\
                            0, & \hbox{for $x_{3}(t)\leq x\leq L$ and $t\in(0,T_{1}]$,}
                          \end{array}
                       \right.
       \end{equation}
         where $\alpha_{1}$ and $\alpha_{2}$ are given by equations $(\ref{19af})$ and $(\ref{19bf}),$ respectively,
       $x_{1}(t)=x_{0}-t\sqrt{gh_{l}},$ $x_{3}(t)=x_{0}+2t\sqrt{gh_{l}}$ and $x_{2}(t)\in[x_{1}(t),x_{3}(t)]$ is the point
       where the velocity $u_{c}(t,x)$ attains its maximum.\\

        With this approach, we should remark that the water height is constant in the tip zone. This is a limit of
       Dressler's approach. Thus the authors \cite{swashes} coded the second order interpolation used in
       \cite{50swashes,51swashes} and recommended by Valerio Caleffi. Even if the authors \cite{swashes} have no information
       concerning the shape of the wave tip, this case shows if the scheme is able to locate and treat correctly the wet/dry
        transition.\\

        Armed with the above tools, we are ready to provide the exact solution of the system of PDEs $(\ref{1}).$ Our analysis
       consider the case where the channel is prismatic with a constant top width $(T)$ and the average velocity $(u)$ is defined
       as $u(t,x)=Q(t,x)/A(t,x).$ Using this, the following formulas hold
       \begin{equation}\label{23f}
        A(t,x)=Th(t,x)\text{\,\,\,\,and\,\,\,\,}Q(t,x)=Th(t,x)u(t,x).
       \end{equation}
        Since the water height is constant in the tip region it comes from relation $(\ref{23f})$ that the cross section $(A)$
        and discharge $(Q)$ are not modified in that region. A combination of relations $(\ref{21f}),$ $(\ref{22f})$ and $(\ref{23f})$
        gives the explicit formulae of cross section and discharge
        \begin{equation}\label{24f}
         A(t,x)=Th(t,x)=\left\{
                          \begin{array}{ll}
                            Th_{l}, & \hbox{for $0\leq x\leq x_{1}(t)$ and $t\in(0,T_{1}]$,} \\
                           \text{\,}\\
                           \frac{T}{g}\left(\frac{2}{3}\sqrt{gh_{l}}-\frac{x-x_{0}}{3t}+\frac{g^{2}}{C^{2}}\alpha_{1}t\right)^{2},
                           & \hbox{for $x_{1}(t)\leq x\leq x_{3}(t)$ and $t\in(0,T_{1}]$,} \\
                           \text{\,}\\
                            0, & \hbox{for $x_{3}(t)\leq x\leq L$ and $t\in(0,T_{1}]$,}
                          \end{array}
                       \right.
       \end{equation}
        and
        \begin{equation}\label{25f}
        Q(t,x)=Th(t,x)u(t,x)=\left\{
                          \begin{array}{ll}
                            0, & \hbox{for $0\leq x\leq x_{1}(t)$ and $t\in(0,T_{1}]$,} \\
                           \text{\,}\\
                           Th_{c}(t,x)u_{c}(t,x), & \hbox{for $x_{1}(t)\leq x\leq x_{2}(t)$ and $t\in(0,T_{1}]$,} \\
                           \text{\,}\\
                          Th_{c}(t,x)u_{tip}(t,x), & \hbox{for $x_{2}(t)\leq x\leq x_{3}(t)$ and $t\in(0,T_{1}]$,} \\
                           \text{\,}\\
                            0, & \hbox{for $x_{3}(t)\leq x\leq L$ and $t\in(0,T_{1}]$.}
                          \end{array}
                       \right.
           \end{equation}
            The following values are considered for simulations: shear stress $\overline{\tau}=1.329N/m^{2};$ Top width $T=348m;$
            wetter perimeter $P=366,4m;$ wavelength $K_{\lambda}=2\pi\simeq6.28m;$ manning's number $n_{1}=0.025s/m^{1/3};$ the
            acceleration of gravity $g=10m/s^{2};$ the rainfall intensity is described as
            \begin{equation}\label{26f}
             I(t,x)=\left\{
                      \begin{array}{ll}
                        1.18\times10^{-5}m/s, & \hbox{if $(t,x)\in[0,T_{1}]\times[0,L]$;} \\
                        0, & \hbox{otherwise.}
                      \end{array}
                    \right.
            \end{equation}
           The mathematical model for this ideal overland flow is the following: we consider a uniform plane catchment whose overall
          length in the direction of flow is $L=1m.$ The surface roughness and shear stress are assumed invariant in space
          and time. It comes from equation $(\ref{26f})$ the constant rainfall excess is defined as
           \begin{equation}\label{26ff}
             r(t,x)=\left\{
                      \begin{array}{ll}
                        I(t,x), & \hbox{for $(t,x)\in[t_{0},T_{1}]\times[0,L]$;} \\
                        0, & \hbox{otherwise.}
                      \end{array}
                    \right.
            \end{equation}
      The mesh size $\Delta x$ takes values: $2^{-4},$ $2^{-5},$ $2^{-6}$ and $2^{-7},$ while the time step $\Delta t$
     varies in range: $2^{-7},$ $2^{-8},$ $2^{-9}$ and $2^{-10}.$ $I$ is the rainfall intensity defined
     by relation $(\ref{26f}),$ $t_{0}=0s$ and $T_{1}=1s$ are initial and final time, respectively, of the rainfall excess
     computed above, and $L=1m$ is the rod interval length. The approximate solutions given by numerical schemes $(\ref{64})$-$(\ref{67})$
     obtained from $1$ to $20$ iterations, respectively, are displayed in Figures $\ref{figure 2}$ and $\ref{figure 3}.$
     Different values of $k=\Delta t=2^{-7},2^{-8},2^{-9},2^{-10},$ numbers obtained from the stability restriction $(\ref{101})$ as the
     steady flow cases and space step $h=\Delta x=2^{-4},2^{-5},2^{-6},2^{-7},$ in the mesh are used. Before $3$ iterations
     are encountered, the discharge wave propagates with almost a perfectly constant value at different positions (see Figures $\ref{figure 2}$
     and $\ref{figure 3}$). Furthermore, after $3$ iterations, the discharge wave also tends to zero at different times
     (Figures $\ref{figure 2}$ and $\ref{figure 3}$). So, the graphs show that the solution of the difference equations cannot grow
     with time and so must still satisfy the Von Neumann necessary condition. We obtain similar observations for the cross section.
     In addition, \textbf{Table 1} suggests that the errors associated with the cross section is of second order accurate while the approximate
     solution corresponding to the discharge coincides with the exact one. This suggests that for the considered analytical
     solution the MacCormack approach $(\ref{64})$-$(\ref{67})$ converges with second order accuracy. Furthermore, the graphs
     indicate that the numerical solutions start to destroy after a fixed time. Specifically, combining the different values of $\Delta x$
     and $\Delta t,$ we observe from the graphs that that good approximate solutions are obtained for small mesh sizes $\Delta x$ and time steps
     $\Delta t$ satisfying the stability restriction $(\ref{101})$. Thus, physical insight must be used when the stability limitation
     $(\ref{101})$ of the MacCormack method is investigated. Finally, both \textbf{Table 1} and Figures $\ref{figure 2}$ and $\ref{figure 3}$
     show that the numerical solutions do not increase with time and converge to the analytical one. More specifically, they indicate that
     stability for the MacCormack scheme is subtle. It is not unconditionally unstable, but stability depends on the parameters $\Delta x$
     and $\Delta t.$
       \text{\,\,}\\
        \text{\,\,}\\
       \text{\,\,}\\

               \textbf{Table 1.}$\label{t1}$ Analyzing of convergence rate $O(h^{\theta}+\Delta t^{\beta})$ for MacCormack scheme by $r_{(\cdot)},$
           with varying spacing $h=\Delta x$ and time step $k=\Delta t$.
            \begin{equation*}
            \begin{tabular}{|c|c|c|c|c|}
            \hline
            $(k,h)$ & $\|A-A_{1}\|_{L^{2}}$&$\|Q-Q_{1}\|_{L^{2}}$ & $r_{(A)}$ & $r_{(Q)}$ \\
            \hline
            $(2^{-7},2^{-4})$ & 0.0384 & 0  &          &   --      \\
            \hline
            $(2^{-8},2^{-5})$ & 0.0192 & 0  & 2.0000   &   --       \\
            \hline
            $(2^{-9},2^{-6})$ & 0.0093 & 0  & 2.0645   &    --      \\
            \hline
            $(2^{-10},2^{-7})$& 0.0047 & 0  & 1.9787   & -- \\
            \hline
          \end{tabular}
            \end{equation*}
           \text{\,}\\

       \section{General conclusions and future works}\label{cfw}
        In this paper we have presented a full description of the MacCormack approach for complete shallow water equations with
        source terms, have studied in details the stability analysis of the numerical scheme and we have provided the convergence
        rate of the method (which is computed numerically). The graphs (Figures $\ref{figure 2}$ and $\ref{figure 3}$) show that
        the considered method is both stable and convergent while \textbf{Table 1} indicates the rate of convergence (second order)
        of the algorithm. After a few number of iterations, the figures suggest that the numerical solutions strongly converge to the
        analytical one. This is not a surprise since the exact solutions are discontinuity and tend to zero whenever the time $t$ is
        different from zero. From Figures $\ref{figure 2}$ and $\ref{figure 3}$, one should observe that the only case where the exact
        solutions are not close to zero corresponds to the initial condition. This comes from assumptions made by Dressler when he constructs
        the analytical solutions. Our future investigations will consist to find a good analytical solution for the proposed problem and to extend
        the analysis to open channel flows.\\

     \textbf{Acknowledgment.} The authors thank the anonymous referees for detailed and valuable comments which have helped to greatly improve
      the quality of this paper.

          \begin{figure}
         \begin{center}
          Graphs of cross section and discharge for shallow water flow.
         \begin{tabular}{c c}
         \psfig{file=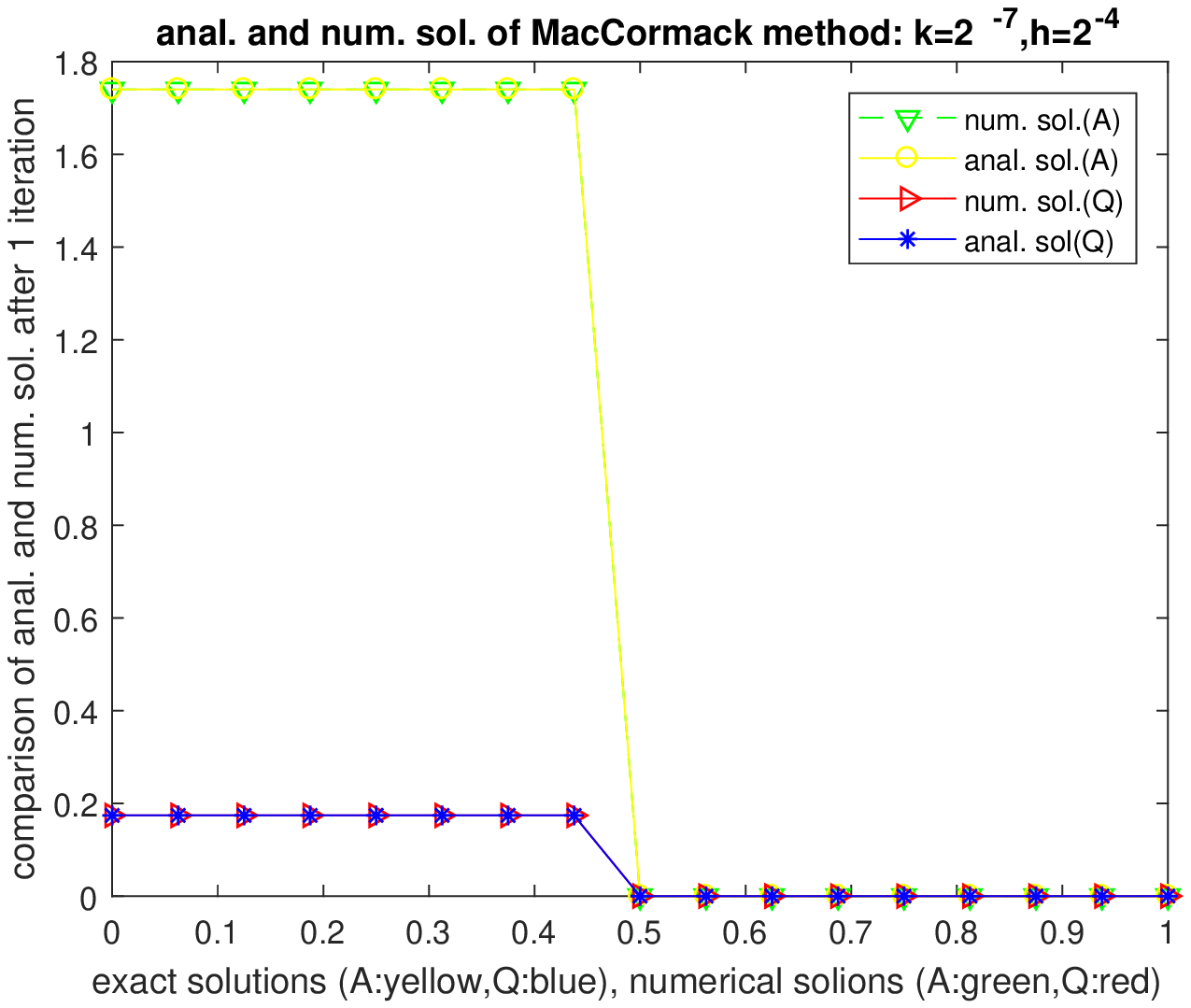,width=7cm} & \psfig{file=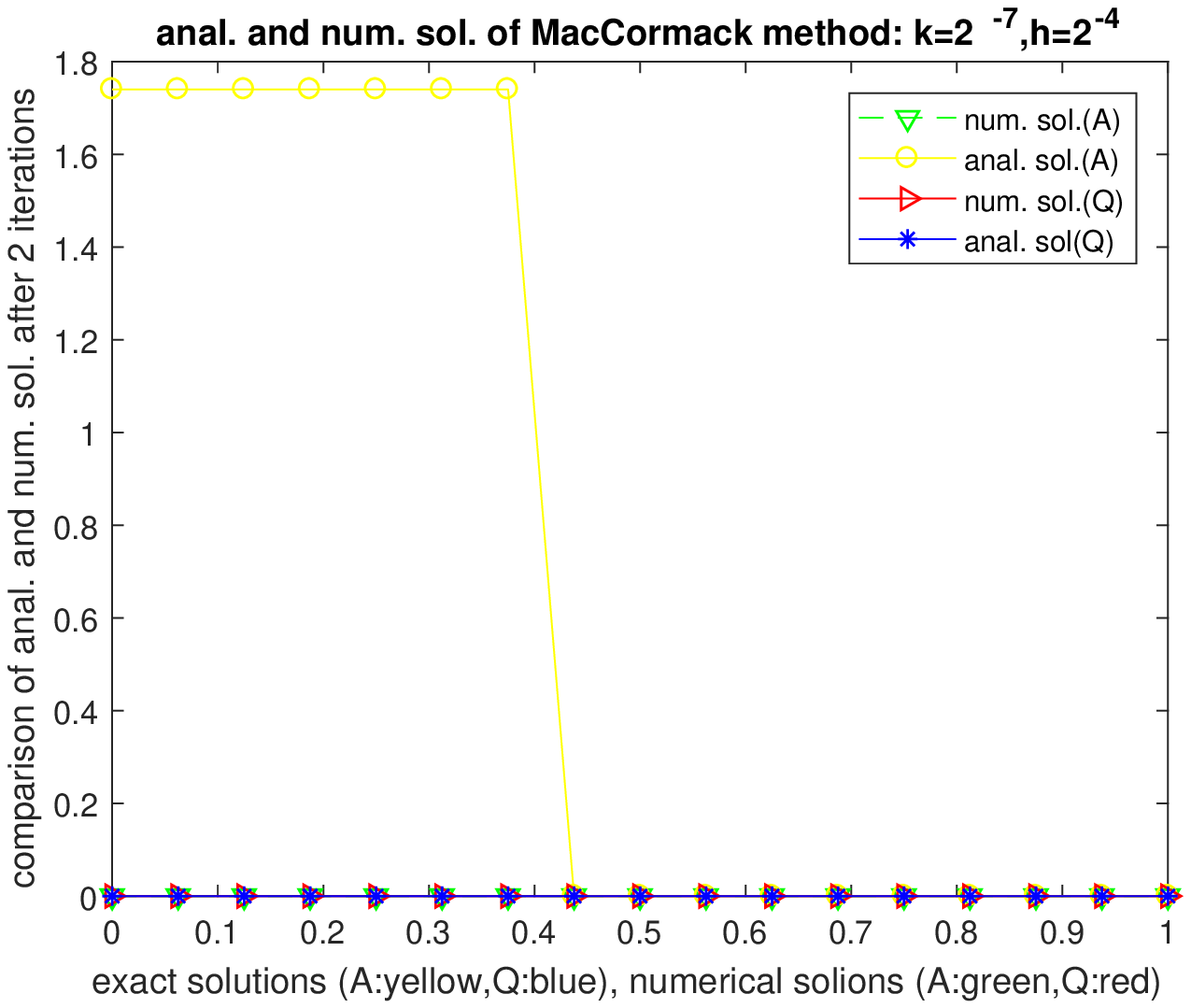,width=7cm}\\
         \psfig{file=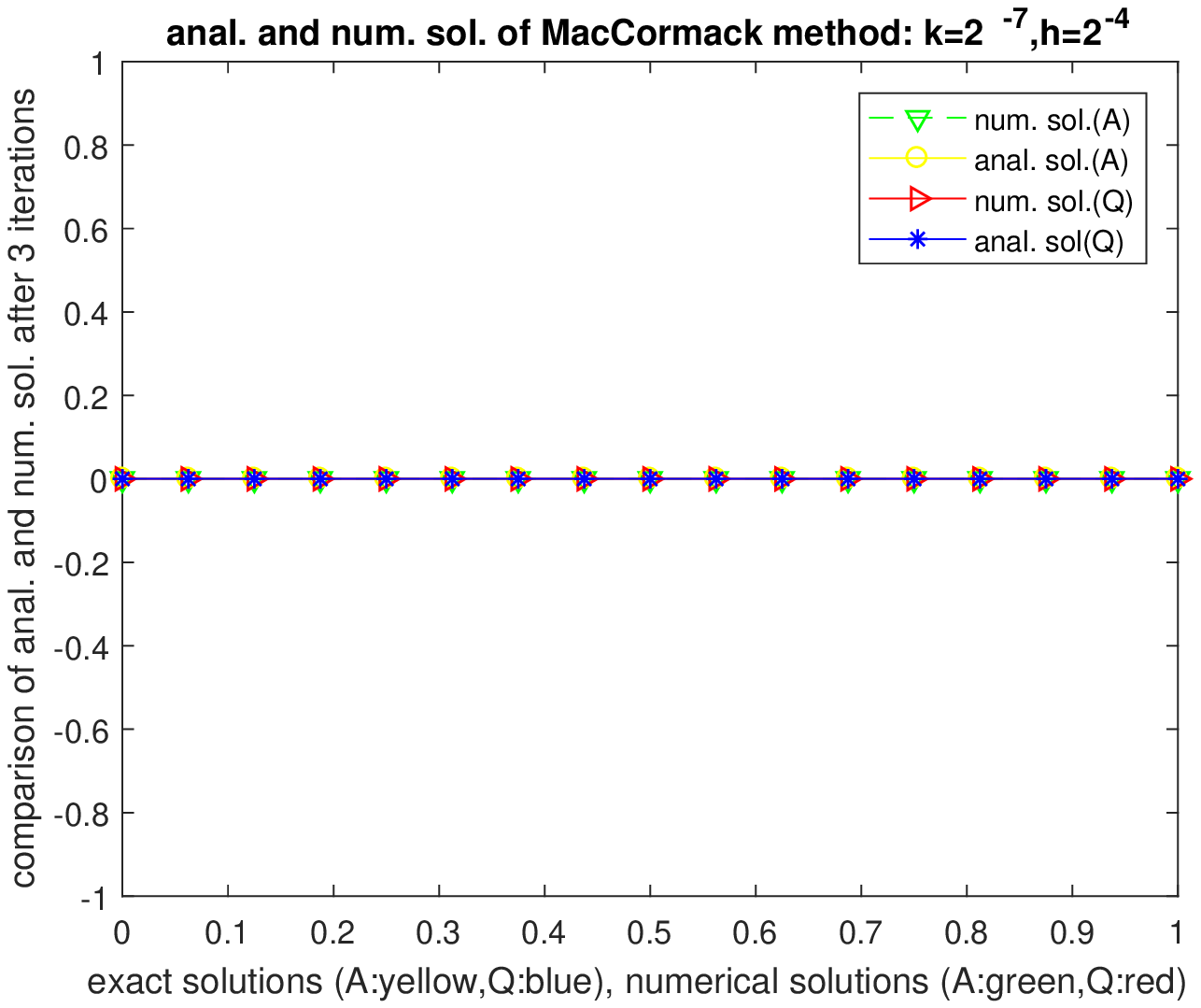,width=7cm} & \psfig{file=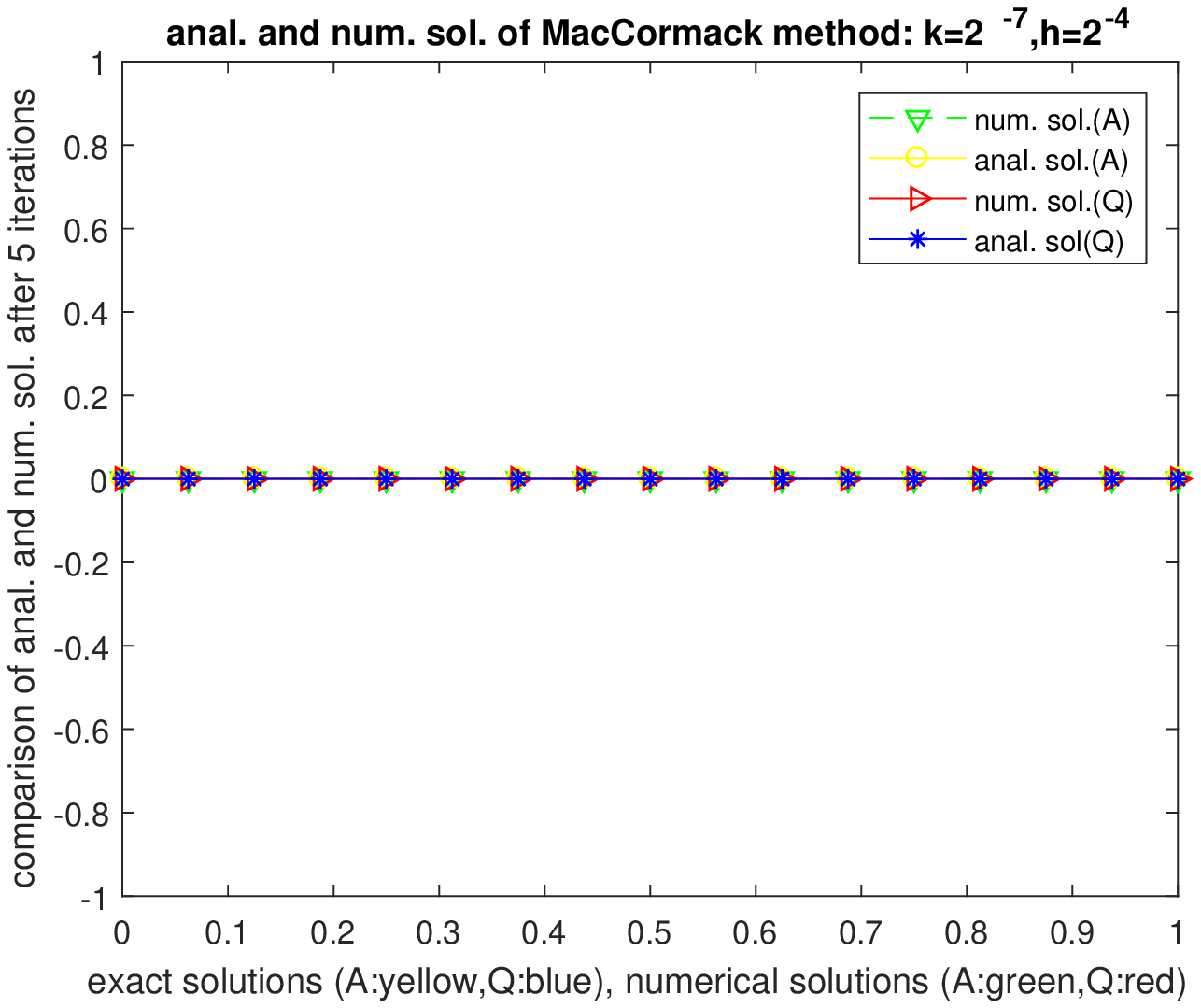,width=7cm}\\
         \psfig{file=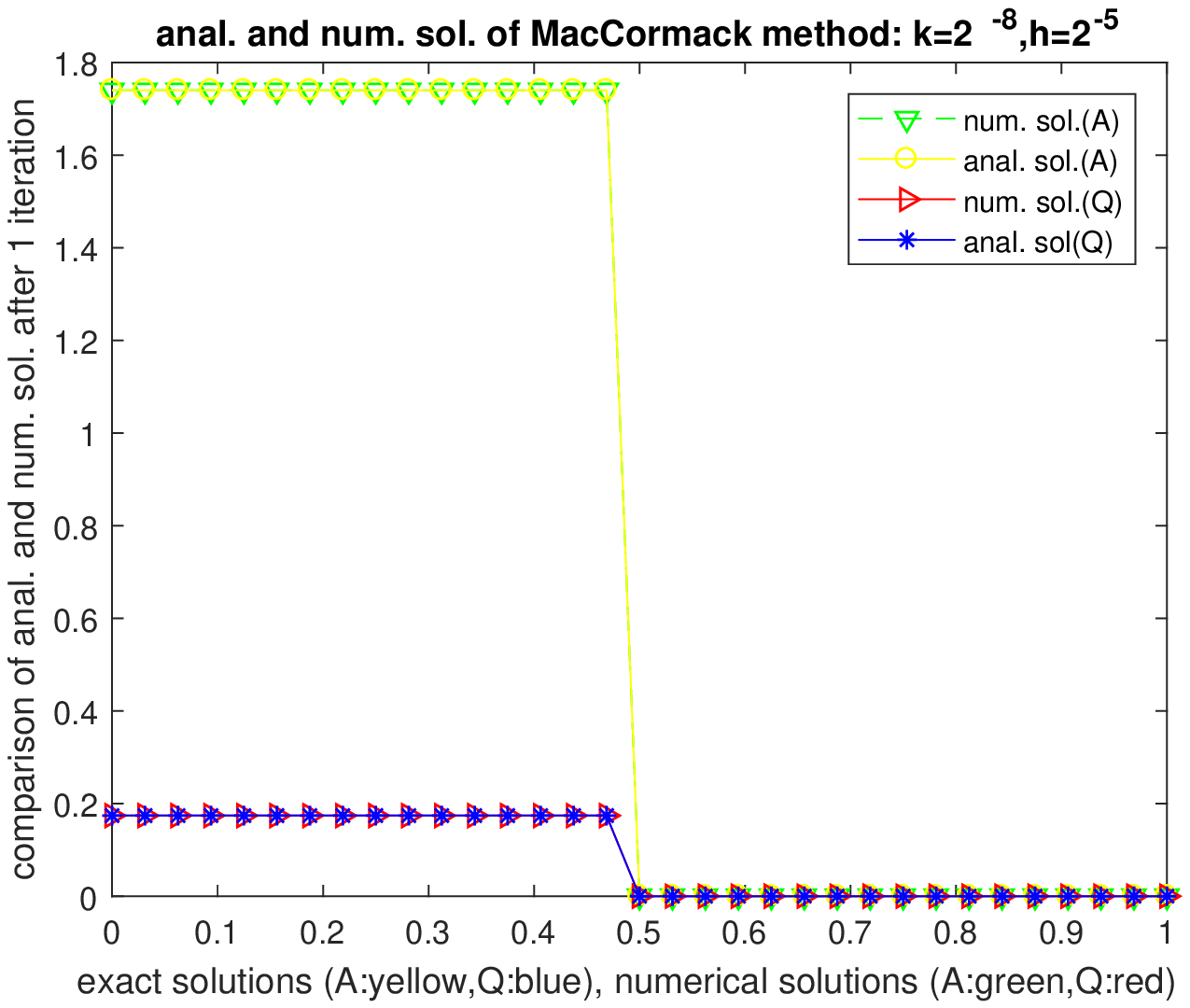,width=7cm} & \psfig{file=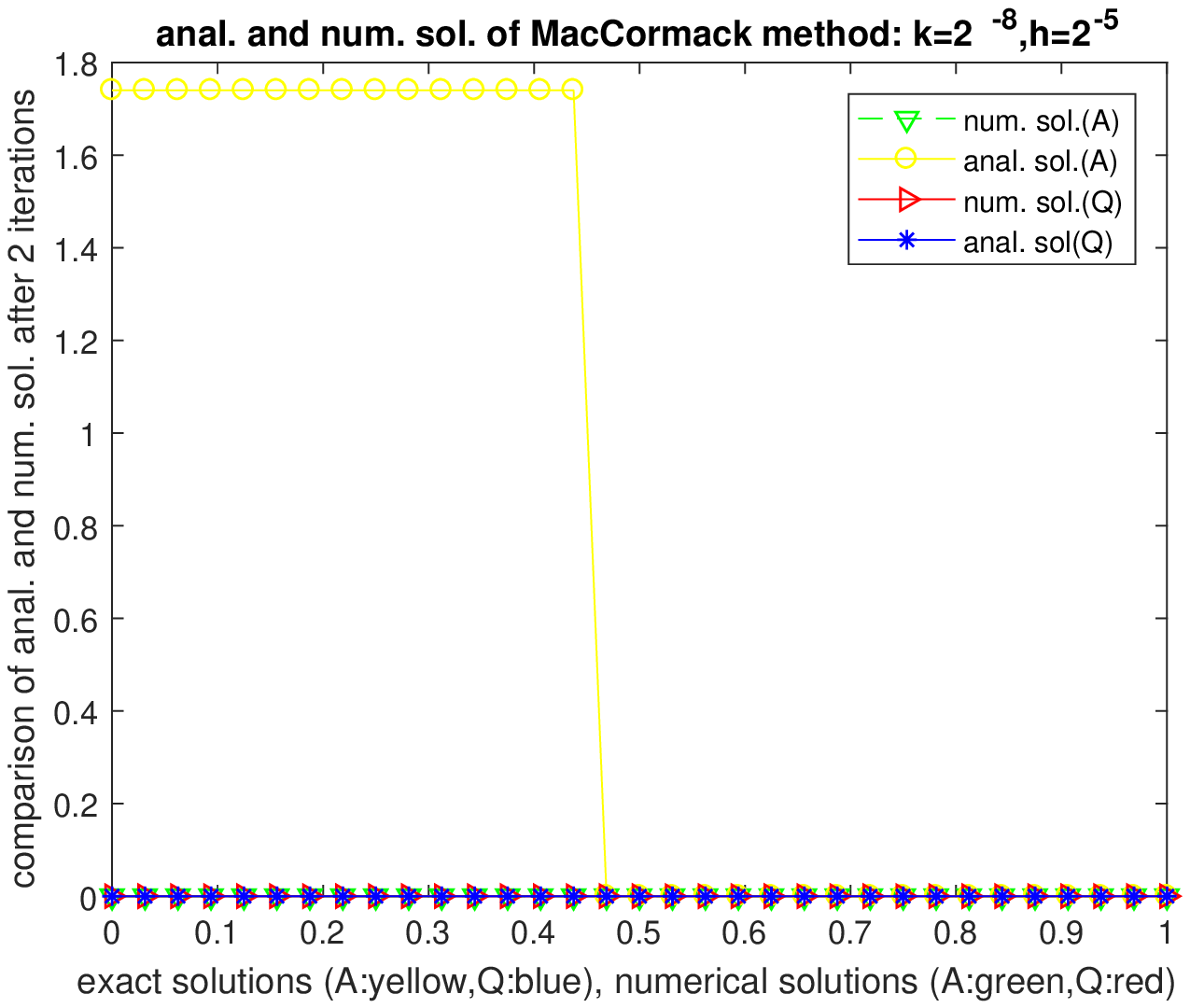,width=7cm}\\
         \psfig{file=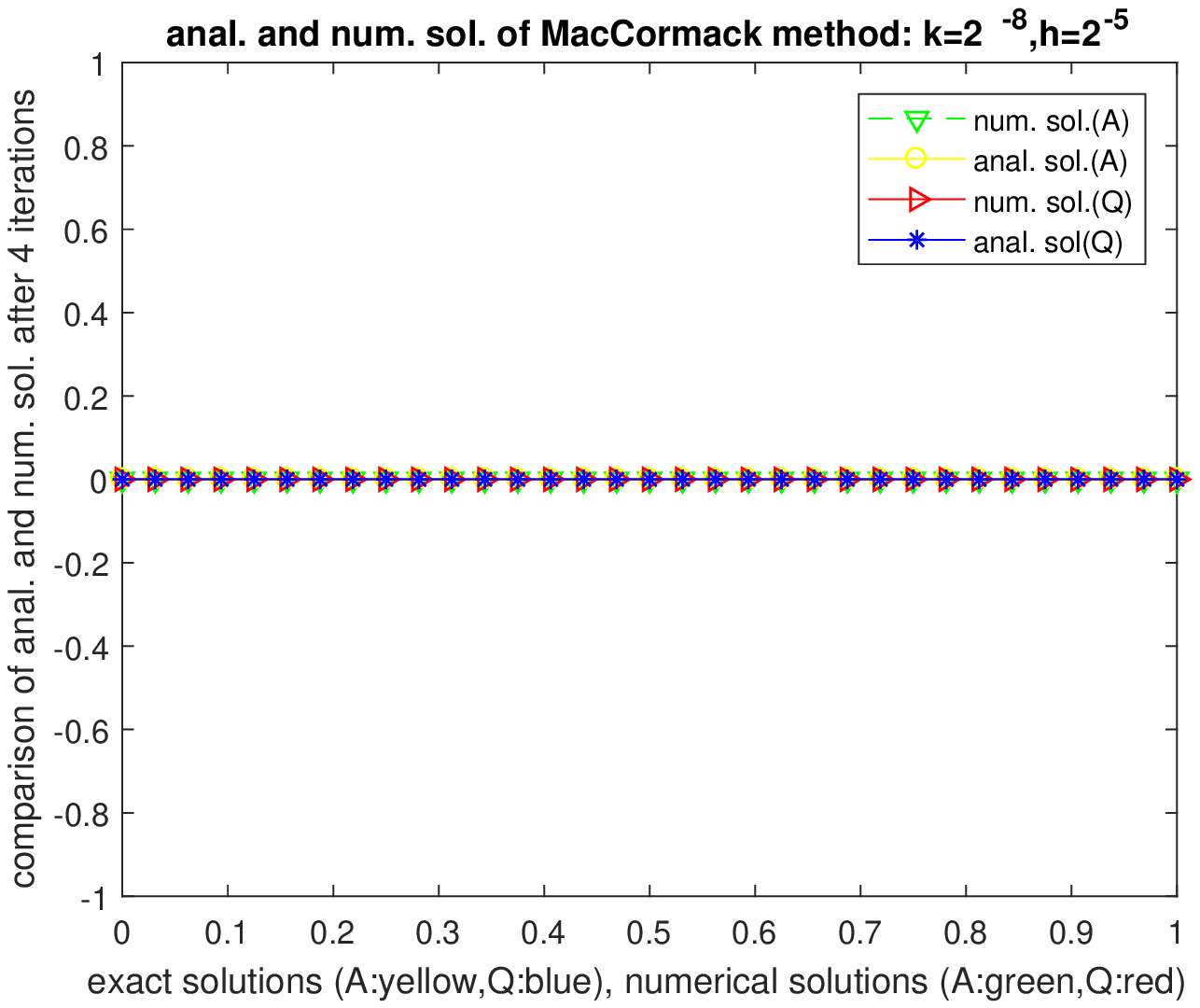,width=7cm} & \psfig{file=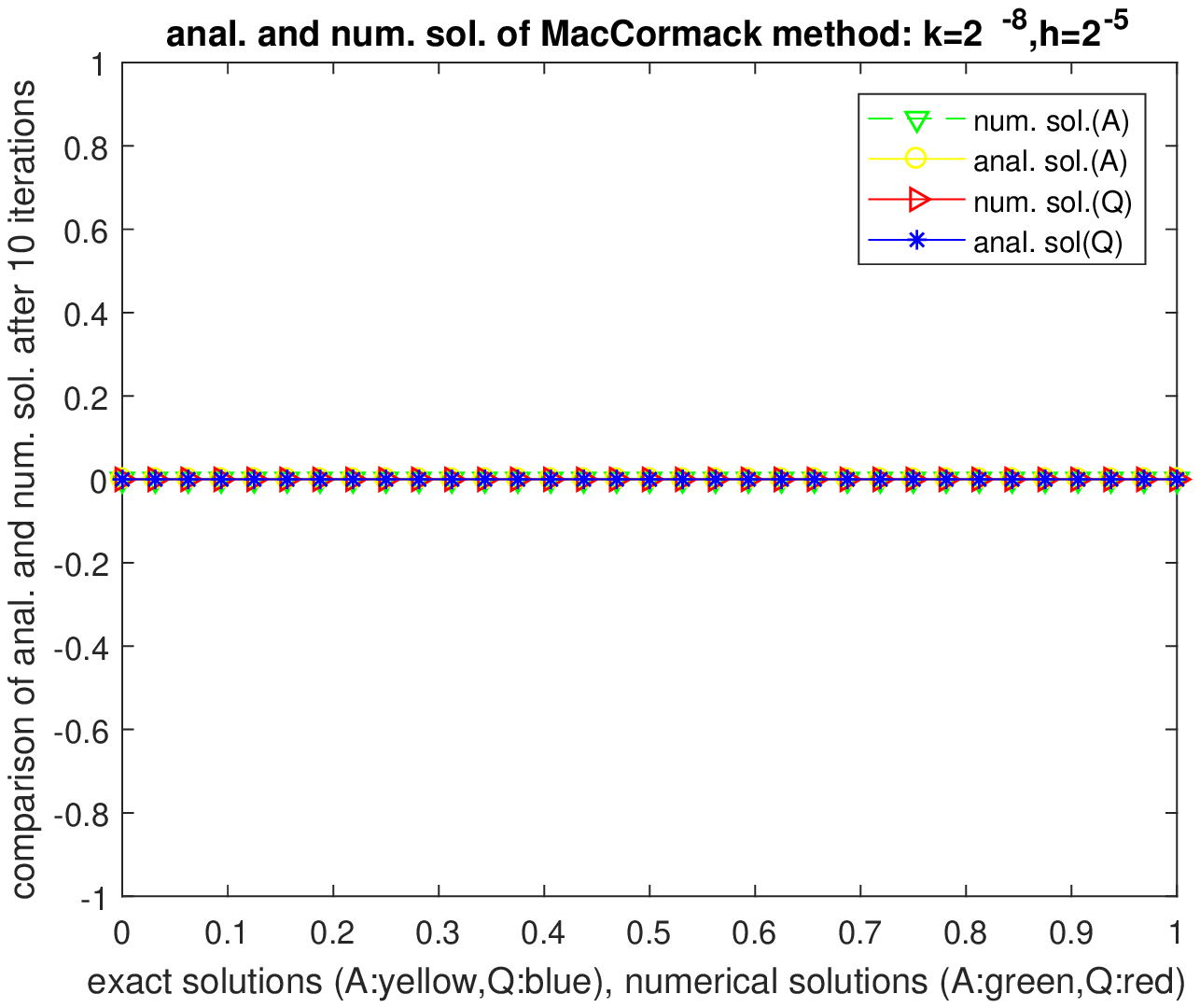,width=7cm}\\
           $ 0\leq x(m)\leq 1$ & $0\leq x(m)\leq 1$
         \end{tabular}
        \end{center}
        \caption{Stability analysis and convergence rate of MacCormack for shallow water equations with source terms.}
        \label{figure 2}
        \end{figure}

      \begin{figure}
     \begin{center}
      Graphs of cross section and discharge for shallow water flow.
      \begin{tabular}{c c}
         \psfig{file=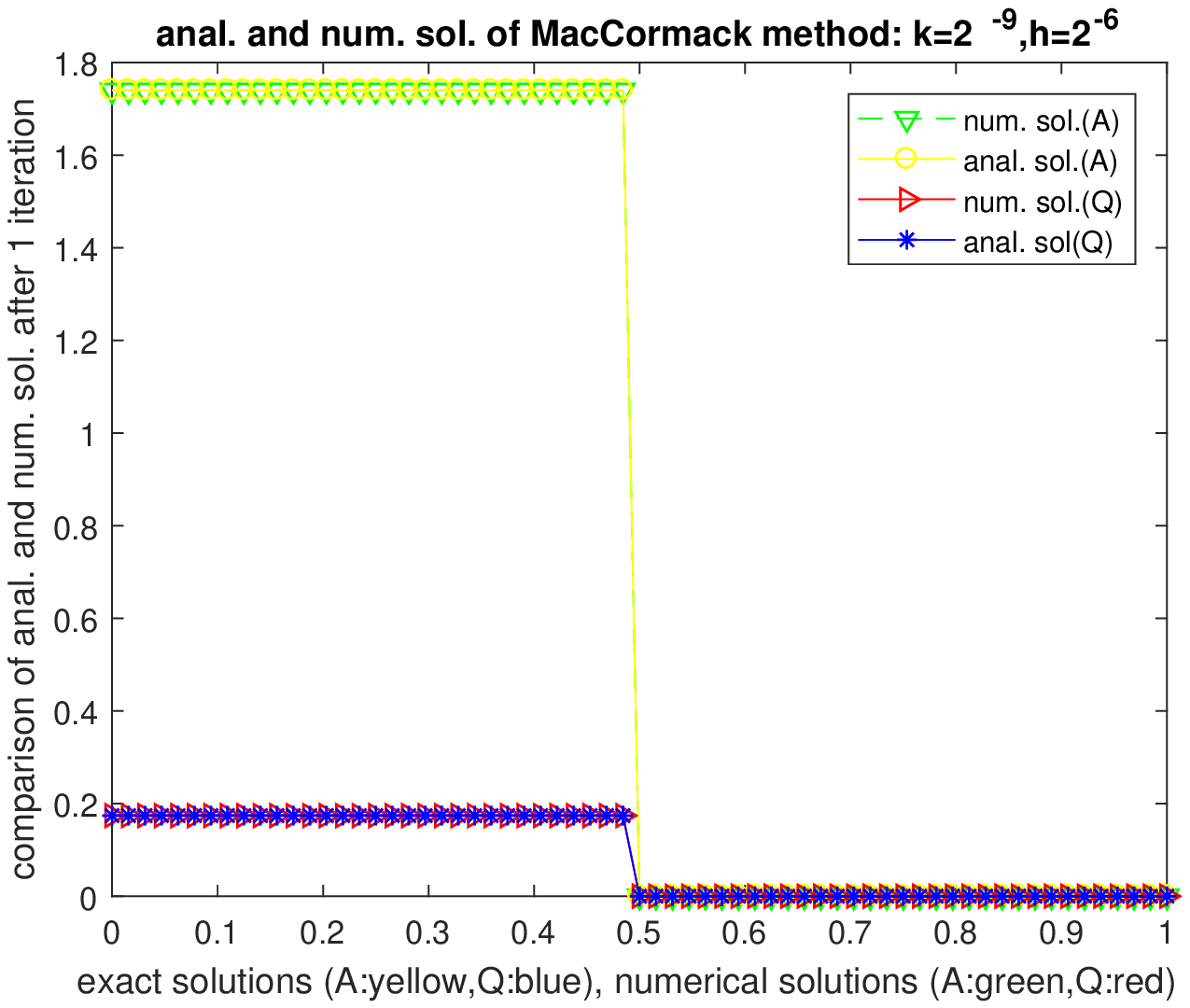,width=7cm}  & \psfig{file=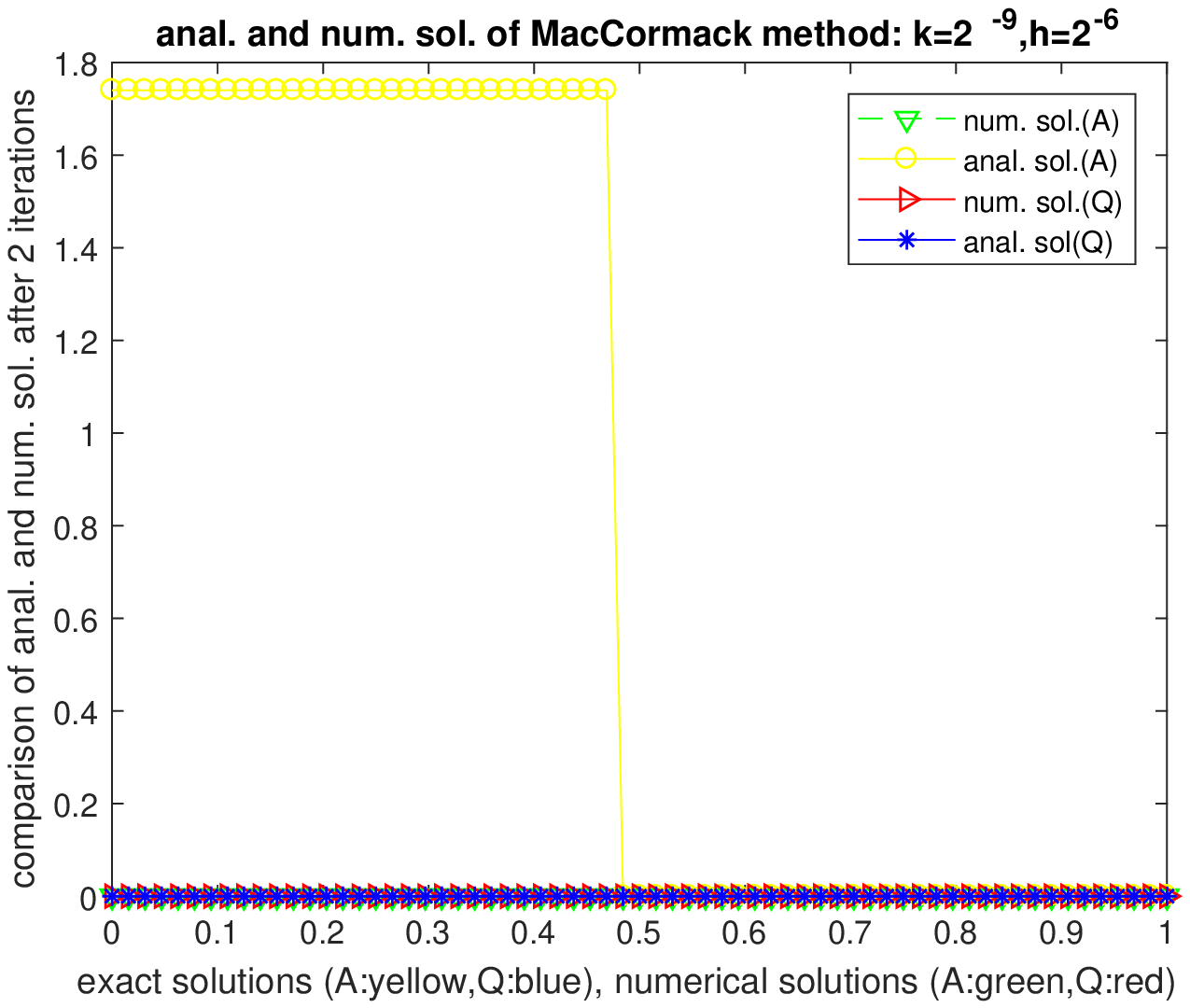,width=7cm}\\
         \psfig{file=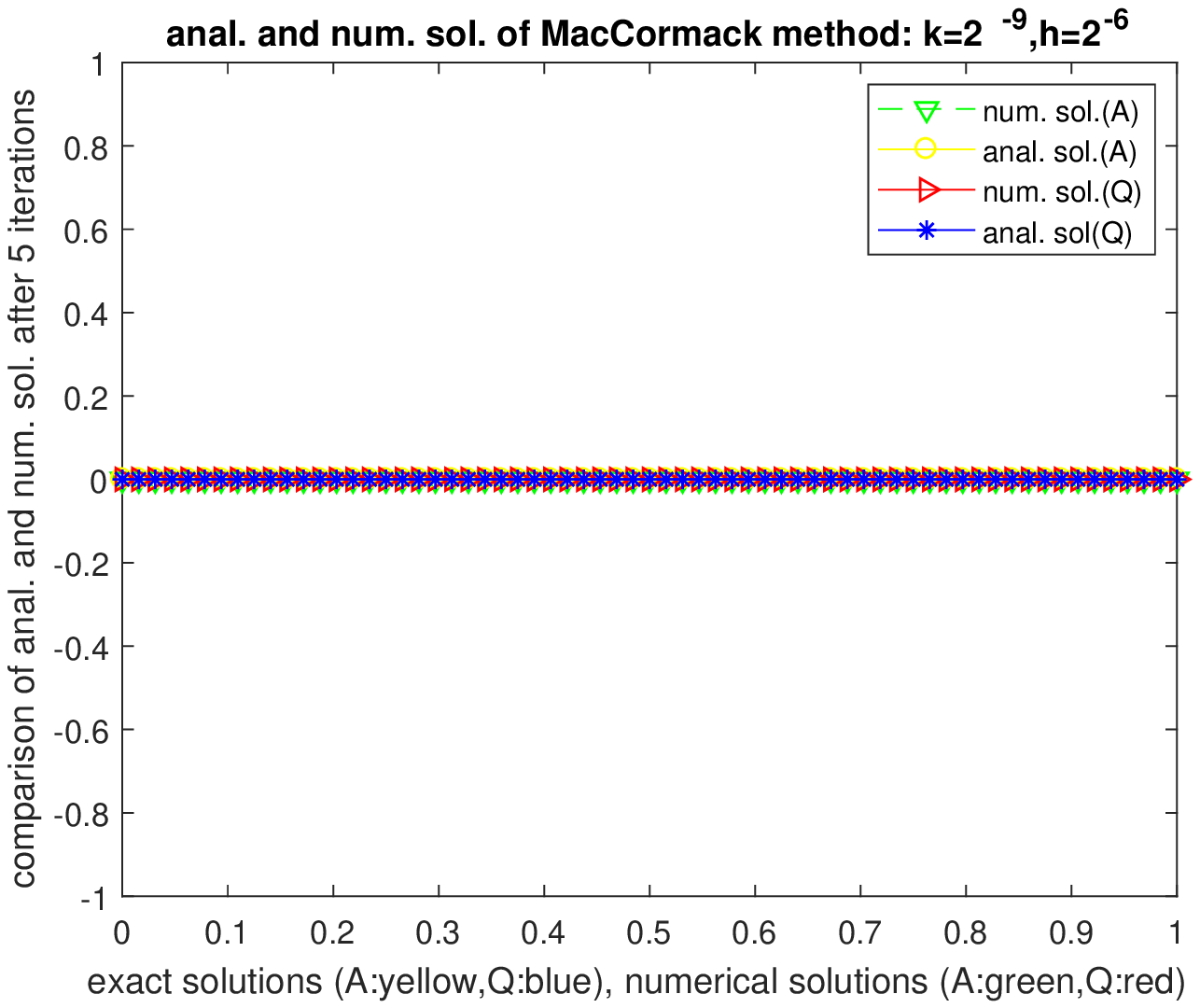,width=7cm} & \psfig{file=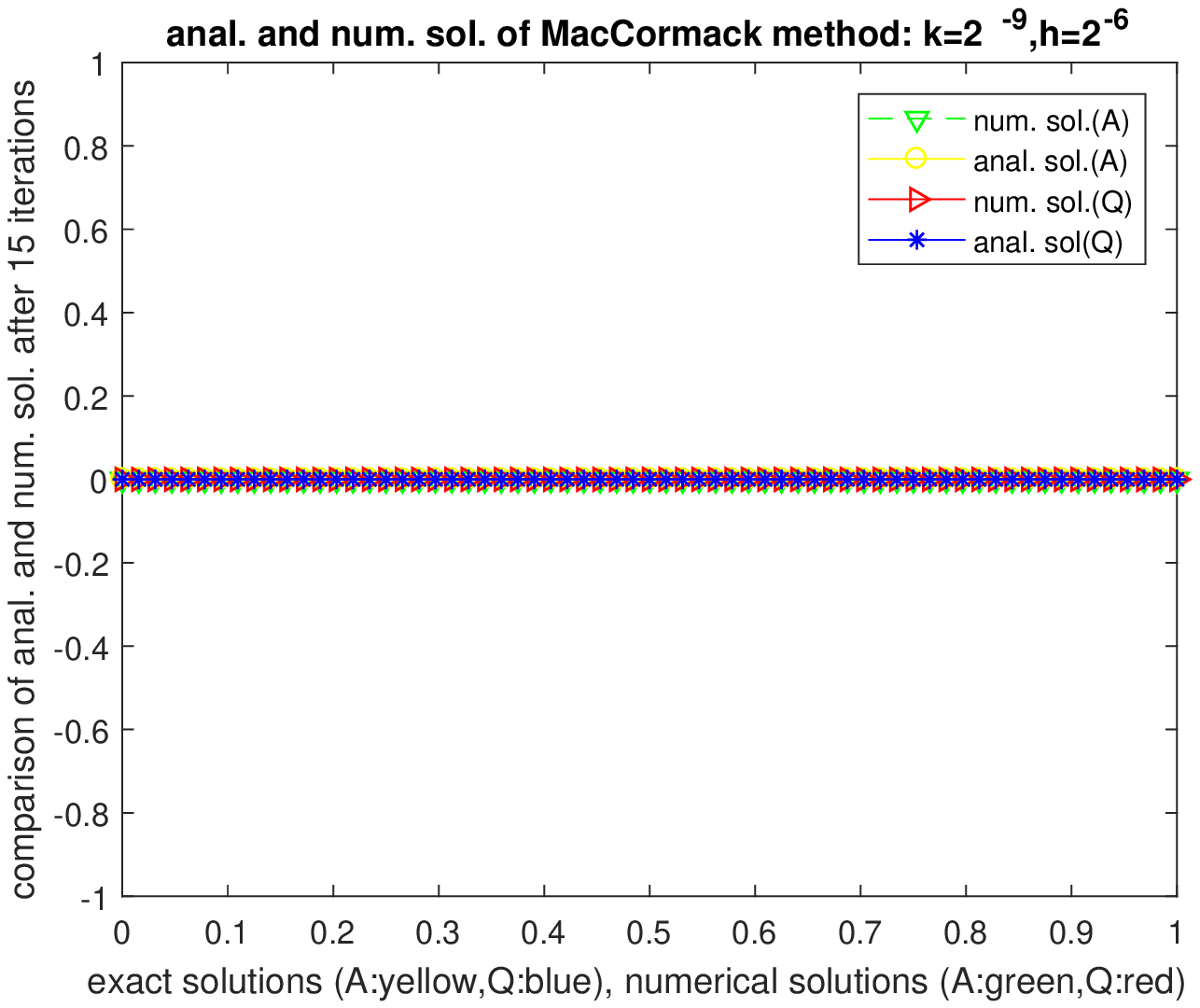,width=7cm}\\
         \psfig{file=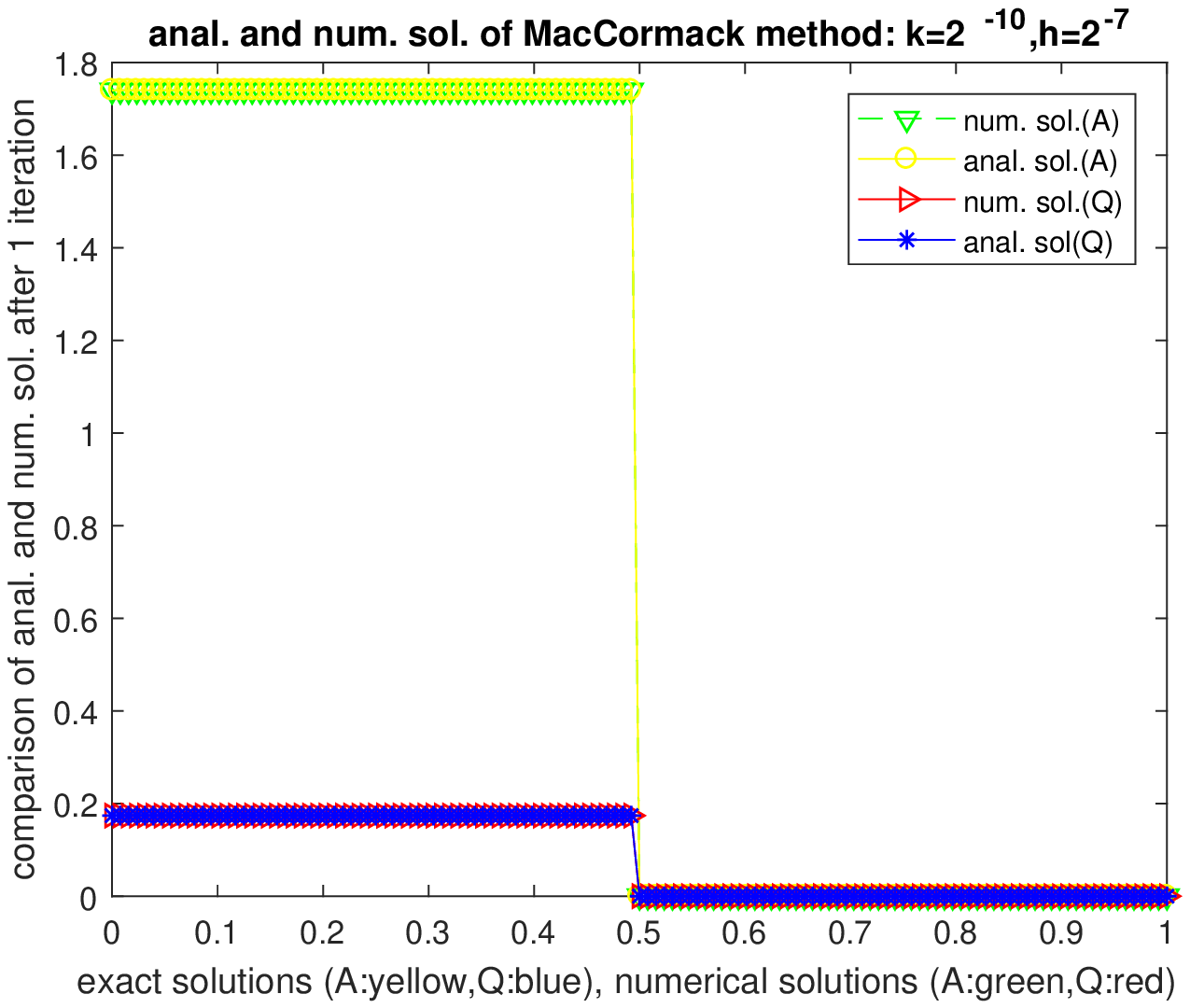,width=7cm} & \psfig{file=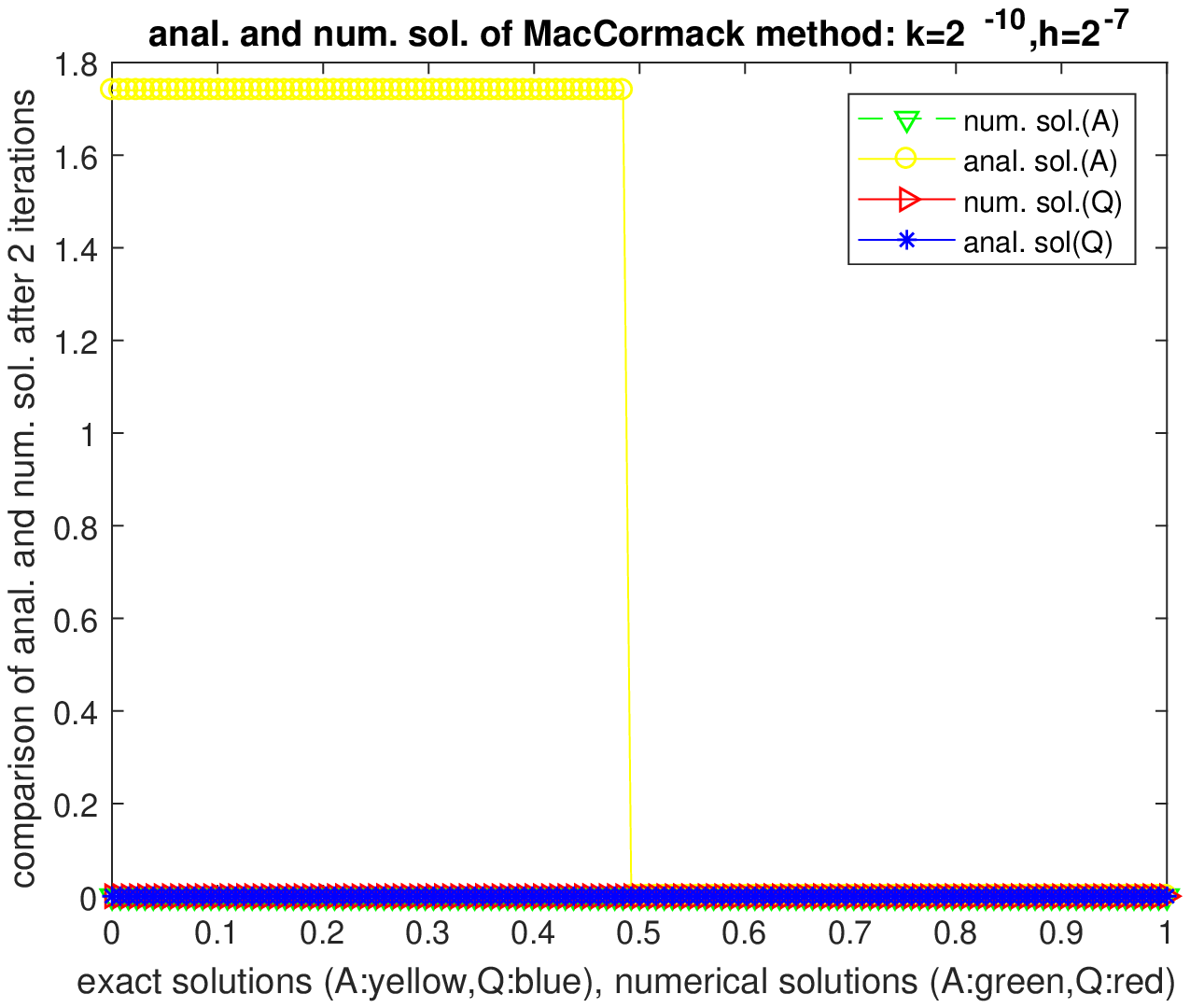,width=7cm}\\
         \psfig{file=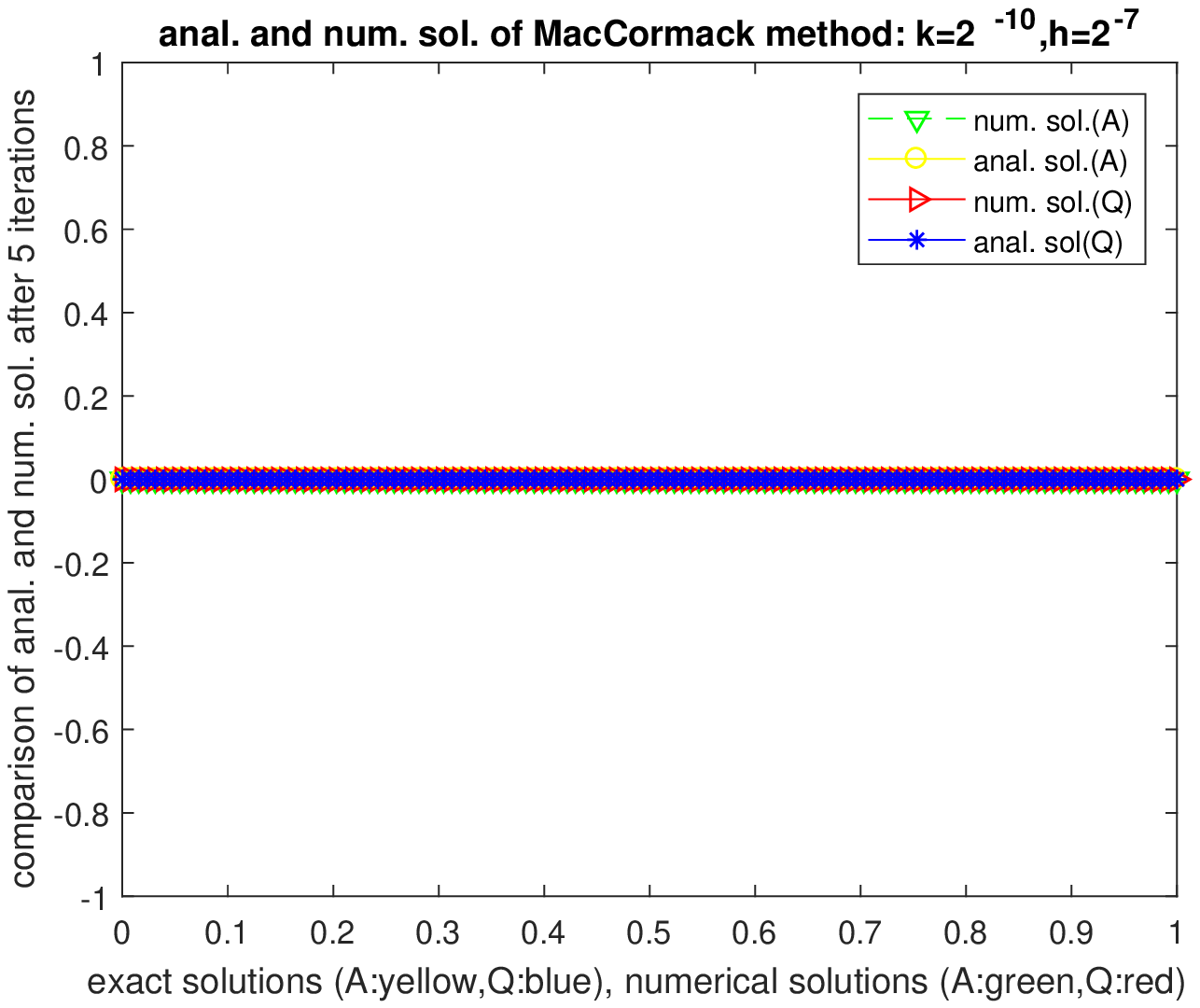,width=7cm} & \psfig{file=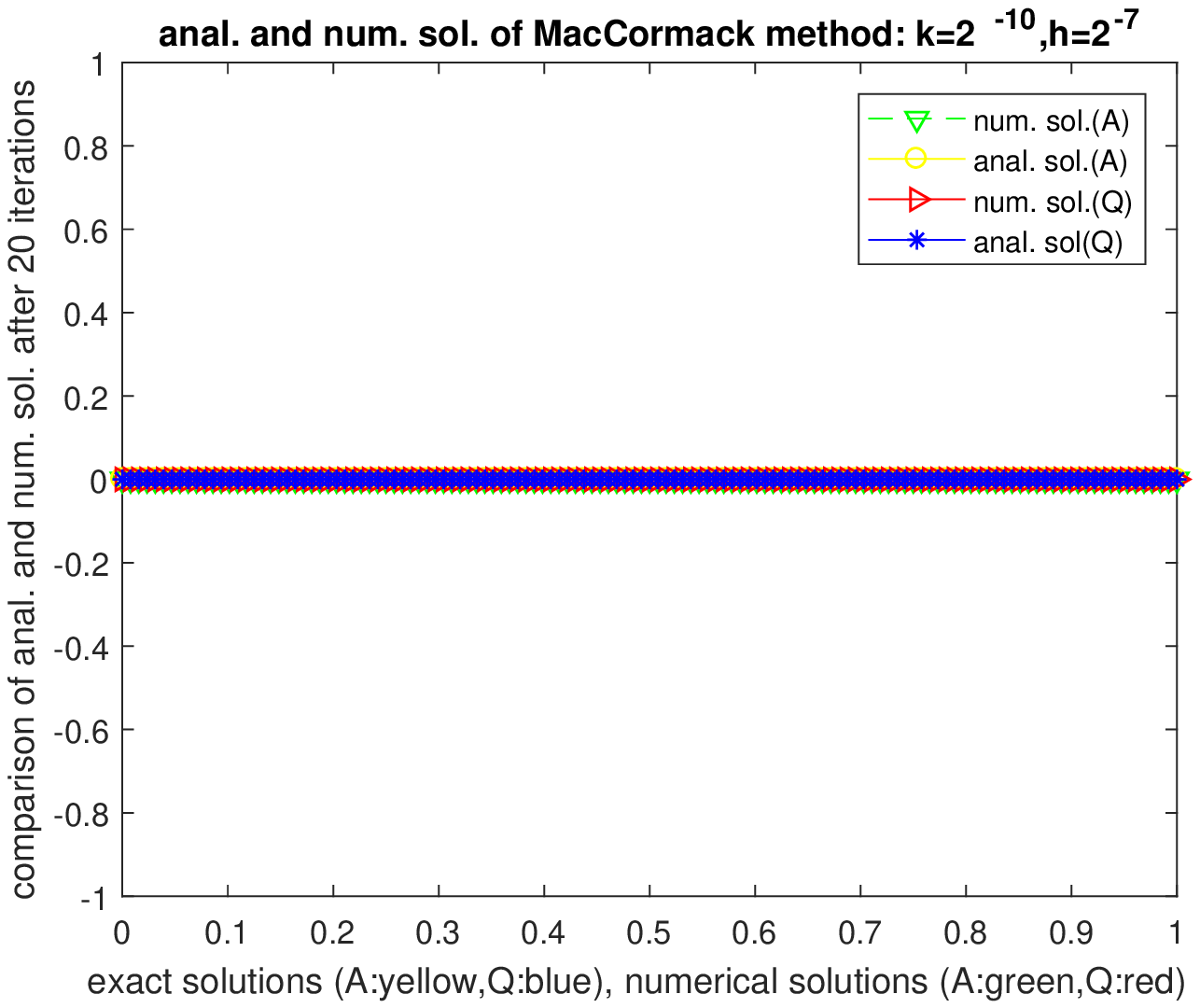,width=7cm}\\
         $ 0\leq x(m)\leq 1$ & $0\leq x(m)\leq 1$
         \end{tabular}
        \end{center}
         \caption{Stability analysis and convergence rate of MacCormack for shallow water equations with source terms.}
          \label{figure 3}
          \end{figure}

     \end{document}